\documentclass[11pt,a4paper,twoside]{amsart}
\usepackage[top=1in, bottom=1.25in, marginparwidth=1in, inner=1in, outer=1in]{geometry}

\usepackage{lmodern}
\usepackage[T2A,T1]{fontenc} 
\usepackage[utf8]{inputenc}
\usepackage[USenglish]{babel}
\usepackage{yfonts}
\usepackage{booktabs}
\usepackage{afterpage}
\usepackage{pinlabel}

\usepackage{graphicx}
\usepackage{tikz}
\usepackage{tikz-cd}
\usetikzlibrary{cd,decorations.pathreplacing}
\usetikzlibrary{shapes,decorations}
\tikzcdset{arrow style=tikz, diagrams={>=latex}}
\tikzset{>=latex}
\tikzstyle{tight}=[font=\scriptsize, inner sep=1pt, outer sep=1pt]

\usepackage{float}
\usepackage[labelfont=bf]{caption}
\DeclareCaptionFormat{myformat}{\fontsize{9}{10}\selectfont#1#2#3}
\usepackage[labelfont=up,margin=5pt]{subcaption}
\usepackage{setspace} 
\usepackage[hidelinks]{hyperref}
\hypersetup{%
	pdfencoding=auto, %
	bookmarksdepth=3%
}
\usepackage{bookmark}

\usepackage{amsmath}
\usepackage{nicefrac}
\usepackage[makeroom]{cancel}
\usepackage{amsfonts}
\usepackage{stmaryrd}
\usepackage{amssymb}
\usepackage{amsthm}
\usepackage{mathtools}
\usepackage{mathrsfs}
\usepackage{bm}
\mathtoolsset{mathic=true}
\def\co{\colon\thinspace\relax}

\newtheorem{theorem}{Theorem}[section]
\newtheorem*{theorem*}{Theorem}
\newtheorem{lemma}[theorem]{Lemma}
\newtheorem{deflemma}[theorem]{Definition/Lemma}
\newtheorem{conjecture}[theorem]{Conjecture}
\newtheorem{question}[theorem]{Question}
\newtheorem{problem}[theorem]{Problem}

\newtheorem{corollary}[theorem]{Corollary}
\newtheorem{proposition}[theorem]{Proposition}
\theoremstyle{definition}
\newtheorem{definition}[theorem]{Definition}

\newtheorem{example}[theorem]{Example}

\newtheorem{observation}[theorem]{Observation}
\newtheorem{remark}[theorem]{Remark}

\makeatletter
\newcommand*{\math@version@bold}{bold}
\DeclareMathOperator\DD{
	\textrm{%
		\usefont{T2A}{cmr}{\ifx\math@version\math@version@bold bx\else m\fi}{n}%
		\CYRD
	}%
} 
\makeatother

\usepackage{enumitem}
\newcounter{dummy}
\makeatletter
\newcommand\myitem[1][]{%
	\item[\textnormal{(}#1\textnormal{)}]\refstepcounter{dummy}%
	\def\@currentlabel{\textnormal{(}#1\textnormal{)}}%
}
\makeatother

\renewcommand{\geq}{\geqslant}
\renewcommand{\leq}{\leqslant}

\newcommand{\fieldTwoElements}{\mathbb{F}}
\newcommand{\F}{{\mathbb{F}}}

\newcommand{\field}{{\mathbf{k}}}
\newcommand{\Z}{\mathbb{Z}}
\newcommand{\ZZ}{\mathbb{Z}/2}
\newcommand{\Q}{\mathbb{Q}} 
\newcommand{\R}{\mathbb{R}}

\newcommand{\BNAlgH}{\mathcal{B}}

\newcommand{\Ad}{\operatorname{\mathcal{A}}^\partial}

\DeclareMathOperator{\Cobl}{\Cob_{/{\mathit{l}}}}

\DeclareMathOperator{\id}{id}

\DeclareMathOperator{\Mod}{Mod}

\DeclareMathOperator{\End}{End}

\DeclareMathOperator{\Cob}{Cob}

\DeclareMathOperator{\gr}{g}
\DeclareMathOperator{\Gen}{\mathcal{G}}
\newcommand{\delV}{\delta_{\!\rotatebox[origin=c]{90}{$-$}\!}}
\newcommand{\delH}{\delta_{-}}

\DeclareMathOperator{\Spinc}{Spin^\textit{c}}

\newcommand{\Diag}{\mathcal{D}} 
\DeclareMathOperator{\HFK}{\widehat{HFK}}
\DeclareMathOperator{\HFL}{\widehat{HFL}}

\DeclareMathOperator{\HFhat}{\widehat{HF}}
\DeclareMathOperator{\CFhat}{\widehat{CF}}
\newcommand{\HF}{\operatorname{HF}}

\DeclareMathOperator{\Homology}{\mathbf{H}_\ast}

\DeclareMathOperator{\HFT}{HFT}

\DeclareMathOperator{\CFTd}{CFT^\partial}

\newcommand{\KhTl}[1]{{\llbracket #1 \rrbracket}_{/l}} 
\DeclareMathOperator{\Kh}{Kh}
\DeclareMathOperator{\CKh}{CKh}
\DeclareMathOperator{\CBN}{CBN}
\DeclareMathOperator{\BN}{BN}
\newcommand{\Khr}{\widetilde{\Kh}}
\newcommand{\CKhr}{\widetilde{\CKh}}
\newcommand{\CBNr}{\widetilde{\CBN}}
\newcommand{\BNr}{\widetilde{\BN}}

\newcommand{\s}{\mathbf{s}}
\newcommand{\ts}{\tilde{\s}}
\renewcommand{\r}{\mathbf{r}}
\newcommand{\tr}{\tilde{\r}}

\newcommand{\Rational}{\mathbf{r}}
\newcommand{\Special}{\mathbf{s}}

\newcommand{\rKh}{\mathbf{r}}
\newcommand{\sKh}{\mathbf{s}}
\DeclareMathOperator{\Slopes}{\mathcal{S}}
\DeclareMathOperator{\Curves}{\mathfrak{C}}
\DeclareMathOperator{\CurvesG}{\mathfrak{C}_\mathit{G}}
\DeclareMathOperator{\CurvesZ}{\mathfrak{C}_{\Z}}
\DeclareMathOperator{\CurvesZZ}{\mathfrak{C}_{\ZZ}}
\DeclareMathOperator{\CurvesHF}{\mathfrak{C}_{\HF}}
\DeclareMathOperator{\CurvesHFwb}{\mathfrak{C}^\mathrm{wb}_{\HF}}
\DeclareMathOperator{\CurvesKh}{\mathfrak{C}_{\Kh}}

\newcommand{\Lbullet}{\tilde{\bullet}}
\newcommand{\InfConLift}[1]{\bar{#1}}
\newcommand{\Lgamma}{\InfConLift{\gamma}}
\newcommand{\Lx}{\tilde{x}}
\newcommand{\Ly}{\tilde{y}}
\newcommand{\Lz}{\tilde{z}}

\newcommand{\LA}{\tilde{A}}
\newcommand{\LB}{\tilde{B}}
\newcommand{\LC}{\tilde{C}}
\newcommand{\LDelta}{\Delta\LA\LB\LC}
\newcommand{\paraHF}{P}
\newcommand{\paraKh}{P}

\DeclareMathOperator{\Lspace}{\mathcal{L}}
\newif\ifA
\newif\ifi
\newif\ifm
\newif\ifd
\newcommand{\thth}[1]{
	\operatorname{%
		\ifd%
			\partial%
		\fi%
		\ifA%
			\ifi%
				\mathring{\mathrm{A}}%
			\else%
				\mathrm{A}%
			\fi%
		\else%
			\ifi%
				\mathring{\Theta}%
			\else%
				\Theta%
			\fi%
		\fi%
		\ifm%
			\ifA%
				^{\!\mirror}%
			\else%
				^{\mirror}%
			\fi%
		\fi%
		_{#1}%
	}%
}
\newcommand{\HFsup}{\mathrm{HF}}
\newcommand{\Khsup}{\mathrm{Kh}}
\newcommand{\Gsup}{\mathit{G}}
\newcommand{\makemirror}[2]  {\newcommand{#1}{\mtrue#2\mfalse}}
\newcommand{\makeALink}[2]   {\newcommand{#1}{\Atrue#2\Afalse}}
\newcommand{\makeInterior}[2]{\newcommand{#1}{\itrue#2\ifalse}}
\newcommand{\makeBoundary}[2]{\newcommand{#1}{\dtrue#2\dfalse}}

\newcommand{\Thin}{\thth{}}
\newcommand{\ThinZ}{\thth{\Z}}
\newcommand{\ThinZZ}{\thth{\ZZ}}
\newcommand{\ThinG}{\thth{\Gsup}}
\newcommand{\ThinHF}{\thth{\HFsup}}
\newcommand{\ThinKh}{\thth{\Khsup}}
\makemirror{\mirrorThin}  {\Thin}
\makemirror{\mirrorThinZ} {\ThinZ} 
\makemirror{\mirrorThinZZ}{\ThinZZ} 
\makemirror{\mirrorThinG} {\ThinG} 
\makemirror{\mirrorThinHF}{\ThinHF} 
\makemirror{\mirrorThinKh}{\ThinKh} 

\makeBoundary{\BdryThin}  {\Thin}
\makeBoundary{\BdryThinZ} {\ThinZ} 
\makeBoundary{\BdryThinZZ}{\ThinZZ} 
\makeBoundary{\BdryThinG} {\ThinG} 
\makeBoundary{\BdryThinHF}{\ThinHF} 
\makeBoundary{\BdryThinKh}{\ThinKh} 
\makemirror{\mirrorBdryThin}  {\BdryThin}
\makemirror{\mirrorBdryThinZ} {\BdryThinZ} 
\makemirror{\mirrorBdryThinZZ}{\BdryThinZZ} 
\makemirror{\mirrorBdryThinG} {\BdryThinG} 
\makemirror{\mirrorBdryThinHF}{\BdryThinHF} 
\makemirror{\mirrorBdryThinKh}{\BdryThinKh} 

\makeInterior{\IntThin}  {\Thin}
\makeInterior{\IntThinZ} {\ThinZ} 
\makeInterior{\IntThinZZ}{\ThinZZ} 
\makeInterior{\IntThinG} {\ThinG} 
\makeInterior{\IntThinHF}{\ThinHF} 
\makeInterior{\IntThinKh}{\ThinKh} 
\makemirror{\mirrorIntThin}  {\IntThin} 
\makemirror{\mirrorIntThinZ} {\IntThinZ} 
\makemirror{\mirrorIntThinZZ}{\IntThinZZ} 
\makemirror{\mirrorIntThinG} {\IntThinG} 
\makemirror{\mirrorIntThinHF}{\IntThinHF} 
\makemirror{\mirrorIntThinKh}{\IntThinKh} 

\makeALink{\ALink}  {\Thin}
\makeALink{\ALinkHF}{\ThinHF} 
\makeALink{\ALinkKh}{\ThinKh} 
\makemirror{\mirrorALink}  {\ALink}
\makemirror{\mirrorALinkHF}{\ALinkHF} 
\makemirror{\mirrorALinkKh}{\ALinkKh} 

\makeBoundary{\BdryALink}  {\ALink}
\makeBoundary{\BdryALinkHF}{\ALinkHF} 
\makeBoundary{\BdryALinkKh}{\ALinkKh} 
\makemirror{\mirrorBdryALink}  {\BdryALink}
\makemirror{\mirrorBdryALinkHF}{\BdryALinkHF} 
\makemirror{\mirrorBdryALinkKh}{\BdryALinkKh} 

\makeInterior{\IntALink}  {\ALink}
\makeInterior{\IntALinkHF}{\ALinkHF} 
\makeInterior{\IntALinkKh}{\ALinkKh} 
\makemirror{\mirrorIntALink}  {\IntALink} 
\makemirror{\mirrorIntALinkHF}{\IntALinkHF} 
\makemirror{\mirrorIntALinkKh}{\IntALinkKh} 

\newcommand{\mirror}{\operatorname{m}} 

\newcommand{\FourPuncturedSphere}{S^2_4}
\newcommand{\FourPuncturedSphereKh}{S^2_{4,\ast}}

\newcommand{\PuncturedPlane}{\mathbb{R}^2\smallsetminus \Z^2}

\newcommand{\Sphere}{S^2}
\newcommand{\QPI}{\operatorname{\mathbb{Q}P}^1}
\newcommand{\RPI}{\operatorname{\mathbb{R}P}^1}

\DeclareMathOperator{\TEi}{\textnormal{\texttt{i}}}
\DeclareMathOperator{\TEj}{\textnormal{\texttt{j}}}
\DeclareMathOperator{\TEk}{\textnormal{\texttt{k}}}
\DeclareMathOperator{\TEl}{\textnormal{\texttt{l}}}

\DeclareMathOperator{\TEI}{\textnormal{\texttt{1}}}
\DeclareMathOperator{\TEII}{\textnormal{\texttt{2}}}
\DeclareMathOperator{\TEIII}{\textnormal{\texttt{3}}}
\DeclareMathOperator{\TEIV}{\textnormal{\texttt{4}}}

\newcommand{\vc}[1]{\vcenter{\hbox{#1}}}%
\newcommand{\mypic}[2]{%
  \newcommand{#2}{%
    \vc{%
      \includegraphics[page=#1]%
      {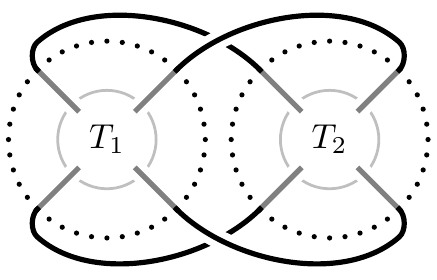}%
    }%
  }%
}%

\mypic{1}{\tanglepairingI}
\mypic{2}{\tanglepairingII}
\mypic{3}{\rigidcurveIntro}
\mypic{4}{\pretzeltangleDownstairs}
\mypic{5}{\pretzeltangleUpstairs}
\mypic{6}{\pretzeltangle}
\mypic{7}{\InlineTrivialTangle}
\mypic{8}{\xI}
\mypic{9}{\xII}
\mypic{10}{\xIII}
\mypic{11}{\xyI}
\mypic{12}{\xyII}
\mypic{13}{\xyIII}
\mypic{14}{\xyIV}
\mypic{15}{\HFTintervalI}
\mypic{16}{\HFTintervalII}
\mypic{17}{\HFTintervalIII}
\mypic{18}{\HFTintervalIV}
\mypic{19}{\HFTtransitivity}
\mypic{20}{\HFTbigon}
\mypic{21}{\Khbigon}
\mypic{22}{\pretzeltangleKh}
\mypic{23}{\pretzeltangleDownstairsKhr}
\mypic{24}{\pretzeltangleUpstairsKhr}
\mypic{25}{\pretzeltangleDownstairsBNr}
\mypic{26}{\pretzeltangleUpstairsBNr}
\mypic{27}{\DotC}
\mypic{28}{\DotB}
\mypic{29}{\DotCblue}
\mypic{30}{\DotBblue}
\mypic{31}{\DotCred}
\mypic{32}{\DotBred}
\mypic{33}{\ConnectivityX}
\mypic{34}{\Li}
\mypic{35}{\Lo}
\mypic{36}{\Ni}
\mypic{37}{\No}
\mypic{40}{\Nil}
\mypic{41}{\Nol}
\mypic{42}{\CrossingL}
\mypic{43}{\CrossingR}
\mypic{50}{\LoDotB}
\mypic{51}{\NiDotL}
\mypic{52}{\NiDotR}
\mypic{53}{\NoDotT}
\mypic{54}{\NoDotB}
\mypic{55}{\closureDiag}
\mypic{56}{\Circle}
\mypic{57}{\CrossingLBasepoint}
\mypic{58}{\CrossingRBasepoint}
\mypic{59}{\PairingTrefoilArcINTRO}
\mypic{60}{\PairingTrefoilLoopINTRO}
\mypic{61}{\KhOneCurve}
\mypic{62}{\KhTwoCurves}
\mypic{63}{\WindingNonSpecialA}
\mypic{64}{\WindingNonSpecialB}
\mypic{65}{\WindingNonSpecialC}
\mypic{66}{\WindingNonSpecialD}
\mypic{67}{\WindingNonSpecialE}
\mypic{68}{\ThinnessObstructionBNr}
\mypic{69}{\FukayaFigureEightArcAlg}
\mypic{70}{\LittleSpecialWrapping}
\mypic{71}{\MuchSpecialWrapping}
\mypic{72}{\GeographyForSpecials}
\mypic{73}{\ExceptionalExample}
\mypic{74}{\tube}
\mypic{75}{\plane}
\mypic{76}{\planedot}
\mypic{77}{\planedotdot}
\mypic{78}{\planedotstar}
\mypic{79}{\SpherePic}
\mypic{80}{\Spheredot}
\mypic{81}{\Spheredotdot}
\mypic{82}{\Spheredotdotdot}
\mypic{83}{\DiscT}
\mypic{84}{\DiscB}
\mypic{85}{\DiscR}
\mypic{86}{\DiscL}
\mypic{87}{\DiscTdot}
\mypic{88}{\DiscBdot}
\mypic{89}{\DiscRdot}
\mypic{90}{\DiscLdot}
\mypic{91}{\GeographyCovering}

\begin{document}
\title{Thin links and Conway spheres}

\author{Artem Kotelskiy}
\address{Mathematics Department \\ Stony Brook University}
\email{artofkot@gmail.com}

\author{Liam Watson}
\address{Department of Mathematics \\ University of British Columbia}
\email{liam@math.ubc.ca}
\thanks{AK is supported by an AMS-Simons travel grant. LW is supported by an NSERC discovery/accelerator grant and was partially supported by funding from the Simons Foundation and the Centre de Recherches Math\'ematiques, through the Simons-CRM scholar-in-residence program. CZ is supported by the Emmy Noether Programme of the DFG, Project number 412851057, and the SFB 1085 Higher Invariants in Regensburg.}

\author{Claudius Zibrowius}
\address{Faculty of Mathematics \\ University of Regensburg}
\email{claudius.zibrowius@posteo.net}

\begin{abstract}
When restricted to alternating links, both Heegaard Floer and Khovanov homology concentrate along a single diagonal $\delta$-grading. 
This leads to the broader class of thin links
that one would like to characterize without reference to the invariant in question. 
We provide a relative version of thinness for tangles and use this to characterize thinness via tangle decompositions along Conway spheres. 
These results bear a strong resemblance to the L-space gluing theorem for three-manifolds with torus boundary. Our results are based on certain immersed curve invariants for Conway tangles, namely the Heegaard Floer invariant \(\HFT\) and the Khovanov invariant \(\Khr\) that were developed by the authors in previous works. 
\end{abstract}

\maketitle

\section{Introduction}\label{sec:intro}

Fox famously asked \begin{center}{\em what is an alternating knot?} \end{center} by which he was interested in knowing if this property could be characterized without reference to knot diagrams; see Lickorish \cite[Chapter 4]{Lickorish-IntroToKnotTheory}. A satisfying answer to Fox's question was provided by Greene \cite{Greene2017} and Howie \cite{Howie2017}: both works prove that a non-split link is alternating if and only if it admits a pair of special spanning surfaces. 

Bar-Natan conjectured \cite{Bar-Natan2002} and  Lee proved \cite{Lee2005} that alternating links have thin Khovanov homology. Subsequently, Ozsv\'ath and Szab\'o proved that alternating links have thin knot Floer homology \cite{OS_alternating,OSHFL}. 
That is to say, the relevant bigraded homology theory in each case is supported along a single diagonal (taking the reduced version in the case of Khovanov's invariant). These diagonals give rise to the integer-valued $\delta$-grading in each theory, so that thinness is defined, algebraically, as follows: 
\begin{definition}
	A \(\delta\)-graded vector space is called \textbf{thin} if it is supported in at most one \(\delta\)-grading.
\end{definition}

A link is called thin if its associated invariant is thin. Bar-Natan's calculations showed that non-alternating thin links exist in Khovanov homology, suggesting a broader class of links that appears harder to pin down. 
Restricting coefficients to the rational numbers for the moment, Dowlin's spectral sequence from Khovanov homology to knot Floer homology \cite{Dowlin2018} implies that if a link is thin as measured by Khovanov homology then it must be thin as measured by knot Floer homology. 
In fact, computations suggest that these notions of thinness coincide.
Thus, the question 
\begin{center}
	{\em what is a thin link?}
\end{center} 
is a natural one. In particular, is there a characterization of thinness that does not depend on the bigraded link homology theory used?
For example, quasi-alternating links were proved to be thin by Manolescu and Ozsv\'ath~\cite{ManolescuOzsvath2008}. Interestingly, thin links that are not quasi-alternating exist~\cite{Greene2010} and indeed arise in infinite families~\cite{GreeneWatson2013}. A larger class has been proposed---two-fold quasi-alternating~\cite{ScadutoStoffregen2018}---and one might ask whether this exactly captures the property of being thin.

Beyond the homology theory in question, thinness may also depend on the coefficient system. 
Indeed, Shumakovitch found a knot whose Khovanov homology is thin when computed over \(\Q\), but not over the two-element field \(\fieldTwoElements\) \cite{shumakovitch2018torsion}; see Example~\ref{exa:thinness_depends_on_coefficients:knots} for further discussion. 
The authors are unaware of any such example for knot Floer homology. 



\subsection*{A change of perspective}
The question {\em `What is a thin link?'} may be placed in a broader context: 
Given any homology theory \(\Homology\) (of CW-complexes, manifolds, links, etc.), a basic observation is that its dimension is bounded below by the absolute value of its Euler characteristic~\(\chi\). Thus, the following is a natural problem:

\begin{problem}\label{problem}
	Characterize the objects \(Y\) for which \(\dim\Homology(Y)=|\chi\Homology(Y)|\). 
\end{problem}

Equivalently, the problem is to classify all objects whose homology is supported in gradings of the same parity. 
Even for singular homology of manifolds, this appears to be a hard question, although some basic facts can be easily established: 
For oriented two-dimensional manifolds, for example, the identity \(\dim\Homology(Y)=|\chi\Homology(Y)|\) characterizes the two-sphere. 
For unoriented two-dimensional manifolds, the situation already becomes more subtle, because the answer depends on the field of coefficients. 
For \(n\) odd, the Euler characteristic of any \(n\)-dimensional closed manifold vanishes, so there are no solutions to this identity. A na\"{i}ve guess for even integers \(n\geq4\) would be that solutions should admit a handle decomposition with no \(i\)-handles for odd \(i\). But such a characterization seems to be difficult to establish even for closed, simply-connected four-manifolds; see~\cite[Problem~4.18]{Kirby}. 

In the context of Ozsváth and Szabó's Heegaard Floer homology \(\HFhat\) for closed oriented three-manifolds,  solutions to \(\dim\HFhat(Y)=|\chi\HFhat(Y)|\) are known under the name \textbf{L-spaces}; see Subsection~\ref{sec:examples:def_Lspaces} for a detailed discussion of this definition. 
In this context, Problem~\ref{problem} relates to the question
\begin{center}
	{\em what is an L-space?}
\end{center} 
(see \cite[Question 11]{OSz-survey}), which 
continues to drive research. 
Ozsv\'ath and Szab\'o proved that L-spaces cannot carry taut foliations~\cite{OS-taut} (see also \cite{Bowden2016,KR2017}). 
At present, the conditions $Y$ not being an L-space, $\pi_1(Y)$ being left-orderable, and $Y$ admitting a taut foliation are known to be equivalent for all graph manifolds~\cite{BoyerClay,L-space_graph_mnflds,Rasmussen2017}. The equivalence of these three conditions is conjectured in general; see~\cite{L-space_conjecture} or \cite{Dunfield18} for further discussion.   

Turning now to link homology theories: 
The reduced Khovanov homology \(\Khr(L;\field)\) of an \(\ell\)-component link \(L\) in \(S^3\) categorifies the Jones polynomial \(V_L(t)\), in the sense that
\[
\chi_{gr}\Khr(L;\field)
\coloneqq
\sum (-1)^h t^{\frac{1}{2}q} \dim \Khr{}^{h,q}(L;\field)=V_L(t)
\]
where \(h\) denotes the homological grading, \(q\) the quantum grading, and $\field$ is some field. 
By setting \(t=1\), we see that the ungraded Euler characteristic with respect to the homological grading is equal to \(V_L(1)=2^{\ell-1}\). Problem~\ref{problem} in this setting was recently solved by Xie and Zhang \cite{XieZhang}, who showed that the identity \(\dim\Khr(L;\field)=|\chi_h\Khr(L;\field)|=2^{\ell-1}\) characterizes forests of unknots (at least if \(\field=\fieldTwoElements\)). 
Similarly, the knot Floer homology \(\HFK(L;\field)\) categorifies the Alexander polynomial \(\Delta_L(t)\):
\[
\chi_{gr}\HFK(L;\field)
\coloneqq
\sum (-1)^h t^{\frac{1}{2}A} \dim \HFK{}^{h,A}(L;\field)=
\Delta_L(t)\cdot(t^{1/2}-t^{-1/2})^{\ell-1}
\]
where \(h\) denotes the homological grading (often called the Maslov grading) and \(A\) the Alexander grading (or more precisely, twice the Alexander grading from~\cite{OSHFK}).
The ungraded Euler characteristic with respect to the homological grading is equal to \(0\) if \(\ell>1\) and \(\Delta_L(1)=1\) if \(\ell=1\). 
Thus, in the first case, there are no solutions to the identity \(\dim\HFK(L;\field)=|\chi_h\HFK(L;\field)|\); in the second case, Problem~\ref{problem} reduces to the question about unknot detection for \(\HFK\), which was settled by Ozsváth and Szabó~\cite{OS-taut}. 

Since both \(\Khr\) and \(\HFK\) are bigraded homology theories, one is not restricted to taking Euler characteristics with respect to the homological grading.
Another choice is the \(\delta\)-grading, which is defined by \(\delta=\frac{1}{2}q-h\) and \(\delta=\frac{1}{2}A-h\), respectively. This corresponds to setting \(t=-1\) in the polynomial invariants:
\begin{align*}
\chi_{\delta}\Khr(L;\field)
&\coloneqq
\sum (-1)^{h+\frac{1}{2}q} \dim \Khr{}^{h,q}(L;\field) = V_L(-1)\\
\chi_{\delta}\HFK(L;\field)
&\coloneqq
\sum (-1)^{h+\frac{1}{2}A} \dim \HFK{}^{h,A}(L;\field) =\pm 2^{\ell-1}\cdot\Delta_L(-1)
\end{align*}
This choice seems to be particularly natural, since
\[
|V_L(-1)|=|\Delta_L(-1)|=\det(L)
\]
where \(\det(L)\) is the determinant of \(L\), a classical link invariant. 
It leads us to consider: 
\begin{definition}
	Given a link homology theory \(\Homology\), an {\bf A-link} is a link $L$ satisfying 
	\[
	\dim\Homology(L)=|\chi_\delta\Homology(L)|
	\]
\end{definition}
In the following, \(\Homology\) will be either \(\Khr\) or \(\HFK\) with coefficients in some field \(\field\). 
Again, there is a dependence on the homology theory $\mathbf{H}_*$ as well as on $\field$, and we will be adding the relevant modifiers where needed. 
Problem~\ref{problem} relates to the question
\begin{center}
	{\em what is an A-link?} 
\end{center} 
by which we are interested in knowing if this property can be described (geometrically or topologically) without reference to a link homology theory. 
Clearly, every thin link is an A-link. 
For \(\Homology=\HFK\), the converse is false, as the family of twisted Whitehead doubles of the trefoil knot that Hedden and Ording consider in~\cite{HO2008} illustrates;  see Example~\ref{exp:whitehead}.
For \(\Homology=\Khr\), we expect that all A-links are thin. This is closely related to the question of full support:

\begin{definition}\label{def:full}
	We say that a link homology theory $\Homology$ has {\bf full support} if for all links \(L\)  and all $\delta$-gradings \(i<j<k\), 
	\[
	\Big(
	\operatorname{\mathbf{H}_\mathit{i}}(L)\neq0
	~\text{ and }~
	\operatorname{\mathbf{H}_\mathit{k}}(L)\neq0
	\Big)
	\Rightarrow
	\operatorname{\mathbf{H}_\mathit{j}}(L)\neq0.
	\] 
\end{definition}

\begin{proposition} Given a link homology theory  \(\Homology\) with full support, a link is thin if and only if it is an A-link. 
\end{proposition}
\begin{proof}
	A link \(L\) is an A-link if and only if \(\mathbf{H}_*(L)\) is supported in gradings of the same parity. Assuming \(\Homology\) has full support, the latter is equivalent to \(\mathbf{H}_*(L)\) being supported in a single grading, ie \(L\) being thin.
\end{proof}

To the best of the authors' knowledge, there is no known example of a link violating full support for Khovanov homology---indeed that this invariant has full support appears to be a folklore conjecture.

Our shift in perspective from thin links to A-links is primarily motivated by the observation that the latter are better-behaved with respect to tangle decompositions along Conway spheres, which is the focus of this article. 
Another reason is the interplay between L-spaces and A-links in the context of two-fold branched covers: There is a spectral sequence due to Ozsváth and Szabó relating the reduced Khovanov homology of a link and the Heegaard Floer homology of the mirror of the two-fold branched cover of the link \cite{OzsvathSzabo2005}. In particular, given an A-link, the associated two-fold branched cover is an L-space. However, the converse is not true: The Poincar\'e homology sphere is an L-space that may be obtained as the two-fold branched cover of the torus knot $10_{124}$, which is not an A-knot. 
Nonetheless, there is a sense in which A-link branch sets might be characterized by sufficiently large L-space surgeries on strongly invertible knots; see the discussion in Section~\ref{sec:examples}, as well as \cite[Conjecture 30]{Watson2017} and \cite{Watson2011} for related examples. 
 
\subsection*{Thin links and Conway spheres}
For simplicity, we now restrict to coefficients in the field of two elements \(\fieldTwoElements\).
We will focus on characterizing thin links and A-links
from the perspective of Conway spheres. 
This is motivated, in part, by results characterizing L-spaces in the presence of an essential torus. Given a three-manifold with torus boundary \(M\) and a parametrization of \(\partial M\) by a meridian \(\mu\) and a longitude \(\lambda\), the space of L-space fillings of \(M\) is defined by
\[
\mathcal{L}(M)
\coloneqq
\{
	\nicefrac{p}{q}\in\QPI
	\mid
	M(\nicefrac{p}{q}) \text{ is an L-space}
\}
\]
where \(M(\nicefrac{p}{q})\) is the closed three-manifold obtained by Dehn filling along the slope \(p\mu+q\lambda\in H_1(\partial M)\). Rasmussen and Rasmussen showed \cite[Proposition~1.3 and Theorem~1.6]{Rasmussenx2}:
\begin{theorem}
	For any three-manifold with torus boundary \(M\), \(\mathcal{L}(M)\) is either empty, a single point, a closed interval or \(\QPI\) minus a single point. 
\end{theorem}

Denote the interior of \(\mathcal{L}(M)\) by \(\mathring{\mathcal{L}}(M)\). Hanselman,  Rasmussen, and the second author establish the following result \cite[Theorem~13]{HRW}:

\begin{theorem}[L-space Gluing Theorem]\label{thm:L_space_gluing}
	Let \(Y = M_0 \cup_h M_1\) be a three-manifold where the \(M_i\)  are boundary incompressible manifolds
	and \(h\co \partial M_1 \rightarrow \partial M_0\) is an orientation reversing homeomorphism between the torus boundaries. Then
	\(Y\) is an L-space if and only if 
	\[
	\mathring{\mathcal{L}}(M_0) \cup h(\mathring{\mathcal{L}}(M_1)) = \QPI
	\]
\end{theorem}
A similar result holds without the assumption that \(M_i\) be boundary incompressible; see Remark~\ref{rem:boundary_compressible}.

 A Conway tangle is an embedding of two intervals and a finite (possibly empty) set of circles into a closed 3-dimensional ball, such that the preimage of the boundary sphere is equal to the four endpoints of the intervals. 
 We consider Conway tangles up to isotopy fixing the boundary sphere pointwise.
 Given a such a tangle \(T\) and a link homology theory, we make analogous definitions:  
\begin{align*}
	\mathrm{A}(T)
	&\coloneqq
	\{
	\nicefrac{p}{q}\in\QPI
	\mid	
	T(\nicefrac{p}{q}) \text{ is an A-link}
	\}
	\\
	\Theta(T)
	&\coloneqq
	\{
	\nicefrac{p}{q}\in\QPI
	\mid	
	T(\nicefrac{p}{q}) \text{ is thin}
	\}
\end{align*}
	where \(T(\nicefrac{p}{q})\) is the \(\nicefrac{p}{q}\)-rational filling of \(T\), that is, the link obtained by closing the tangle $T$ with a \(\nicefrac{-p}{q}\)-rational tangle. 
	We call these the \textbf{A-link filling space} and the \textbf{thin filling space} of the tangle \(T\), respectively. 
	Strictly speaking, we should decorate each of these with the homology theory in question; we will use subscripts to do so where needed. 
	Often, however, the four spaces $\mathrm{A}_{\Kh}(T)$, $\mathrm{A}_{\HF}(T)$, $\Theta_{\HF}(T)$, and \(\Theta_{\Kh}(T)\) coincide. 
	Moreover, the central statements in this paper hold in both the Khovanov and Heegaard Floer setting. 
	Therefore, we will not specify the homology theory in the remainder of this introduction.

\begin{theorem}[Characterization of A-link filling spaces]\label{thm:charactisation:ALink:intro}
	For any Conway tangle \(T\), \(\mathrm{A}(T)\) is either empty, a single point or an interval in \(\QPI\). 
\end{theorem}

\begin{theorem}[Characterization of thin filling spaces]\label{thm:charactisation:Thin:intro}
	For any Conway tangle \(T\),	\(\Thin(T)\) is either empty, a single point, two distinct points or an interval in \(\QPI\).
\end{theorem}

Theorems~\ref{thm:charactisation:ALink:intro} and~\ref{thm:charactisation:Thin:intro} illustrate the difference between A-links and thin links. 
However, while Heegaard Floer A-links need not be Heegaard Floer thin, see Example~\ref{exp:whitehead}, we do not know any tangle for which \(\Thin(T)\) consists of two distinct points, neither in Heegaard Floer nor Khovanov theory. 
If such a tangle exists in Khovanov homology, then the rational fillings of this tangle include Khovanov A-links that are not Khovanov thin. In particular, this would establish that Khovanov homology does not have full support. 
Despite the (potential) existence of such pathological examples, we know that thin and A-link filling spaces coincide generically in the following sense:

\begin{proposition}\label{prop:spaces_coincide:intro}
	If \(\Thin(T)\) is an interval then \(\Thin(T)=\mathrm{A}(T)\). 
\end{proposition}

In contrast with \(\mathcal{L}(M)\), when \(\ALink(T)\) or \(\Thin(T)\) is an interval with two distinct boundary points the interval need not necessarily be closed. 
This suggests that analogues of the L-space Gluing Theorem have to be slightly more subtle. The proofs of all results in this paper rely on the homological invariants \(\HFT(T)\) and \(\Khr(T)\), which are generalizations of Heegaard Floer and Khovanov homology of links to Conway tangles \cite{HDsForTangles,pqMod,pqSym,KWZ}; these invariants are reviewed in Sections~\ref{sec:review:HFT} and~\ref{sec:review:Kh}, respectively. Given two Conway tangles \(T_1\) and \(T_2\), let \(\Gamma_1\) denote the Heegaard Floer/Khovanov tangle invariant of \(T_1^*\), the mirror of \(T_1\), and let \(\Gamma_2\) be the corresponding invariant of \(T_2\). 
The link \(T_1\cup T_2\) is obtained by identifying the two tangles according to the prescription in Figure~\ref{fig:tanglepairing}. 

\begin{definition}\label{def:mirror_slopes}
For a subset of slopes $X \in \Q P^1$, define its \textbf{mirror} as $X^{\mirror}=\{-s~|~ s\in X\}$.
\end{definition}

\begin{figure}[bt]
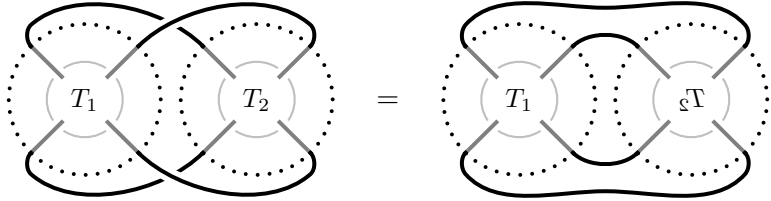

	\centering
	\(
	\tanglepairingI
	\quad = \quad
	\tanglepairingII
	\)
	\caption{Two Conway tangle decompositions defining the link \(T_1\cup T_2\). The tangle \protect\reflectbox{\(T_2\)} is the result of rotating \(T_2\) around the vertical axis. By rotating the entire link on the right-hand side around the vertical axis, we can see that \(T_1\cup T_2=T_2\cup T_1\).}
	\label{fig:tanglepairing} 
\end{figure}

\begin{theorem}[A-Link Gluing Theorem]\label{thm:gluing:ALink:intro}
	\(T_1\cup T_2\) is an A-link if and only if
	\begin{enumerate}
		\item 	\(
		\mirrorALink(T_1)
		\cup 
		\ALink(T_2)
		=
		\QPI
		\); and
		\item \label{exc:ALink:local} certain conditions indexed by 
		\(
		\mirrorBdryALink(T_1)
		\cap
		\BdryALink(T_2)
		\) 
		hold for \(\Gamma_1\) and \(\Gamma_2\). 
	\end{enumerate}
\end{theorem}

The condition \eqref{exc:ALink:local} 
is easy to describe, once the relevant tangle invariants have been reviewed. Note that this condition is vacuously satisfied if \(
\mirrorBdryALink(T_1)
\cap
\BdryALink(T_2)
=
\varnothing
\), which is true generically. This allows us to obtain the following: 

\begin{corollary}\label{cor:one_direction:ALink:intro}
	Let \(\IntALink(T_i)\) denote the interior of \(\ALink(T_i)\) for \(i=1,2\). Then
	\[
	\mirrorIntALink(T_1)
	\cup 
	\IntALink(T_2)
	=
	\QPI
	\quad\Longrightarrow\quad
	\text{\(T_1\cup T_2\) is an A-link}
	\]
\end{corollary}

There is also an analogue of the A-link Gluing Theorem for thinness. 
However, due to the characterization results of A-link versus thin filling spaces, this analogue requires an additional hypothesis about the tangle invariants \(\HFT(T)\) and \(\Khr(T)\). For this, we introduce the notion of Heegaard Floer/Khovanov exceptionality for tangles (Definitions~\ref{def:HF:exceptional} and~\ref{def:Kh:exceptional}). 
Heegaard Floer exceptional tangles do exist, see Example~\ref{exp:whitehead}. 
We conjecture that Khovanov exceptional tangles do not exist. That such a conjecture is reasonable is supported by the following: 

\begin{proposition}\label{prop:exceptional:intro}
	If a Khovanov exceptional tangle exists then there exists a link whose Khovanov homology is supported in precisely two non-adjacent \(\delta\)-gradings.
\end{proposition}
Once more, the question of full support is brought to the foreground.  

\begin{theorem}[Thin Gluing Theorem]\label{thm:gluing:Thin:intro}
Suppose at most one of \(T_1\) and \(T_2\) is exceptional. 
Then	\(T_1\cup T_2\) is thin if and only if
\begin{enumerate}
	\item \(
	\mirrorThin(T_1)
	\cup 
	\Thin(T_2)
	=
	\QPI
	\);
	and
	\item  certain conditions indexed by 
	\(
	\mirrorBdryThin(T_1)
	\cap
	\BdryThin(T_2)
	\) 
	hold for \(\Gamma_1\) and \(\Gamma_2\).
	\end{enumerate}
\end{theorem}

\begin{corollary}\label{cor:one_direction:Thin:intro}
	Let \(\IntThin(T_i)\) denote the interior of \(\Thin(T_i)\) for \(i=1,2\). Then
	\[
		\mirrorIntThin(T_1)
		\cup 
		\IntThin(T_2)
		=
		\QPI
	\quad\Longrightarrow\quad
	\text{\(T_1\cup T_2\) is thin}
	\]
\end{corollary}

Corollary \ref{cor:one_direction:ALink:intro} and Corollary \ref{cor:one_direction:Thin:intro} provide the condition one checks in practice. Some examples are discussed in Section~\ref{sec:examples}.

 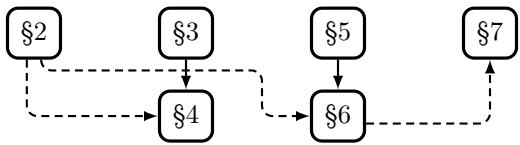
\begin{figure}[t]
 	\tikzstyle{secnode} = [rectangle,rounded corners,draw,very thick,inner sep=5pt]
 	\tikzstyle{secdpath} = [thick,densely dashed,rounded corners]
 	\begin{tikzpicture}[xscale=2,yscale=1.1]\small
 	\node (S2) at (0,1)[secnode] {\hyperref[sec:main]            {\S\ref*{sec:main}}};
 	\node (S3) at (1,1)[secnode] {\hyperref[sec:review:HFT]      {\S\ref*{sec:review:HFT}}};
 	\node (S4) at (1,0)[secnode] {\hyperref[sec:HFT:ThinFillings]{\S\ref*{sec:HFT:ThinFillings}}};
 	\node (S5) at (2,1)[secnode] {\hyperref[sec:review:Kh]       {\S\ref*{sec:review:Kh}}};
 	\node (S7) at (2,0)[secnode] {\hyperref[sec:Kh:ThinFillings] {\S\ref*{sec:Kh:ThinFillings}}};
 	\node (S8) at (3,1)[secnode] {\hyperref[sec:examples]        {\S\ref*{sec:examples}}};
 	\draw [->,thick] (S3) -- (S4);
 	\draw [->,thick] (S5) -- (S7);
 	\draw [->,secdpath] (S7.-15) -| (S8);
 	\draw [->,secdpath] (S2.-105) |- (S4);
 	\draw [->,secdpath] (S2.-75) |- (1.5,0.55) |- (S7);
 	\end{tikzpicture}
 	\caption{The paper's sections and their dependencies. Dashed arrows indicate dependencies that need only statements of results and not the machinery that arise in the proofs, so that the sections in each column may be read in isolation. }
 	\label{fig:map}
 \end{figure}
 
 \subsection*{How to read this paper} 
 The similarities between Heegaard Floer and Khovanov homology, highlighted by the main results of this paper, extend to the arguments that go into the proofs of these results. 
 In fact, the arguments are so similar that they can be presented without reference to either link homology theory. 
 This is done in Section \ref{sec:main}, which requires no specialized knowledge. 
 We then show that both the Heegaard Floer invariant \(\HFT(T)\) (Sections \ref{sec:review:HFT} and \ref{sec:HFT:ThinFillings}) and the Khovanov invariant \(\Kh(T)\) (Sections \ref{sec:review:Kh} and \ref{sec:Kh:ThinFillings})  fit into this general framework.  
 Section~\ref{sec:examples} discusses examples and applications of our main results, focussing primarily on thinness in Khovanov homology. 
 
 The sections of this paper need not be read in order, and depending on the interests of the reader certain sections can be skimmed or even skipped. A flow chart of dependencies is given in Figure \ref{fig:map}. 
 For instance, having read this introduction, the reader may wish to turn immediately to the Examples in Section \ref{sec:examples} in order to get a sense of what one observes in nature. 
 Section \ref{sec:main} is entirely combinatorial and makes no reference to any link homology theory. 
 Sections \ref{sec:review:HFT} and \ref{sec:HFT:ThinFillings} focus on knot Floer homology while Sections \ref{sec:review:Kh} and \ref{sec:Kh:ThinFillings} focus on Khovanov homology following a similar structure: In both cases, we review the relevant tangle invariant in the first section and establish our new results in the second.
\section{Abstracting the main argument}\label{sec:main}

This section lays the combinatorial foundation on which the main results of this paper rely. 
Towards characterizing thin links and A-links without reference to a given homology theory, we find it compelling that, relative to tangle decompositions, thinness is amenable to the elementary combinatorial abstraction described below.  

\subsection{Combinatorics of slopes and lines}
The space of slopes \(\QPI\subset \RPI \cong S^1\), endowed with the subspace topology, carries a natural cyclic order:
Given a finite set of slopes \(\{s_1,\dots,s_n\}\) for some \(n\geq3\), we write 
\[
s_1\leq s_2\leq \dots\leq s_n\leq s_1
\]
if the loop \([0,1]\ni t\mapsto s_1 \cdot e^{2\pi i t}\in S^1\subset\mathbb{C}\) based at \(s_1\) meets \(s_l\) not before \(s_k\) if \(k<l\); in short, we choose the counter-clockwise order, as illustrated in Figure~\ref{fig:Thin:Intervals}. 
We call a tuple \((s_1,\dots,s_n)\) that satisfies this condition \textbf{increasing}. 
Note that \(s_n\neq s_1\) for any such tuple, unless \(s_1=s_2=\dots=s_n\).
If the order is opposite, the tuple is called a \textbf{decreasing tuple}.  
For pairs of distinct slopes the interval notation \((s_1,s_2)\) denotes the set of slopes \(s\) for which \((s_1,s,s_2)\) is increasing. As usual, square and round brackets are used to indicate the inclusion and exclusion of the interval boundaries.

Let \(\Curves=\QPI\times\,G\times\{0,1\}\) where $G$ is either \(\Z\) or \(\ZZ\). When it is necessary to make the distinction between the choice of $G$, we will write \(\Curves=\CurvesG\). 
Elements \(c\in\Curves\) will be called \textbf{lines}; one might represent them geometrically as slopes together with decorations in \(G\times \{0,1\}\). (We choose the terminology \emph{line} for distinction with \emph{curve}, which will have a slightly different meaning in subsequent sections.) 
Given a triple \(c\in\Curves\), denote the first component, the slope of \(c\), by \(s(c)\); denote the second component, the grading of \(c\), by \(\gr(c)\); the third component is denoted by \(\varepsilon(c)\). A line \(c\) is \textbf{rational} if \(\varepsilon(c)=0\) and \textbf{special} if \(\varepsilon(c)=1\). 
Note that \(G\) acts on the set \(\Curves\), and we write \[n\cdot c=n\cdot (s,g,\varepsilon)=(s,g+n,\varepsilon)
\quad \text{for any }n\in G.
\]
Let \(\gr\co \Curves^2\rightarrow G\) be a function satisfying the following identities for all \(c,c',c''\in\Curves\): 
\begin{gather}
\gr(c,c')+\gr(c',c)
=
\begin{cases*}
0 & if \(s(c)=s(c')\)\\
-1 & otherwise
\end{cases*}
\tag{symmetry}\label{eq:symmetry}
\\
\gr(c,c')+\gr(c',c'')=\gr(c,c'') 
\quad\text{ if (\(s(c),s(c'),s(c''))\) is increasing}
\tag{transitivity}\label{eq:transitivity}
\\
\gr(n\cdot c,n'\cdot c')=\gr(c,c')+n'-n 
\tag{linearity}\label{eq:linearity}
\end{gather}

A finite non-empty collection of lines  \(C = \{c_1,\ldots,c_n\}\subset \Curves\) is called a \textbf{line set}. 
We call \(C\) \textbf{\(s\)-rational} if \(\varepsilon(c)=0\) for all \(\{c\in C~|~s(c)=s\}\), and \textbf{\(s\)-special} if \(\varepsilon(c)=1\) for all \(\{ c\in C ~|~ s(c)=s\}\). 

It is often useful to consider the underlying slopes realized by a given line set $C$ in the projection $\Curves\to\QPI$. 
For this purpose we define the set of \textbf{supporting slopes} as
\[\Slopes_C\coloneqq\{s(c)\mid c\in C\}\subset\QPI\]
We call a line set \(C\) \textbf{trivial} if all its lines are special and concentrated in a single slope; in other words, if \(\Slopes_C=\{s\}\) for some slope \(s\in\QPI\), and \(C\) is \(s\)-special. Otherwise, we call \(C\) \textbf{non-trivial}. 

Note that the quotient homomorphism \(\Z\rightarrow\ZZ\) induces a canonical map \(\CurvesZ\rightarrow\CurvesZZ\), which allows us to relate lines and line sets with respect to the two choices of~\(G\). Specifically, the image of a line set \(C\subset\CurvesZ\) under this map is a multi-set; after removing any duplicate elements, we obtain a line set in \(\CurvesZZ\), which by abuse of notation, we will also denote by~\(C\).

\begin{remark}\label{rem:CurveSetAsCoveringSpace}
	In Sections~\ref{sec:HFT:ThinFillings} and~\ref{sec:Kh:ThinFillings}, we will construct the function \(\gr\) with the desired properties in the Heegaard Floer and the Khovanov setting, respectively. 
	However, it is not hard to see that such a function exists and that it is essentially unique. 
	For this, it is useful to think of \(\CurvesZ\) in terms of a covering space of \(\QPI\). 
	More precisely, we can identify \(\Curves_0\coloneqq\QPI\times\Z\times\{0\}\subseteq\CurvesZ\) with the pullback of the universal cover \(p\co\R\rightarrow\RPI\) along the inclusion \(\QPI\hookrightarrow\RPI\). 
	This is done as follows: To define a map \(\eta\co\Curves_0\rightarrow\R\), fix some \(c_\ast\in\Curves_0\) as a basepoint and define \(\eta(c_\ast)\) to be some point \(x_\ast\in p^{-1}(s(c_\ast))\). 
	For each \(s\in\QPI\smallsetminus\{s(c_\ast)\}\), there is some element \(c_s\in\Curves_0\) of slope \(s\) such that \(\gr(c_\ast,c_s)=0\). 
	Let \(\gamma_s\) be an injective path from \(s(c_\ast)\) to \(s\) which goes in counter-clockwise direction. Define \(\eta(c_s)\) as the endpoint of the lift of \(\gamma_s\) starting at \(x_\ast\). 
	Then extend \(\eta\) equivariantly using the action of \(G=\Z\) on \(\Curves_0\) and the action by Deck transformations on \(\R\), where \(+1\) corresponds to a counter-clockwise loop based at \(s_\ast\). 
	Under this identification of \(\Curves_0\) with a subspace of \(\R\), the function \(\gr\) is simply the floor function of the signed distance:
	\[
	\gr(c,c')=\lfloor\eta(c')-\eta(c)\rfloor
	\quad\text{for any }c,c'\in\Curves_0.
	\]
	By taking the product with \(\{0,1\}\), one can easily extend this construction to \(\Curves\). 
	For \(\CurvesZZ\), a similar interpretation is possible---we simply replace the universal cover of \(\RPI\) by the connected two-fold cover.
\end{remark}

\begin{figure}[t]
	\centering
	\labellist \small
	\pinlabel $s_1$ at 130 0
	\pinlabel $s_2$ at 142 19
	\pinlabel $s_3$ at 148 37
	\pinlabel $s_n$ at 153 75
	\pinlabel $s$ at 122 109
	\tiny \pinlabel {\bf THIN} at 69 90
	\endlabellist
	\includegraphics[scale=0.75]{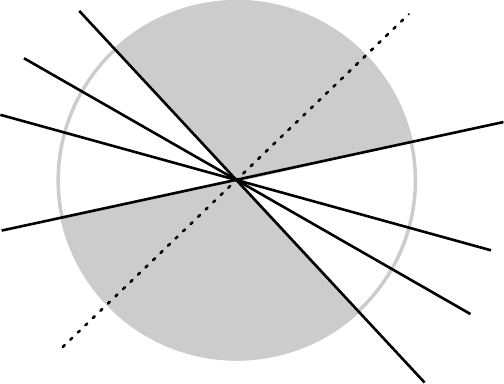}
	\caption{The thin interval relative to an increasing sequence of  slopes $(s_1,s_2,s_3,\ldots,s_n)$}\label{fig:Thin:Intervals}
\end{figure}
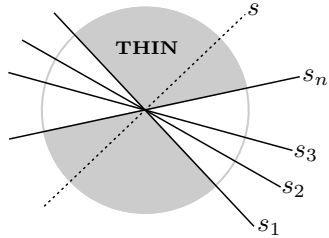

While the expression for the function~\(\gr\) in Remark~\ref{rem:CurveSetAsCoveringSpace} is very concise, we will not make any further use of this perspective. 
Instead, we will only use the properties of the function~\(\gr\), in particular the symmetry and transitivity property.

\begin{definition}\label{def:Consistent}Given \(s\in\QPI\), a line set \(C\) is called \textbf{\(s\)-consistent} if \(\gr(c,c')=0\) for all \(c,c'\in C\) with \(s(c)=s=s(c')\). 
\end{definition}
Notice that if $C$ contains a unique line $c$ for which $s(c)=s$ (or, indeed, if $s\notin \Slopes_C$) then it is $s$-consistent. So in particular, this is a condition that is relevant when multiple lines project to the same slope. 
In fact, if \(C\) is \(s\)-consistent there are at most two lines of slope \(s\) in \(C\), since $C\subset\Curves$. 
We will relax this point of view and allow multi-sets when discussing curves in the Heegaard Floer and Khovanov settings in later sections. 

\begin{definition}\label{def:Thin}
We call a pair \((C,D)\) of line sets \textbf{thin}, or more precisely \textbf{\(\bm{G}\)-thin}, if there exists some constant \(n\in G\) such that for all \((c,d)\in C\times D\),
\[
\begin{cases*}
(\varepsilon(c),\varepsilon(d))\in\{(0,1),(1,0)\}
& 
if \(s(c)= s(d)\)\\
\gr(c,d)=n 
& 
otherwise
\end{cases*}
\]
Note that if \((C,D)\) is thin, then so is \((D,C)\).
For any line set \(C\), we define 
\[
\Thin_G(C)
=
\{s\in\QPI\mid ((s,0,0),C)\text{ is thin}\}
\]
\end{definition}

We write \(\Thin\) in place of \(\Thin_G\) when the statements are true for both \(G=\Z\) and \(G=\ZZ\) or when this group is clear from the context. 

\begin{lemma}\label{lem:Thin:endpoints_are_not_contained}
	Given a line set \(C\), \(s(c)\not\in\Thin(C)\) for every rational \(c\in C\). 
\end{lemma}
\begin{proof}
	This is an immediate consequence of the definitions.
\end{proof}

\begin{lemma}\label{lem:Thin:consistency}
	 Given a line set \(C\), suppose \(s_0\in\Thin(C)\). Then \(C\) is \(s\)-consistent for all slopes \(s\in\QPI\smallsetminus\{s_0\}\).
\end{lemma}

\begin{proof}
	Let us write \(c_0=(s_0,0,0)\). Since \(s_0\in\Thin(C)\), \(\gr(c_0,c)=\gr(c_0,c')\) for any lines \(c,c'\in C\) of slopes different from \(s_0\). In particular, this holds for lines \(c,c'\) of the same slope \(s\neq s_0\). In this case, the triple $(c_0,c,c')$ is increasing, so by the~\ref{eq:transitivity} of the function \(\gr\), $\gr(c,c')=0$. 
\end{proof}

When $\Slopes_C$ is a singleton, there are four cases that arise for the set $\Thin(C)$ depending on consistency and the values of $\varepsilon$. These are recorded in the following lemma. 

\begin{lemma}\label{lem:Thin:SingleSupportingSlope}
	Given a line set \(C\), suppose \(\Slopes_C=\{s\}\) for some \(s\in\QPI\). Then 
	\[
	\Thin(C)=
	\begin{cases*}
	\{s\} & if \(C\) is not \(s\)-consistent and \(s\)-special
	\\
	\QPI & if \(C\) is \(s\)-consistent and \(s\)-special\\
	\varnothing & if \(C\) is not \(s\)-consistent and not \(s\)-special
	\\
	\QPI\smallsetminus\{s\} & if \(C\) is \(s\)-consistent and not \(s\)-special
	\end{cases*}
	\]
\end{lemma}

\begin{proof} 
	Suppose \(C\) is not \(s\)-consistent so that  there exist \(c,c'\in C\) such that \(\gr(c,c')\neq0\). Now consider some ``test'' slope \(s_0\neq s\) and let \(c_0=(s_0,0,0)\). The triple $(c_0,c,c')$ is increasing, so by \ref{eq:transitivity} $\gr(c_0,c)\ne \gr(c_0,c')$. Thus, $s_0\notin\Thin(C)$ and \(\Thin(C)\subseteq\{s\}\). 
Similarly, if \(C\) is \(s\)-consistent, \ref{eq:transitivity} implies 
\(\QPI\smallsetminus\{s\}\subseteq\Thin(T)\).
Finally, appealing to Lemma~\ref{lem:Thin:endpoints_are_not_contained}, \(s\in\Thin(C)\) if and only if all lines \(c\in C\) are special. 
\end{proof}

More generally, for a generic line set $C$ the set $\Thin(C)$ is an interval in $\QPI$, whenever it is non-empty. This behaviour can be characterized precisely as follows. 

\begin{lemma}\label{lem:Thin:Intervals}
  Given a line set \(C=\{c_1,\dots,c_n\}\) write \(s_i=s(c_i)\) and suppose \((s_1,\dots,s_n)\) is increasing with \(s_1\neq s_n\); see Figure~\ref{fig:Thin:Intervals}. Then the following conditions are equivalent:
  \begin{enumerate}
    \item \label{enu:Thin:Intervals:interior} There exists some \(s\in\Thin(C)\) with  \(s\in (s_n,s_1)\);
    \item \label{enu:Thin:Intervals:g_vanishes} \(\gr(c_i,c_j)=0\) for all \(i<j\);
    \item \label{enu:Thin:Intervals:trapped} \((s_n,s_1)\subseteq\Thin(C)\subseteq[s_n,s_1]\);
    \item \label{enu:Thin:Intervals:inclusion} \((s_n,s_1)\subseteq\Thin(C)\).
  \end{enumerate}
\end{lemma}

\begin{proof}
	The implications \(
	\eqref{enu:Thin:Intervals:trapped}\Rightarrow
	\eqref{enu:Thin:Intervals:inclusion}\Rightarrow
	\eqref{enu:Thin:Intervals:interior}\) are obvious. Moreover, the implication 
	\(
	\eqref{enu:Thin:Intervals:interior}\Rightarrow
	\eqref{enu:Thin:Intervals:g_vanishes}
	\) 
	follows from \ref{eq:transitivity} of the function \(\gr\), as in the proof of Lemma \ref{lem:Thin:consistency}. 
	So it suffices to show 
	\(
  \eqref{enu:Thin:Intervals:g_vanishes}\Rightarrow
  \eqref{enu:Thin:Intervals:trapped}.
  \)
  If \eqref{enu:Thin:Intervals:g_vanishes} holds, then, by \ref{eq:transitivity}, \(\gr((s',0,0),c_i)\) is constant for all \(s'\in(s_n,s_1)\), so \((s_n,s_1)\subseteq\Thin(C)\). 
  Moreover, since $s_1$ and $s_n$ differ \[\gr(c_n,c_1)=-1-\gr(c_1,c_n)=-1\] by \ref{eq:symmetry} of the function \(\gr\). Then for any \(s'\in(s_1,s_n)\)
  \[
  \gr((s',0,0),c_1)
  =
  \gr((s',0,0),c_n)+\gr(c_n,c_1)
  =
  \gr((s',0,0),c_n)-1
  \] 
  and hence \(\Thin(C)\cap(s_1,s_n)=\varnothing\). 
  This establishes \eqref{enu:Thin:Intervals:trapped}.
\end{proof}

Taken together, Lemmas \ref{lem:Thin:SingleSupportingSlope} and \ref{lem:Thin:Intervals} capture nearly all of the behaviour that is possible:

\begin{lemma}\label{lem:Thin:discreteSlopes}
	With the same notation as in Lemma~\ref{lem:Thin:Intervals}, suppose \(|\Slopes_C|>2\) and 
	\(\Thin(C)\subseteq\Slopes_C\). Then \(\Thin(C)\subseteq\{s_i\}\) for some \(i\). 
\end{lemma}

\begin{proof}	
	Suppose there exist two distinct slopes \(s,s'\in\Thin(C)\). Then by Lemma~\ref{lem:Thin:consistency}, \(C\) is \(t\)-consistent for all \(t\in\Slopes_C\). 
	Since \(|\Slopes_C|>2\), we may assume that, after potentially reindexing the lines, the slopes \(s_1=s\), \(s_i=s'\), and \(s_n\) are pairwise distinct, that \((s_1,\dots,s_n)\) is increasing, and that \(s_{i-1}\neq s_i\). Let \(j\) be minimal such that \(s_j\neq s\).
	Then, \(\gr(c_k,c_\ell)=0\) for all \(j\leq k<\ell\leq n\), since \(s\in\Thin(C)\). 
	In particular, \(\gr(c_j,c_n)=0\). 
	Since also \(s_i\in\Thin(C)\), we get in addition that \(\gr(c_n,c_k)=0\) for all \(1\leq k<i\). 
	This contradicts the \ref{eq:symmetry} of the function \(\gr\) unless \(i=j\). 
	However, if \(i=j\) then \((s,s')\subset\Thin(C)\) by the direction 
	\(
	\eqref{enu:Thin:Intervals:g_vanishes}\Rightarrow
	\eqref{enu:Thin:Intervals:inclusion}
	\) 
	of Lemma~\ref{lem:Thin:Intervals}, contradicting our initial assumption about \(\Thin(C)\). 
\end{proof}

Therefore, continuing with our observation preceding Lemma \ref{lem:Thin:discreteSlopes}, the only additional case that needs special attention is $|\Slopes_C|=2$. We can now collect all of the forgoing into a clean statement: 

\begin{theorem}[Characterization of \(G\)-thin filling spaces]\label{thm:charactisation:ThinG:main}
	Let \(C\) be a non-trivial line set. Then \(\Thin(C)\) is either empty, a single point, two distinct points, or an interval in \(\QPI\).
	For \(\ThinZZ(C)\), the third case does not arise.  
\end{theorem}

\begin{observation}\label{obs:BoundaryOfThinSubseteqSlopes}
	\(\BdryThin(C)\subseteq \Slopes_C\) for any line set \(C\) by Lemmas~\ref{lem:Thin:SingleSupportingSlope} and~\ref{lem:Thin:Intervals}.  Moreover, if \(C\) is non-trivial, \(\Slopes_C\) is disjoint from the interior of \(\Thin(C)\).
\end{observation}

\begin{proof}[Proof of Theorem~\ref{thm:charactisation:ThinG:main}]
	If \(|\Slopes_C|=1\), both statements follow from Lemma~\ref{lem:Thin:SingleSupportingSlope}. 
	So we can assume in the following that \(|\Slopes_C|\geq2\).
	Let us also assume that \(\Thin(C)\) contains some slope \(s\). 
	If \(s\not\in\Slopes_C\) then \(\Thin(C)\) is an interval by Lemma~\ref{lem:Thin:Intervals}. 
	If \(\Thin(C)\subseteq\Slopes_C\) and \(|\Slopes_C|>2\), the set \(\Thin(C)\) contains at most one slope by Lemma~\ref{lem:Thin:discreteSlopes}. This concludes the proof of the first statement. 
	Suppose \(|\ThinZZ(C)|=2\), say \(\ThinZZ(C)=\{s,s'\}\) for some distinct \(s,s'\in\QPI\). By Lemma~\ref{lem:Thin:discreteSlopes}, \(\Slopes_C=\{s,s'\}\).
	By Lemma~\ref{lem:Thin:consistency}, $C$ is \(s\)- and \(s'\)-consistent. 
	Then, modulo 2, either \(\gr(c,c')=0\) or \(\gr(c',c)=0\) for any two lines \(c,c'\in C\) with \(s(c)=s\) and \(s(c')=s'\).
	So the condition \eqref{enu:Thin:Intervals:g_vanishes} of Lemma~\ref{lem:Thin:Intervals} is met, and thus \(\ThinZZ(C)\) is a (closed) interval, contradicting our initial assumption.
\end{proof}

In the generic situation, the difference between \(G=\Z\) and \(G=\ZZ\) vanishes:

\begin{proposition}\label{prop:spaces_coincide:main}
	If \(\ThinZ(C)\) is an interval, \(\ThinZ(C)=\ThinZZ(C)\).
\end{proposition}
\begin{proof}
	If \(|\Slopes_C|=1\), this follows from the observation that a line set is \(s\)-consistent with respect to \(G=\ZZ\) if it is \(s\)-consistent with respect to \(G=\Z\). 
	If \(|\Slopes_C|\geq2\) and \(\ThinZ(C)\) is an interval then by Lemma~\ref{lem:Thin:Intervals}, \(\ThinZZ(C)\) is an interval with the same endpoints. Moreover, whether an endpoint is contained in \(\Thin_G(C)\) is independent of \(G\).
\end{proof}

\subsection{Characterizing thin pairs of line sets}
We now turn to a characterization of thinness. Before stating the main theorem of this subsection, we discuss a certain exceptional class of line sets which requires special care, but which in the Heegaard Floer and Khovanov settings is ultimately a pathology that we have not observed in practice. 

\begin{definition}\label{def:exceptional}
	We call a line set \(C\) \textbf{exceptional} if  \(\Slopes_C=\{s,s'\}\) for distinct slopes \(s,s'\in\QPI\), \(C\) is \(s\)- and \(s'\)-consistent, but there are lines \(c,c'\in C\) with \(s(c)=s\) and \(s(c')=s'\) such that neither \(\gr(c,c')\) nor \(\gr(c',c)\) are equal to 0.
\end{definition}

Note that if \(G=\ZZ\), there do not exist exceptional line sets. In particular, we have the following result. 

\begin{proposition}
	If \(\ThinZ(C)=\{s,s'\}\) with \(s\neq s'\), then \(\Slopes_C=\{s,s'\}\) and \(\ThinZZ(C)=[s,s']\) or \([s',s]\).
\end{proposition}

\begin{proof}
	If \(|\Slopes_C|=1\), \(|\ThinZ(C)|\neq2\) by Lemma~\ref{lem:Thin:SingleSupportingSlope}. 
	For the case \(|\Slopes_C|\geq2\), the statement follows from the same arguments as the proof of the second statement of Theorem~\ref{thm:charactisation:ThinG:main}.
\end{proof}

\begin{theorem}[\(G\)-thin Gluing Theorem]\label{thm:glueing:ThinG:main}
	Let \((C,D)\) be a pair of non-trivial line sets. Suppose not both \(C\) and \(D\) are exceptional. Then \((C,D)\) is thin if and only if
	\begin{enumerate}
		\item \label{enu:thm:glueing:ThinG:main:global} \(
		\Thin(C)
		\cup
		\Thin(D)
		=\QPI
		\); and 
		\item \label{enu:thm:glueing:ThinG:main:local} for all \(s\in\BdryThin(C)\cap\BdryThin(D)\), at least one of \(C\) and \(D\) is \(s\)-rational.
	\end{enumerate}
\end{theorem}

We first prove a technical lemma that will simplify the proof of Theorem~\ref{thm:glueing:ThinG:main}.

\begin{lemma}\label{lem:Thin:SlopeCaps_equal_ThinCaps}
	Let \((C,D)\) be a pair of non-trivial line sets. Suppose 
	\(
	\Thin(C)
	\cup
	\Thin(D)
	=\QPI
	\).
	Then \(\Slopes_C\cap\Slopes_D=\BdryThin(C)\cap\BdryThin(D)\).
\end{lemma}
\begin{proof}
	The inclusion \(\supseteq\) follows from the first part of Observation~\ref{obs:BoundaryOfThinSubseteqSlopes}. For the inclusion \(\subseteq\), we distinguish four cases, depending on the size of \(|\Slopes_C|\) and \(|\Slopes_D|\). 
	If \(|\Slopes_C|=1=|\Slopes_D|\), either \(\Slopes_C\cap\Slopes_D=\varnothing\), so there is nothing to show, or \(\Slopes_C=\{s\}=\Slopes_D\) for some slope \(s\), in which case 
	\(
	\Thin(C)
	\cup
	\Thin(D)
	=
	\QPI\smallsetminus\{s\}
	\subsetneq
	\QPI
	\) by Lemma~\ref{lem:Thin:SingleSupportingSlope} and the non-triviality of \(C\) and~\(D\). 
	Suppose \(|\Slopes_C|>1\) and \(|\Slopes_D|=1\), say \(\Slopes_D=\{s\}\). If \(D\) is not \(s\)-consistent, the hypothesis is not satisfied by the non-triviality of~\(D\), Lemma~\ref{lem:Thin:SingleSupportingSlope}, and Theorem~\ref{thm:charactisation:ThinG:main}. If \(D\) is \(s\)-consistent, \(\Thin(D)=\QPI\smallsetminus\{s\}\) by Lemma~\ref{lem:Thin:SingleSupportingSlope}, so in particular \(s\in\BdryThin(D)=\Slopes_D\). Moreover, the hypothesis implies that \(s\in\Thin(C)\). If \(s\not\in\Slopes_C\), there is nothing to show, whereas if \(s\in\Slopes_C\), then also \(s\in\BdryThin(C)\) by Lemma~\ref{lem:Thin:Intervals}. 
	If \(|\Slopes_C|=1\) and \(|\Slopes_D|>1\), we repeat the argument with the roles of \(C\) and \(D\) reversed. So it remains to consider the case that \(|\Slopes_C|, |\Slopes_D|>1\). Combining Lemma~\ref{lem:Thin:Intervals} with the hypothesis shows that \(\Thin(C)\) and \(\Thin(D)\) are two intervals. The claim now follows from the second part of Observation~\ref{obs:BoundaryOfThinSubseteqSlopes}. 
\end{proof}

\begin{proof}[Proof of Theorem~\ref{thm:glueing:ThinG:main}]
	We start with a reformulation of condition \eqref{enu:thm:glueing:ThinG:main:local} on the right-hand side of the asserted equivalence. 
	Suppose for a moment that condition (1) in Theorem~\ref{thm:glueing:ThinG:main} holds.
	Then by non-triviality of \(C\) and \(D\) and Lemma~\ref{lem:Thin:SlopeCaps_equal_ThinCaps}, 
	\(
	\BdryThin(C)
	\cap
	\BdryThin(D)
	=
	\Slopes_{C}\cap\Slopes_{D}
	\).	
	Suppose further that \(C\) is \(s\)-rational for some slope \(s\). Then \(s\not\in\Thin(C)\) by Lemma~\ref{lem:Thin:endpoints_are_not_contained}. Therefore, \(s\in\Thin(D)\) and so by the same lemma, \(D\) is \(s\)-special. Similarly, if \(D\) is \(s\)-rational, we can apply the same argument with reversed roles of \(C\) and \(D\).
	It therefore suffices to show that \((C,D)\) is thin if and only if
	\begin{enumerate}
		\item 
		\(
		\Thin(C)
		\cup
		\Thin(D)
		=\QPI
		\); and 
		\item[(2')] for all \(s\in\Slopes_C\cap\Slopes_D\), \(C\) is \(s\)-rational and \(D\) is \(s\)-special or vice versa.
	\end{enumerate}
	Clearly, \((C,D)\) being thin implies condition~(2'). So let us assume from now on that \(C\) and \(D\) satisfy (2').
	Write \(C=\{c_1,\dots,c_m\}\) and \(D=\{d_1,\dots,d_n\}\) for some \(m,n\geq1\), and let \(s_i=s(c_i)\) for \(i=1,\dots,m\) and \(t_j=s(d_j)\) for \(j=1,\dots,n\). We order the components of \(C\) and \(D\) such that both \((s_1,\dots,s_m)\) and \((t_1,\dots,t_n)\) are increasing tuples. The proof proceeds in four cases indexed by |\(\Slopes_C\cap\Slopes_D\)|.
	
	\newcounter{casecount}
	\newcommand{\casei}[1]{%
		\pagebreak[2]\medskip\noindent%
		\textbf{%
			Case~\thecasecount: \(#1\).\stepcounter{casecount}%
		}%
		\nopagebreak\relax%
	}
	
	\casei{\Slopes_C\cap\Slopes_D=\varnothing}
	In this case \((C,D)\) is thin if and only if there exists some \(M\in G\) such that \(\gr(c,d)=M\) for all \((c,d)\in C\times D\). 
	By \ref{eq:transitivity}, this is the case if and only if after some cyclic permutation of the indices
	\[
	(s_1,\dots, s_m,t_1,\dots,t_n)
	\]
	is an increasing tuple such that \(\gr(c_i,c_j)=0\) and \(\gr(d_i,d_j)=0\) for all \(i<j\). (Otherwise, if $\Slopes_C$ and $\Slopes_D$ intertwine, in the sense that there exist  $i,j,k,\ell$ such that $(s_i,t_j,s_k,t_\ell)$ is increasing, \ref{eq:transitivity} implies $M=g(c_i,d_\ell)=g(c_i,d_j)+g(d_j,c_k)+g(c_k,d_\ell)=M+(-1-M)+M$, which is false.) By Lemmas~\ref{lem:Thin:SingleSupportingSlope} and~\ref{lem:Thin:Intervals}, the latter condition is equivalent to \(\Thin(C)\) and \(\Thin(D)\) being two overlapping intervals.
	
	\casei{\Slopes_C\cap\Slopes_D=\{s\}} 
	\begin{enumerate}[label=\textnormal{(\alph*)}]
		\item Suppose \(\Slopes_C=\Slopes_D=\{s\}\). 
		Then, because neither \(C\) nor \(D\) are trivial, \(C\) and \(D\) each contain at least one rational line of slope \(s\). So \((C,D)\) is not thin. Moreover, \(s\) is neither in \(\Thin(C)\) nor in \(\Thin(D)\), so property \eqref{enu:thm:glueing:ThinG:main:global} does not hold either. 

		\item Suppose  \(\Slopes_C=\{s\}\) and \(\Slopes_D\supsetneq\{s\}\). 
		If \(C\) is not \(s\)-consistent, condition~\eqref{enu:thm:glueing:ThinG:main:global} is false. This is because in this case, \(\Thin(C)=\varnothing\) by Lemma~\ref{lem:Thin:SingleSupportingSlope} and non-triviality of \(C\), and  \(\Thin(D)\neq\QPI\) by Lemma~\ref{lem:Thin:Intervals}. On the other hand, \(C\) not being \(s\)-consistent, in conjunction with \ref{eq:transitivity}, implies that \((C,D)\) is not thin, so the equivalence holds in this case. 
		Suppose now that \(C\) is \(s\)-consistent. Then \(\Thin(C)=\QPI\smallsetminus\{s\}\) by non-triviality of \(C\) and Lemma~\ref{lem:Thin:SingleSupportingSlope}. 
		Therefore, condition~\eqref{enu:thm:glueing:ThinG:main:global}
		is equivalent to \(s\in\Thin(D)\).
		Now observe that since \(C\) is non-trivial, it is not \(s\)-special. Since we are assuming that (2') holds, this implies that \(C\) is \(s\)-rational and \(D\) is \(s\)-special. In particular, \(C\) consists of a single rational line. Thus, by~\ref{eq:linearity}, \((C,D)\) is thin if and only if \(((s,0,0),D)\) is thin, ie \(s\in\Thin(D)\).

		\item Suppose \(\Slopes_C\supsetneq\{s\}\) and  \(\Slopes_D=\{s\}\). Same as Case~1b with reversed roles of \(C\) and \(D\). 

		\item Suppose \(|\Slopes_C|,|\Slopes_D|>1\). 
		Let us reindex the lines such that \(s_m\neq s_1=s=t_n\neq t_1\), and  \((s_1,\dots,s_m)\) and \((t_1,\dots,t_n)\) are increasing. After potentially interchanging \(C\) and \(D\), we may assume without loss of generality that \((s,s_k,t_\ell)\) is increasing for some \(k,\ell\) such that \(s_k\neq s \neq t_\ell\). 
		By \ref{eq:transitivity}, \((C,D)\) is thin if and only if 
		(i) \((s,s_m,t_1)\) is an increasing tuple, and (ii) \(\gr(s_i,s_j)=0\) and \(\gr(t_i,t_j)=0\) for \(i<j\), or equivalently, (ii') \((s,t_1)\subseteq\Thin(D)\), and \((s_m,s)\subseteq\Thin(C)\), by Lemma~\ref{lem:Thin:Intervals}.
		Conditions (ii') and (2') imply that \(s\in\Thin(D)\) or \(s\in\Thin(C)\). Together with condition (i), (1) follows. 
		Conversely, suppose conditions (1) and (2') hold. Since by Lemma~\ref{lem:Thin:Intervals}, \(\Thin(C)\) and \(\Thin(D)\) are contained in the closures of open intervals disjoint from any supporting slopes of \(C\) and \(D\), respectively, condition (1) implies (i) and (ii').  
	\end{enumerate}
	
	\casei{\Slopes_C\cap\Slopes_D=\{s,t\}}	
	\begin{enumerate}[label=\textnormal{(\alph*)}]
		\item Suppose \(|\Slopes_C|=2=|\Slopes_D|\). 
		Suppose further that \(C\) is not \(s\)-consistent. 
		Then \(\Thin(C)=\varnothing\) and hence
		\(
		\Thin(C)
		\cup
		\Thin(D)
		\neq\QPI
		\).  
		And indeed, \((C,D)\) is not thin in this case. Similarly, one can show that the theorem holds whenever \(C\) or \(D\) are not \(s\)- and \(t\)-consistent. So now let us assume that \(C\) and \(D\) are \(s\)- and \(t\)-consistent.
		By the assumptions that we have already made, we can write \(C=\{c,c'\}\) and \(D=\{d,d'\}\) where \(s(c)=s(d)=s\) and \(s(c')=s(d')=t\). 
		Then, by Lemma~\ref{lem:Thin:Intervals}, 
		\begin{align*}
		(s,t)\subseteq\Thin(C)
		&\Leftrightarrow
		\gr(c',c)=0
		&
		(t,s)\subseteq\Thin(D)
		&\Leftrightarrow
		\gr(d,d')=0
		\\
		(t,s)\subseteq\Thin(C)
		&\Leftrightarrow
		\gr(c,c')=0
		&
		(s,t)\subseteq\Thin(D)
		&\Leftrightarrow
		\gr(d',d)=0
		\end{align*}
		Now, \((C,D)\) being thin is equivalent to  \(\gr(c,d')=\gr(c',d)\). By \ref{eq:transitivity} \(\gr(c,d')=\gr(c,d)+\gr(d,d')\) and   \(\gr(c',d)=\gr(c',c)+\gr(c,d)\), and so the condition \(\gr(c,d')=\gr(c',d)\) is equivalent to \(\gr(c',c)=\gr(d,d')\). By \ref{eq:symmetry} of \(\gr\), this is equivalent to \(\gr(c,c')=\gr(d',d)\). The latter two conditions, in conjunction with the four equivalences above, are equivalent to the condition  
		\(
		\Thin(C)
		\cup
		\Thin(D)
		\supseteq\QPI\smallsetminus\{s,t\}
		\), 
		since we are assuming that not both \(C\) and \(D\) are exceptional.	
		This is equivalent to condition \eqref{enu:thm:glueing:ThinG:main:global} since by condition~(2'), either \(C\) or \(D\) is \(s\)-special and either \(C\) or \(D\) is \(t\)-special.
			
		\item Suppose \(|\Slopes_C|>2\).
		After potentially interchanging \(t\) and \(s\), we may assume without loss of generality that \((s,s_k,t)\) is increasing for some \(k\) such that \(s\neq s_k \neq t\).
		As in Case~1c, let us reindex the lines such that \(s_m\neq s_1=s=t_n\neq t_1\), and  \((s_1,\dots,s_m)\) and \((t_1,\dots,t_n)\) are increasing. 
		Then by \ref{eq:transitivity}, \((C,D)\) is thin if and only if \(s_m=t=t_1\) and \(\gr(s_i,s_j)=0\) and \(\gr(t_i,t_j)=0\) for all \(i<j\). 
		This, in turn, is equivalent to \(
		\Thin(C)
		\cup
		\Thin(D)
		\supseteq\QPI\smallsetminus\{s,t\}
		\). 
		Now conclude as in Case~2a.
		\item Suppose \(|\Slopes_D|>2\). Same as Case~2b with reversed roles of \(C\) and \(D\). 
	\end{enumerate}
	
	\casei{|\Slopes_C\cap\Slopes_D|>2}
	Say \(s,s',s''\in\Slopes_C\cap\Slopes_D\) are pairwise distinct slopes such that \((s,s',s'')\) is an increasing triple. Then there exist lines \(c,c',c''\in C\) and \(d,d',d''\in D\) such that \(s=s(c)=s(d)\), \(s'=s(c')=s(d')\), and \(s''=s(c'')=s(d'')\). We claim that in this case \((C,D)\) is not thin. Suppose \((C,D)\) were thin. Then \(\gr(c, d')=\gr(c,d'')\), so \(\gr(d',d'')=0\). Cyclically permuting the variables gives \(\gr(d'',d)=\gr(d,d')=0\). Applying \ref{eq:transitivity} and \ref{eq:symmetry} of the function $\gr$, this leads to a contradiction. 
	Now observe that \(\Thin(C)\cup\Thin(D)\neq\QPI\) according to Lemma~\ref{lem:Thin:Intervals}.
\end{proof}

Given any line set \(C\), let \(\IntThin(C)\) denote the interior of \(\Thin(C)\).

\begin{corollary}\label{cor:one_direction:main}
	 Let \((C,D)\) be a pair of non-trivial line sets for which \(\IntThin(C)\cup \IntThin(D)=\QPI\).  Then \((C,D)\) is thin.
\end{corollary}
\begin{proof}
	If \(\IntThin(C)\cup \IntThin(D)=\QPI\) then \(\Thin(C)\cup \Thin(D)=\QPI\) and \(\BdryThin(C)\cap\BdryThin(D)=\varnothing\).
	So under the assumption that not both \(C\) and \(D\) are exceptional, the corollary follows from Theorem~\ref{thm:glueing:ThinG:main}. However, we may drop this assumption, because the only case in the proof of Theorem~\ref{thm:glueing:ThinG:main} in which we use it is Case~2a, which supposes \(|\Slopes_{C}\cap\Slopes_{D}|=2\). Here, however, \(\Slopes_{C}\cap\Slopes_{D}=\varnothing\) by Lemma~\ref{lem:Thin:SlopeCaps_equal_ThinCaps}. 
\end{proof}

This highlights what turns out to be the generic behaviour, in practice, and gives rise to a quick certification of thinness. As the proof of Theorem \ref{thm:glueing:ThinG:main} indicates, the main work is in treating the behaviour at the boundaries of the relevant intervals. Indeed, the converse of Corollary \ref{cor:one_direction:main} is not true as the following example illustrates. 

\begin{figure}[ht]
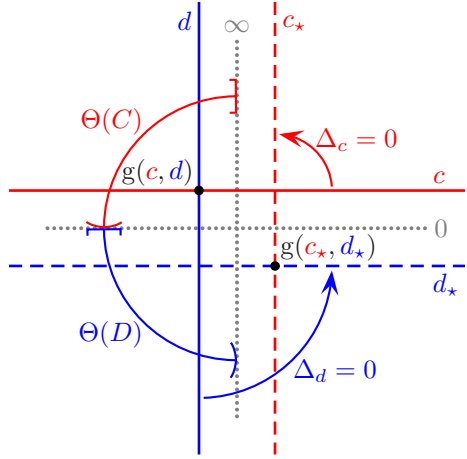

	\centering
	\(\ExceptionalExample\)
	\caption{An illustration of Example~\ref{exa:ExceptionalExample} for the case \(\Delta_c=0=\Delta_d\)}\label{fig:ExceptionalExample}
\end{figure}

\begin{example}\label{exa:ExceptionalExample}
	Let \(C=\{c,c_\star\}\) and \(D=\{d,d_\star\}\) where \(c\) and \(d\) are rational, \(c_\star\) and \(d_\star\) are special, \(s(c)=0=s(d_\star)\), and 
	\(s(c_\star)=\infty=s(d)\). Let \(\Delta_c=\gr(c,c_\star)\) and \(\Delta_d=\gr(d,d_\star)\). 
	Then, 
	\[
	\Thin(C)=
	\begin{cases*}
	[\infty,0) & if \(\Delta_c=0\)\\
	(0,\infty] & if \(\Delta_c=-1\)\\
	\{\infty\} & otherwise
	\end{cases*}
	\quad\text{and}\quad
	\Thin(D)=
	\begin{cases*}
	[0,\infty) & if \(\Delta_d=0\)\\
	(\infty,0] & if \(\Delta_d=-1\)\\
	\{0\} & otherwise
	\end{cases*}
	\]
	See Figure~\ref{fig:ExceptionalExample} for an illustration of one of those cases. 
	Clearly, the hypothesis of Corollary~\ref{cor:one_direction:main} is not satisfied for any values of \(\Delta_c\) and \(\Delta_d\). 
	Moreover, 
	\[
	\gr(c,d)-\gr(c_\star,d_\star)=
	\Delta_c+\gr(c_\star,d)-(\gr(c_\star,d)+\Delta_d)=
	\Delta_c-\Delta_d,
	\]
	so \((C,D)\) is thin if and only if \(\Delta_c=\Delta_d\). If \(\Delta_c\in\{0,-1\}\) or \(\Delta_d\in\{0,-1\}\), we can verify this independently using Theorem~\ref{thm:glueing:ThinG:main}. Otherwise, both line sets are exceptional. 
\end{example}

\section{\texorpdfstring{The tangle invariant \(\HFT\)}{The tangle invariant HFT}}\label{sec:review:HFT}

We review some properties of the immersed curve invariant \(\HFT\) of Conway tangles due to the third author \cite{pqMod}; see also \cite{HDsForTangles, pqSym}.

\subsection{\texorpdfstring{The definition of \(\HFT\)}{The definition of HFT}}\label{sec:review:HFT:definition}

Given a Conway tangle \(T\) in a three-ball \(B^3\), the invariant \(\HFT(T)\) takes the form of a multicurve on a four-punctured sphere \(\FourPuncturedSphere\), which can be naturally identified with the boundary of \(B^3\) minus the four tangle ends \(\partial T\). 
Here, a multicurve is a collection of immersed curves with local systems. 
To make this precise: 
An immersed curve in \(\FourPuncturedSphere\) is an immersion of \(S^1\), considered up to homotopy, that defines a primitive element of \(\pi_1(\FourPuncturedSphere)\), and 
each of these curves is decorated with a local system, ie an invertible matrix over \(\fieldTwoElements\) 
considered up to matrix similarity. 
Local systems can be viewed as vector bundles up to isomorphism, where either \(\fieldTwoElements\) is equipped with the discrete topology or the bundle is equipped with a flat connection. 
We will always drop local systems from our notation when they are trivial, ie if they are equal to the unique one-dimensional local system.
Immersed curves carry a \(\delta\)-grading (described in Section~\ref{sec:HFT:ThinFillings}) and multiple parallel immersed curves in the same \(\delta\)-grading are set to be equivalent to a single curve with a local system that is the direct sum of the individual local systems. 
We will always assume that parallel immersed curves are bundled up this way. 
Finally, a multicurve is a collection of \(\delta\)-graded immersed curves. 

With this terminology in place, we can sketch the construction of \(\HFT(T)\). It is defined in two steps; for details, see~\cite{pqMod}. 

First, one fixes a particular auxiliary parametrization of \(\partial B^3\smallsetminus \partial T\) by four embedded arcs connecting the tangle ends. For example, the four gray dotted arcs in Figure~\ref{fig:HFT:example:tangle} define such a parametrization for the \((2,-3)\)-pretzel tangle. 
A tangle with such a parametrization can be encoded in a Heegaard diagram \((\Sigma,\bm{\alpha},\bm{\beta})\), where \(\Sigma\) is some surface with marked points. 
From this, one defines a relatively \(\delta\)-graded curved chain complex \(\CFTd(T)\) over a certain fixed \(\fieldTwoElements\)-algebra \(\Ad\) as the multi-pointed Heegaard Floer theory of the triple \((\Sigma,\bm{\alpha},\bm{\beta})\), similar to Ozsváth and Szabó's link Floer homology~\cite{OSHFL}.  
One can show that the relatively \(\delta\)-graded chain homotopy type of \(\CFTd(T)\) is an invariant of the tangle \(T\) with the chosen parametrization \cite[Theorem~2.17]{pqMod}. 

The second step uses a classification result, which states that the chain homotopy classes of \(\delta\)-graded curved chain complexes over \(\Ad\) are in one-to-one correspondence with free homotopy classes of \(\delta\)-graded immersed multicurves on the four-punctured sphere \(\FourPuncturedSphere\) \cite[Theorem~0.4]{pqMod}. This correspondence uses a fixed parametrization of \(\FourPuncturedSphere\) by four arcs, and we will generally assume that the multicurves intersect this parametrization minimally. Roughly speaking, the intersection points of arcs with a multicurve correspond to generators of the according curved chain complexes and paths between those intersection points correspond to the differentials. Now, \(\HFT(T)\) is defined as the collection of relatively \(\delta\)-graded immersed curves on \(\FourPuncturedSphere\) corresponding to the curved complex \(\CFTd(T)\). In this definition the parametrization of \(\FourPuncturedSphere\) (needed for multicurves) is identified with the parametrization of \(\partial B^3\smallsetminus \partial T\) (needed for \(\CFTd(T)\)), and one can show that this identification is natural. Namely, if a tangle \(T'\) is obtained from \(T\) by adding twists to the tangle ends, the complex \(\CFTd(T')\) determines a new set of immersed curves \(\HFT(T')\), which agrees with the one obtained by twisting the immersed curves \(\HFT(T)\) accordingly \cite[Theorem~0.2]{pqSym}:
\begin{theorem}\label{thm:HFT:Twisting}
	For all 
	\( 
	\tau\in \Mod(\FourPuncturedSphere)
	\), 
	\(
	\HFT(\tau(T))
	=
	\tau(\HFT(T))
	\).
	In other words, the invariant \(\HFT\) commutes with the action of the mapping class group of the four-punctured sphere.
\end{theorem}

\begin{example}
  Figure~\ref{fig:HFT:example:Curve:Downstairs} shows the four-punctured sphere \(\FourPuncturedSphere\), drawn as the plane plus a point at infinity minus the four punctures labelled \(\TEI\), \(\TEII\), \(\TEIII\), and \(\TEIV\), together with the standard parametrization that identifies \(\FourPuncturedSphere\) with \(\partial B^3\smallsetminus \partial T\). 
  The dashed curve along with the two immersed curves winding around the punctures form the invariant \(\HFT(P_{2,-3})\) for the \((2,-3)\)-pretzel tangle \cite[Example~2.26]{pqMod}.   
\end{example}

\begin{definition}
A (parametrized) tangle is called {\bf rational} if it is obtained from the trivial tangle \(\InlineTrivialTangle\) by adding twists to the tangle ends.
\end{definition}

The name rational tangle originated with Conway, who showed that these tangles are in one-to-one correspondence with fractions \(\nicefrac{p}{q}\in\QPI\)~\cite{Conway}. 
  We denote the rational tangle corresponding to a slope \(\nicefrac{p}{q}\in\QPI\) by \(Q_{p/q}\). 
  The invariant \(\HFT(Q_{p/q})\) consists of a single embedded curve which is the boundary of a disk embedded into \(B^3\) that separates the two tangle strands of \(Q_{p/q}\) \cite[Example~2.25]{pqMod}. 
  The local system on this curve is one-dimensional.
It is known that \(\HFT\) detects rational tangles, as follows.

\begin{theorem}\cite[Theorem~6.2]{pqMod}
\label{thm:HFT:rational_tangle_detection}
A tangle \(T\) is rational if and only if \(\HFT(T)\) consists of a single embedded component carrying the unique one-dimensional local system. 
\end{theorem}

\begin{figure}[t]
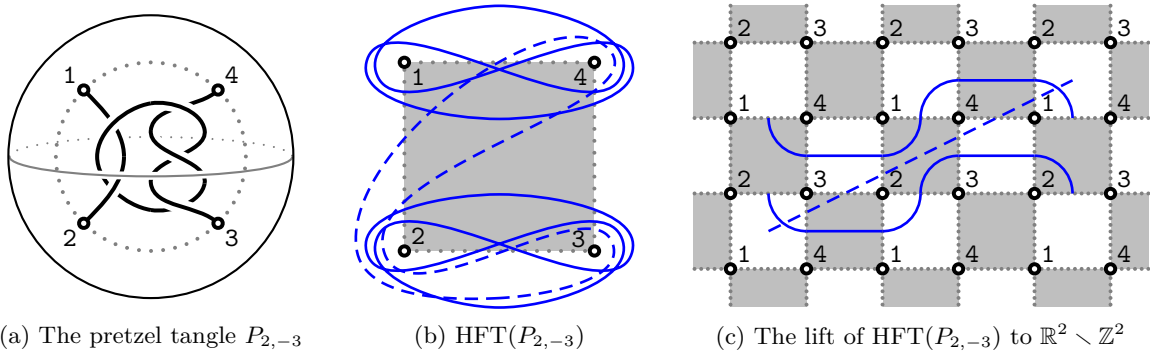

  \centering
  \begin{subfigure}{0.28\textwidth}
  	\centering
  	\(\pretzeltangle\)
  	\caption{The pretzel tangle \(P_{2,-3}\)}\label{fig:HFT:example:tangle}
  \end{subfigure}
	\begin{subfigure}{0.28\textwidth}
		\centering
		\(\pretzeltangleDownstairs\)
		\caption{\(\HFT(P_{2,-3})\)}\label{fig:HFT:example:Curve:Downstairs}
	\end{subfigure}
	\begin{subfigure}{0.4\textwidth}
		\centering
		\(\pretzeltangleUpstairs\)
		\caption{The lift of \(\HFT(P_{2,-3})\) to \(\PuncturedPlane\)}\label{fig:HFT:example:Curve:Upstairs}
	\end{subfigure}
  \caption{A simple non-rational tangle and its Heegaard Floer tangle invariant}\label{fig:HFT:example}
\end{figure} 


\subsection{\texorpdfstring{A gluing theorem for \(\HFT\)}{A gluing theorem for HFT}}\label{sec:review:HFT:gluing}
The invariant \(\HFT(T)\) can be also defined via Zarev's bordered sutured Heegaard Floer theory~\cite{Zarev}. 
In this alternate construction, the curved chain complex \(\CFTd(T)\) is replaced by an {\it a posteriori} equivalent algebraic object, namely the bordered sutured type~D structure associated with the tangle complement, which is equipped with a certain bordered sutured structure; see~\cite[Section~5]{pqSym}. 
This perspective 
gives rise to a 
gluing result \cite[Theorem~5.9]{pqMod} which relates the invariant \(\HFT\) to link Floer homology \(\HFL\) via Lagrangian Floer homology \(\HF\).  
We will always assume that tangles are glued as in Figure~\ref{fig:tanglepairing}, and when such a decomposition exists, we write \(L=T_1\cup T_2\) for a link $L$ consisting of tangles $T_1$ and $T_2$. The mirror image of the link $L$ is expressed as $L^*$; this notation extends to tangles so that the mirror of $T_i$, expressed $T_i^*$, is obtained by interchanging over- and under-crossings. So  \(L^*=T_1^*\cup T_2^*\). (Note that the mirror $T^*$ of a tangle is expressed as $\mirror T$ in other papers.) 
Let \(V\) be a two-dimensional vector space supported in a single relative \(\delta\)-grading. 

\begin{theorem}\label{thm:GlueingTheorem:HFT}
 If  \(L=T_1\cup T_2\) then
  \[
  \HFK(L)\otimes V
  \cong
  \HF(\HFT(T_1^*),\HFT(T_2))
  \]
  if the four open components of the tangles become identified to the same component and 
  \[
  \HFK(L)
  \cong
  \HF(\HFT(T_1^*),\HFT(T_2))
  \]
  otherwise. 
\end{theorem}

In this theorem, the knot Floer homology \(\HFK(L)\) should be understood as a \(\delta\)-graded theory. A similar gluing theorem holds in the bigraded setting and also for link Floer homology, using a multivariate Alexander grading  on the tangle invariants. 

The Lagrangian Floer homology \(\HF(\gamma,\gamma')\) of two immersed curves with local systems \(\gamma\) and \(\gamma'\) is a vector space generated by intersection points between the two curves. More precisely, one first arranges that the components are transverse and do not cobound immersed annuli; then, \(\HF(\gamma,\gamma')\) is the homology of the following chain complex: for each intersection point between \(\gamma\) and \(\gamma'\), there are \(n\cdot n'\) corresponding generators of the underlying chain module, where \(n\) and \(n'\) are the dimensions of the local systems of \(\gamma\) and \(\gamma'\), respectively. The differential is defined by counting certain bigons between these intersection points. As a consequence, the dimension of \(\HF(\gamma,\gamma')\) is equal to the minimal intersection number between the two curves times the dimensions of their local systems, provided that the curves are not parallel. If the curves are parallel, the dimension of \(\HF(\gamma,\gamma')\) may be greater than the minimal geometric intersection number for certain choices of local systems; for details, see~\cite[Sections~4.5 and~4.6, in particular Theorem~4.45]{pqMod}. For a more explicit example, suppose $\gamma$ and $\gamma'$ are parallel embedded curves of the same slope equipped with local systems of dimensions $n$ and $n'$ respectively.  Then, $\dim \HF(\gamma,\gamma')$ can realize any even number between 0 and $2 (n \cdot n')$, depending on the local systems, even though the minimal geometric intersection number between these curves is zero.  Throughout, we will always assume that \(\gamma\) and \(\gamma'\) intersect minimally without cobounding immersed annuli.   

\begin{definition}\label{def:mirror_curve}
For a curve $\gamma$ in \(\FourPuncturedSphere\), its \textbf{mirror} $\mirror (\gamma)$ is the image under the involution of \(\FourPuncturedSphere\) that fixes the punctures and arcs and interchanges the gray and white faces from Figure~\ref{fig:HFT:example:Curve:Downstairs}.
\end{definition}
This operation is important in relating 
\(\HFT(T_1^*)\) to \(\HFT(T_1)\): (see \cite[Definition~5.3 and Proposition~5.4]{pqMod})

\begin{lemma}\label{lem:mirroring:HFT}
For any Conway tangle \(T\), \(\HFT(T^*)=\mirror(\HFT(T))\).
\end{lemma}

For example, because rational tangles satisfy $Q_s^*=Q_{-s}$ we have that \(\HFT(Q_{-s})=\mirror(\HFT(Q_s))\). 

\subsection{\texorpdfstring{The geography problem for \(\HFT\)}{The geography problem for HFT}}
\label{sec:review:HFT:geography}

Often, it is useful to lift immersed curves to a covering space of \(\FourPuncturedSphere\), namely the plane \(\mathbb{R}^2\) minus the integer lattice \(\mathbb{Z}^2\). 
We may regard \(\mathbb{R}^2\) as the universal cover of the torus, and the torus as the two-fold branched cover of the sphere \(\Sphere\) branched at four marked points; then the integer lattice \(\mathbb{Z}^2\) is the preimage of the branch set.
This covering space is illustrated in Figure~\ref{fig:HFT:example:Curve:Upstairs}, where the standard parametrization of \(\FourPuncturedSphere\) has been lifted to \(\PuncturedPlane\) and the front face and its preimage under the covering map are shaded grey. 
This picture also includes the lifts of the curves in \(\HFT(P_{2,-3})\): The lift of the embedded (dashed) curve is a straight line of slope \(\nicefrac{1}{2}\), while the lifts of the two non-embedded components of  \(\HFT(P_{2,-3})\) look more complicated. Remarkably, however, this example shows almost all the complexity of the immersed curves that can appear as components of \(\HFT(T)\) for Conway tangles \(T\). 

To understand the geography of components of \(\HFT(T)\) for general tangles \(T\), observe that the linear action on the covering space \(\PuncturedPlane\) by \(\mathit{SL}(2,\mathbb{Z})\) corresponds to Dehn twisting in \(\FourPuncturedSphere\), or equivalently, adding twists to the tangle ends; for specific conventions see \cite[Observation~3.2]{pqSym}.

\begin{definition}\label{def:HFT_rational_special}
We call a curve in \(\FourPuncturedSphere\) \textbf{rational} if its lift to \(\PuncturedPlane\) is a straight line of slope \(\nicefrac{p}{q}\). We denote such a curve by \(\Rational(\nicefrac{p}{q})\) if it has a trivial local system, and \(\Rational_{X}(\nicefrac{p}{q})\) if it has a local system $X$.

We call a curve in \(\FourPuncturedSphere\) \textbf{special} if, after some twisting, it is equal to the curve \(\Special_n(0;\TEi,\TEj)\) whose lift to \(\PuncturedPlane\) is shown in Figure~\ref{fig:intro:curves}. The lift of any special curve can be isotoped into an arbitrarily small neighbourhood of a straight line of some rational slope \(\nicefrac{p}{q}\in\QPI\) going through some punctures \(\TEi\) and \(\TEj\), in which case we denote this curve by \(\Special_n(\nicefrac{p}{q};\TEi,\TEj)\).
\end{definition}

The term \emph{rational} is chosen because for rational tangles \(\HFT(Q_{p/q})=\Rational(\nicefrac{p}{q})\). 
One can then show \cite[Theorem~0.5]{pqSym}:

\begin{figure}[t]
  \centering
  \(\rigidcurveIntro\)
  \caption{The lift of the curve \(\Special_n(0;\TEi,\TEj)\), where \(n\in\mathbb{N}\) and \((\TEi,\TEj)=(\TEIV,\TEI)\) or \((\TEII,\TEIII)\)}\label{fig:intro:curves}
\end{figure}

\begin{theorem}\label{thm:geography_of_HFT}
  For a Conway tangle \(T\) the underlying curve of each component of \(\HFT(T)\) is either rational or special. 
  Moreover, if it is special, its local system is equal to an identity matrix.
\end{theorem}
For example, we can now write
\(\HFT(P_{2,-3})\) as the union of the rational curve \(\Rational(\nicefrac{1}{2})\) and the two special components \(\Special_1(0;\TEIV,\TEI)\) and \(\Special_1(0;\TEII,\TEIII)\). 
Whether rational components with non-trivial local systems appear in \(\HFT\) is currently not known. 
Special components for $n>1$ show up in the invariants of two-stranded pretzel tangles with more twists \cite[Theorem~6.9]{pqMod}. 
Special components always come in pairs according to the following result, which is a simplified version of 
\cite[Theorem~0.10]{pqSym}.

\begin{theorem}[Conjugation symmetry]\label{thm:Conjugation}
  Let \( (\TEi, \TEj , \TEk, \TEl )\) be some permutation of \( (\TEI,\TEII,\TEIII,\TEIV )\) and let \(\nicefrac{p}{q}\in\QPI\). 
  Then, for any Conway tangle \(T\), the numbers of components \(\Special_n(\nicefrac{p}{q};\TEi,\TEj)\) and \(\Special_n(\nicefrac{p}{q};\TEk,\TEl)\) in \(\HFT(T)\) in any given \(\delta\)-grading agree. 
\end{theorem}

There are also restrictions on rational components. The following is \cite[Observation 6.1]{pqMod}.

\begin{lemma}\label{lem:HFT_detects_connectivity}
  Each rational component of \(\HFT(T)\) separates the four punctures into two pairs, which agrees with how the two open components of \(T\) connect the tangle ends.
\end{lemma}

In analogy with Section~\ref{sec:main}, given some slope \(s\in\QPI\), we will call a multicurve \textbf{\(s\)-rational} if it does not contain any special components of slope \(s\), and \textbf{\(s\)-special} if it does not contain any rational components of slope \(s\).

%
%
%

%

\section{Heegaard Floer thin fillings}\label{sec:HFT:ThinFillings}

We now turn out attention to gradings. Following \cite[Definitions~4.28 and~5.1]{pqMod}, the \(\delta\)-grading of an immersed multicurve \(\Gamma\) is a function 
	\[
	\delta\co \Gen(\Gamma)\longrightarrow \tfrac{1}{2}\mathbb{Z}
	\]
	where \(\Gen(\Gamma)\) is the set of intersection points between the four parametrizing arcs in \(\FourPuncturedSphere\) and~\(\Gamma\), assuming that this intersection is minimal. The function \(\delta\) is subject to the following compatibility condition: 
	Suppose \(x,x'\in \Gen(\Gamma)\) are two intersection points such that there is a path \(\psi\)  on \(\Gamma\) which connects \(x\) to \(x'\) without meeting any parametrizing arc (except at the endpoints). We distinguish three cases, which are illustrated in Figure~\ref{fig:domains:xx}: a path can turn left (a), can go straight across (b), or can turn right (c). Then 
	\[
	\delta(x')-\delta(x)=
	\begin{cases*}
	\tfrac{1}{2} & if the path \(\psi\) turns left,\\
	0 & if the path \(\psi\) goes straight across,\\
	-\tfrac{1}{2} & if the path \(\psi\) turns right.\\
	\end{cases*}
	\]

Given a Conway tangle \(T\), the generators of the underlying module of the invariant \(\CFTd(T)\) are in one-to-one correspondence with elements of \(\Gen(\HFT(T))\). Moreover, these generators are homogeneous with respect to the \(\delta\)-grading, so the relative \(\delta\)-grading on \(\CFTd(T)\) determines the relative \(\delta\)-grading on the corresponding multicurve \(\HFT(T)\). 

\begin{figure}[bt]
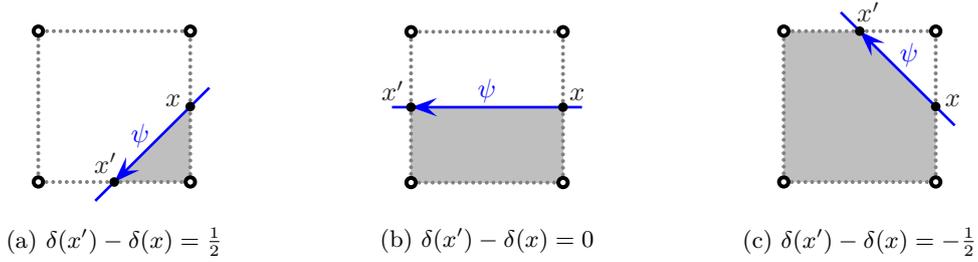

	\begin{subfigure}{0.3\textwidth}
		\centering
		\(\xI\)
		\caption{\(\delta(x')-\delta(x)=\frac 1 2 \)}\label{fig:domains:xx:1}
	\end{subfigure}
	\begin{subfigure}{0.3\textwidth}
		\centering
		\(\xII\)
		\caption{\(\delta(x')-\delta(x)=0 \)}\label{fig:domains:xx:2}
	\end{subfigure}
	\begin{subfigure}{0.3\textwidth}
		\centering
		\(\xIII\)
		\caption{\(\delta(x')-\delta(x)=-\frac 1 2 \)}\label{fig:domains:xx:3}
	\end{subfigure}
	\caption{Basic regions illustrating the definition of the \(\delta\)-grading on a single curve}\label{fig:domains:xx}
\end{figure}

Like link Floer homology, the invariant \(\HFT\) comes with a relative bigrading. In this paper we are not concerned with the Alexander grading; our focus is exclusively on the \(\delta\)-grading. However, we note that the treatment of the grading that follows runs along similar lines to that of~\cite{LMZ} used to study the Alexander grading. 

\subsection{\texorpdfstring{The \(\delta\)-grading of curves in the covering space}{The δ-grading of curves in the covering space}}\label{sec:HFT:Thin:Covering}
We now develop tools that enable us to better understand the \(\delta\)-grading in terms of the covering space \(\PuncturedPlane\) of the four-punctured sphere \(\FourPuncturedSphere\). 
In Subsection~\ref{sec:HFT:Thin:results}, this will allow us  to reduce to the situation of Section~\ref{sec:main} and to apply the theorems from that section to \(\HFT\).

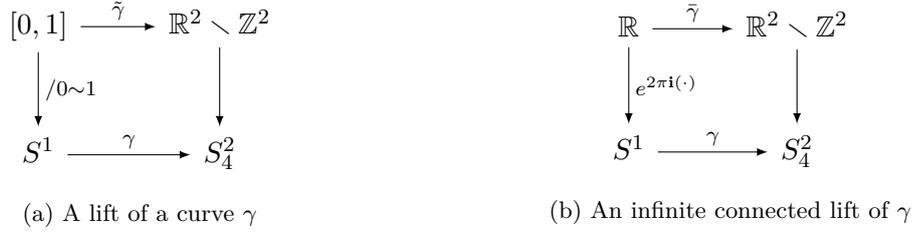
\begin{figure}[t]
  \begin{subfigure}{0.48\textwidth}
    $$
    \begin{tikzcd}[row sep=1cm, column sep=1cm]
    [0,1] 
    \arrow{r}{\tilde{\gamma}}
    \arrow{d}{/0\sim1}
    &
    \PuncturedPlane
    \arrow{d}
    \\
    S^1
    \arrow{r}{\gamma}
    &
    \FourPuncturedSphere
    \end{tikzcd}
    $$
    \caption{A lift of a curve \(\gamma\)}\label{subfig:lift}
  \end{subfigure}
  \begin{subfigure}{0.48\textwidth}
    $$
    \begin{tikzcd}[row sep=1cm, column sep=1cm]
    \R 
    \arrow{r}{\InfConLift{\gamma}}
    \arrow{d}{e^{2\pi \mathbf{i} (\cdot)}}
    &
    \PuncturedPlane
    \arrow{d}
    \\
    S^1
    \arrow{r}{\gamma}
    &
    \FourPuncturedSphere
    \end{tikzcd}
    $$
    \caption{An infinite connected lift of \(\gamma\)}\label{subfig:infconlift}
  \end{subfigure}
  \caption{Lifts versus infinite connected lifts, used for studying curves via the planar cover. In this section we mainly use infinite connected lifts; in Section~\ref{sec:examples} we use lifts nearly exclusively, but for illustration purposes the infinite connected lift is sometimes depicted as well. Note that with this nomenclature, the preimage of a curve in the cover may be called the infinite non-connected lift. }\label{}
\end{figure}

\begin{definition}
Recall that given a map  $\gamma\co S^1\to\FourPuncturedSphere$, its \textbf{lift} to \(\PuncturedPlane\) is  a map $\tilde{\gamma}\co [0,1]\to \PuncturedPlane$ such that the diagram in Figure~\ref{subfig:lift} commutes.
Given a map  $\gamma\co S^1\to\FourPuncturedSphere$, its \textbf{infinite connected lift}  to \(\PuncturedPlane\) is  a map $\InfConLift{\gamma}\co \R \to \PuncturedPlane$ such that the diagram in Figure~\ref{subfig:infconlift} commutes. 
\end{definition}

For notation, in this section any symbol decorated with a tilde~\(\sim\) on top  will denote the lift to \(\PuncturedPlane\); likewise, the symbol~\(-\) on top will denote the infinite connected lift. Infinite connected lifts are sometimes referred to as ‘‘lifts'' for simplicity, where the difference is clear from the context.
In the following, we will treat all points in the integer lattice as marked points (as opposed to punctures). Denote by \(\paraHF\) the union of the integer lattice points with the preimage of the parametrization of \(\FourPuncturedSphere\).  

\begin{definition}
	Suppose \(\Gamma=\{\Lgamma_1,\dots,\Lgamma_n\}\) is a set of curves in \(\PuncturedPlane\) avoiding the integer lattice points such that \(\paraHF\cup\Gamma=\paraHF\cup\Lgamma_1\cup\dots\cup\Lgamma_n\) is a planar graph whose vertices have all valence four. \(\paraHF\cup\Gamma\) divides the plane into polygons, which we call \textbf{regions}. A \textbf{domain} is a formal linear combination of regions. In other words, a domain is an element of \(H_2(\mathbb{R}^2,\paraHF\cup\Gamma)\). 

	Let us fix a metric on the plane such that \(\paraHF\cup\Gamma\) is geodesic and the angles at the vertices are~\(\frac{\pi}{2}\). 
	We then define the \textbf{Euler measure}  \(e(D)\) of a domain \(D\) to be  \(\frac{1}{2\pi}\)~times the
	integral over \(D\) of the curvature of the metric.
	
\end{definition}

The figures in this section follow the same conventions as in~\cite{pqMod}: we use the right-hand rule to determine the orientation of domains and the normal vector fields of \(\FourPuncturedSphere\) and \(\PuncturedPlane\) are pointing into the page. Thus, the boundary of a region of multiplicity \(+1\) is oriented clockwise. 

Note that the Euler measure is additive in the sense that \(e(D+D')=e(D)+e(D')\) for any two domains \(D\) and \(D'\). In practice, one computes the Euler measure of a domain \(D\) using the following formula, which follows from the Gauss–Bonnet theorem:
\[
e(D)=\chi(D)-\tfrac{1}{4}\{\text{acute corners of }D\}+\tfrac{1}{4}\{\text{obtuse corners of }D\}
\]

\begin{definition}\label{def:connecting_domain:same_curve}
	Given an absolutely \(\delta\)-graded curve \(\gamma\), consider two intersection points \(\Lx\) and \(\Lx'\) of the lift \(\Lgamma\) with the integer lattice graph \(\paraHF\). A \textbf{connecting domain} from \(\Lx\) to \(\Lx'\) is a domain \(\varphi\in H_2(\mathbb{R}^2,\paraHF\cup\Lgamma)\) with the property 
	\[
	\partial\Big(\partial\varphi\cap \Lgamma\Big)=\Lx-\Lx'.
	\] 
\end{definition}

\begin{remark}For readers familiar with Heegaard Floer homology, it can be helpful to think of the curve \(\Lgamma\) as playing the role of a \(\beta\)-curve and \(\paraHF\) playing the role of an \(\alpha\)-curve.\end{remark}

\begin{lemma}\label{lem:delta:HFT:Euler:one_curve}
  For any connecting domain \(\varphi\) as in Definition~\ref{def:connecting_domain:same_curve}, 
  \[
  \delta(x')-\delta(x)=2e(\varphi). 
  \]
\end{lemma}

\begin{proof}
	For the domains consisting of just the single regions shown in Figure~\ref{fig:domains:xx}, the lemma follows directly from the definition of the \(\delta\)-grading: 
	The Euler measure of \(\varphi\) in these three cases is \(\frac{1}{4}\), 0 and \(-\frac{1}{4}\), respectively. 
  Let us now consider a general connecting domain \(\varphi\) from \(\Lx\) to \(\Lx'\). 
  By hypothesis, \(\partial\varphi\cap\Lgamma\) is a one-chain connecting \(\Lx\) to \(\Lx'\). Let us first assume that this one-chain corresponds to a path from \(\Lx\) to \(\Lx'\). That is, either there are no cycles in the one-chain or, in the case \(\Lx=\Lx'\), this path is the only cycle. 
  In this case the path can be written as the intersection of \(\Lgamma\) with the boundary of a connecting domain \(\psi\), which is a sum of finitely many of the basic regions in Figure~\ref{fig:domains:xx} that we have just considered. The difference \(\varphi-\psi\) is a domain whose boundary lies entirely in \(\paraHF\), so it consists entirely of square regions and hence the Euler measure vanishes. 
  
  Finally, suppose the one-chain \(\partial\varphi\cap\Lgamma\) connecting \(\Lx\) to \(\Lx'\) also has cycles. Each of them is the boundary of some domain, and we claim that its Euler measure vanishes. To see this, we can apply the previous argument with \(\Lx=\Lx'\) being some intersection point of this cycle with \(\paraHF\) the connecting path being the whole cycle.
\end{proof}

\begin{definition}\label{def:connecting_domain:two_curves}
  Let \(\bullet\) be an intersection point between two absolutely \(\delta\)-graded curves \(\gamma\) and~\(\gamma'\). Consider the lifts \(\Lgamma\) and \(\Lgamma'\) of these two curves, such that they intersect at a lift \(\Lbullet\) of the intersection point~\(\bullet\). A \textbf{connecting domain} for \(\Lbullet\) from \(\Lgamma\) to \(\Lgamma'\) is a domain \(\varphi\in H_2(\mathbb{R}^2,\paraHF\cup\Lgamma\cup \Lgamma')\) with the property 
  \[
  \partial\Big(\partial\varphi\cap\Lgamma\Big)=\Lx-\Lbullet
  \quad
  \text{and}
  \quad
  \partial\Big(\partial\varphi\cap \Lgamma'\Big)=\Lbullet-\Ly
  \quad
  \text{for some \(\Lx\in\Lgamma\cap \paraHF\) and \(\Ly\in\Lgamma'\cap \paraHF\).}
  \] 
\end{definition}

Intersection points between bigraded curves can be endowed with a $\delta$-grading~\cite[Definition~4.40]{pqMod}, and this can be easily calculated according to the following result.  

\begin{lemma}\label{lem:delta:HFT:Euler:two_curves}
  With notation as in Definition~\ref{def:connecting_domain:two_curves}, the \(\delta\)-grading of \(\bullet\) satisfies
  \[
  \delta(\bullet\co\gamma\to\gamma')=\delta(y)-\delta(x)+\tfrac{1}{2}-2e(\varphi).
  \]
\end{lemma}

\begin{figure}[t]
  \centering
  \begin{subfigure}{0.24\textwidth}
    \centering
    \(\xyI\)
    \caption{\(\delta(\bullet)=\delta(y)-\delta(x)\)}\label{fig:domains:xy:1}
  \end{subfigure}
  \begin{subfigure}{0.24\textwidth}
    \centering
    \(\xyII\)
    \caption{\(\delta(\bullet)=\delta(y)-\delta(x)+\frac 1 2\)}\label{fig:domains:xy:2}
  \end{subfigure}
  \begin{subfigure}{0.24\textwidth}
    \centering
    \(\xyIII\)
    \caption{\(\delta(\bullet)=\delta(y)-\delta(x)+1\)}\label{fig:domains:xy:3}
  \end{subfigure}
  \begin{subfigure}{0.24\textwidth}
    \centering
    \(\xyIV\)
    \caption{\(\delta(\bullet)=\delta(y)-\delta(x)+\frac 3 2\)}\label{fig:domains:xy:4}
  \end{subfigure}
  \caption{Basic connecting domains satisfying the formula \(\delta(\bullet\co {\color{red}\gamma}\to{\color{blue}\gamma}')=\delta(y)-\delta(x)+\tfrac{1}{2}-2e(\varphi)\)}\label{fig:domains:xy}
\end{figure}

\begin{observation}\label{obs:delta:reverse_order}
  The domain \(-\varphi\) is a connecting domain for the same intersection point \(\bullet\), but regarded as a generator of \(\HF(\gamma',\gamma)\). Its \(\delta\)-grading is equal to 1 minus the original \(\delta\)-grading:
  \[
  \delta(\bullet\co\gamma\rightarrow\gamma')
  =
  1-\delta(\bullet\co\gamma'\rightarrow\gamma).
  \]
\end{observation}

\begin{proof}[Proof of Lemma~\ref{lem:delta:HFT:Euler:two_curves}]
	First consider the simplest case in which the domain \(\varphi\) consists of a single region of multiplicity 1.
	Up to rotation, there are only four cases, as shown in Figure~\ref{fig:domains:xy}. 
	The lemma then follows directly from \cite[Definition~4.40]{pqMod}, since in each of those cases, the intersection point corresponds to some algebra element \(a\in\Ad\) and its \(\delta\)-grading \(\delta(a)\) is equal to \(\tfrac{1}{2}-2e(\varphi)\) \cite[Definitions~2.10 and~4.5]{pqMod}. 
  
 Now consider a general connecting domain \(\varphi\). Then near \(\Lbullet\), \(\varphi\) looks like one of the basic connecting domains \(\psi\) that we have just considered (up to adding multiples of square regions). Suppose \(\psi\) connects \(\Lx'\in\Lgamma\cap \paraHF\) to \(\Ly'\in\Lgamma'\cap \paraHF\). Then, as we have just verified,
  \[
  \delta(\bullet\co\gamma\rightarrow\gamma')
  =
  \delta(y')-\delta(x')+\tfrac{1}{2}-2e(\psi).
  \]
  Let \(\psi_x\) and \(\psi_y\) be connecting domains from \(\Lx\) to \(\Lx'\) and from \(\Ly'\) to \(\Ly\), respectively. Then, by Lemma~\ref{lem:delta:HFT:Euler:one_curve}, 
  \[
  \delta(x')-\delta(x)=2e(\psi_x)
  \quad\text{and}\quad
  \delta(y)-\delta(y')=2e(\psi_y).
  \]
  Combining all three relations, we see that 
  \[
  \delta(\bullet\co\gamma\rightarrow\gamma')
  =
  \delta(y)-\delta(x)+\tfrac{1}{2}-2e(\psi_x+\psi+\psi_y).
  \]
  By construction, \(\psi_x+\psi+\psi_y-\varphi\) is a sum of square regions and domains bounding closed components of \(\Lgamma\) and \(\Lgamma'\), so \(e(\psi_x+\psi+\psi_y)=e(\varphi)\). 
\end{proof}

\begin{definition}\label{def:symmetric_domain}
	Suppose  \(\Lgamma_i\) is an infinite connected lift to \(\PuncturedPlane\)  of some absolutely \(\delta\)-graded curve \(\gamma_i\) in \(\FourPuncturedSphere\) for \(i=1,\dots,n\), and let \(x_i\in\HF(\gamma_i,\gamma_{i+1})\) be an intersection point between \(\gamma_{i}\) and \(\gamma_{i+1}\), where we take indices modulo \(n\). 
	A \textbf{symmetric domain} for the tuples \((\Lgamma_i)_{i=1,\dots,n}\) and \((\Lx_i)_{i=1,\dots,n}\) is a domain \(\varphi\) satisfying
	\[
	\partial\Big(\partial\varphi\cap \Lgamma_{i}\Big)=\Lx_{i-1}-\Lx_i
	\] 
	where, again, indices are taken modulo \(n\).
\end{definition}

\begin{proposition}\label{prop:HFT:Euler:multiple-curves}
	For any connecting domain \(\varphi\) as in Definition~\ref{def:symmetric_domain}, 
	\[
	\sum_{i=1}^n \delta(x_i)=\tfrac{n}{2}-2 e(\varphi).
	\]
\end{proposition}
\begin{proof}
	For each \(i=1,\dots, n\), choose some intersection point \(\Ly_i\) of \(\Lgamma_i\) with \(P\). Then we can write \(\varphi\) as a sum of \(n\) connecting domains \(\varphi_i\) for \(\Lx_i\) from \(\Ly_i\) to \(\Ly_{i+1}\). By Lemma~\ref{lem:delta:HFT:Euler:two_curves}, 
	\[
	\delta(x_i)=\delta(y_{i+1})-\delta(y_i)+\tfrac{1}{2}-2e(\varphi_i)
	\]
	for \(i=1,\dots,n\). Taking the sum over all \(n\) equations, we obtain the desired identity.
\end{proof}

\begin{definition}\label{def:asymmetric_domain}
	Given two intersection points \(x,y\in\HF(\gamma,\gamma')\) between two curves \(\gamma\) 
	and \(\gamma'\), 
	a \textbf{domain} (or, asymmetric domain)
	from \(\Lx\) to \(\Ly\) is a domain \(
	\varphi\in H_2(\mathbb{R}^2,\InfConLift{\gamma}\cup\InfConLift{\gamma}')\) with
	the property 
	\[
	\partial\Big(\partial\varphi\cap \InfConLift{\gamma}\Big)=\Ly-\Lx 
	\] 
\end{definition}

\begin{corollary}\label{cor:HFT:true_domain}
	For any domain \(\varphi\) as in Definition~\ref{def:asymmetric_domain}, 
	\[\delta(y)-\delta(x)=2e(\varphi).\]
\end{corollary}

\begin{proof}
	Let us set \(n=2\), \(\gamma_1=\gamma\), \(\gamma_2=\gamma'\), and \(x_1=x\in\HF(\gamma_1,\gamma_2)\). 
	Also, let \(x_2\) be the intersection point \(y\in \HF(\gamma_1,\gamma_2)\), but regarded as an intersection point in \(\HF(\gamma_2,\gamma_1)\). 
	Then \(\varphi\) can be interpreted as a symmetric domain from \(\Lx_1\) to \(\Lx_2\). By Proposition~\ref{prop:HFT:Euler:multiple-curves}, this implies that \(\delta(x_1)+\delta(x_2)=1-2 e(\varphi)\). By Observation~\ref{obs:delta:reverse_order}, \(\delta(x_2)=1-\delta(y)\). These two identities combined prove the claim. 
\end{proof}

\begin{example}\label{exa:HFT:bigon}
	If \(\varphi\) is a bigon of multiplicity 1 as in Figure~\ref{fig:HFT:bigon}, Corollary~\ref{cor:HFT:true_domain} implies that \(\delta(y)-\delta(x)=2e(\varphi)=1\); see also~\cite[Lemma~4.41]{pqMod}.
\end{example}

\begin{figure}[t]
	\centering
	\(\HFTbigon\)
	\caption{A bigon illustrating Example~\ref{exa:HFT:bigon}; compare with Figure~\ref{fig:Kh:bigon} and~\cite[Figure~31]{pqMod}}\label{fig:HFT:bigon}
\end{figure}

\subsection{Linear curves}\label{sec:HFT:Thin:Linear}

\begin{definition}\label{def:linearity_via_derivative}
An immersed curve \(\gamma\) in \(\FourPuncturedSphere\) is called \textbf{linear} if there exists some \(\nicefrac{p}{q}\in\QPI\) such that for every open neighbourhood \(U\) of \(\nicefrac{p}{q}\) in \(\QPI\) there exists a curve $\gamma_U$ homotopic to \(\gamma\) with the property that all the slopes \(\tilde{\gamma}'_U(t)\) of the lift $\tilde{\gamma}_U$ are contained in \(U\). If there exists such a number \(\nicefrac{p}{q}\in\QPI\), it is unique, and we call it the \emph{slope} of \(\gamma\). 

If \(\Gamma\) is a collection of linear curves, we denote the set of their slopes by  \(\Slopes_{\Gamma}\). We also say that a collection of curves \(\Gamma\) is supported on a slope if it contains a curve of that slope. 
\end{definition}

By Theorem~\ref{thm:geography_of_HFT}, the invariant \(\HFT(T)\) of any Conway tangle \(T\) consists of rational and special curves, which are linear. Thus, for the remainder of this section we restrict our attention to linear curves.

The slope of the mirror of a linear curve \(\gamma\) (see Definition~\ref{def:mirror_curve}) is equal to the mirror of the slope of \(\gamma\) (see Definition~\ref{def:mirror_slopes}). So by Lemma~\ref{lem:mirroring:HFT}, mirroring operation commutes with taking the curve invariant $\HFT(-)$ and its slopes:
\[\Slopes_{\HFT(T^*)}=\Slopes_{\mirror(\HFT(T))}=\Slopes^{\mirror}_{\HFT(T)}\]

%

\begin{deflemma}\label{deflem:delta:linear_curve}
  Let \(\gamma\) be a linear curve of slope \(s\in\QPI\). Then unless \(s=0\), the images of all intersection points of \(\Lgamma\) with the horizontal lines of \(\paraHF\) have the same \(\delta\)-grading \(\delH\coloneqq\delH(\gamma)\), and unless \(s=\infty\), the images of all intersection points of \(\Lgamma\) with the vertical lines of \(\paraHF\) have the same \(\delta\)-grading \(\delV\coloneqq\delV(\gamma)\). Moreover, 
  \[
  \delV=
  \begin{cases}
  \delH-\frac{1}{2} &\text{ if \(0<s<\infty\)} \\
  \delH+\frac{1}{2} &\text{ if \(\infty<s<0\)}
  \end{cases}
  \]
\end{deflemma}
\begin{proof}
  If \(s\neq 0\), any two horizontal intersection points are connected via a rectangular domain, and so are any two vertical intersection points in the case \(s\neq \infty\). This proves the first two statements. For the third, suppose first that \(0<s<\infty\). Then we can connect any vertical intersection point to a horizontal intersection point by a triangular connecting domain of multiplicity \(+1\). The Euler measure is equal \(+\frac{1}{4}\) for any such pair of intersection points, so \(\delH-\delV=\frac{1}{2}\) by Lemma~\ref{lem:delta:HFT:Euler:one_curve}. For \(\infty<s<0\), the argument is the same except that the order of the intersection points is reversed; so in this case \(\delV-\delH=\frac{1}{2}\). 
\end{proof}

\begin{definition}\label{def:delta:HF:linear_curves}
  Given two linear curves \(\gamma\) and \(\gamma'\) of the same slope \(s\in\QPI\), we define
  \[
  \delta(\gamma,\gamma')
  =
  \begin{cases*}
  \delH(\gamma')-\delH(\gamma)
  &
  if \(s\neq0\)
  \\
  \delV(\gamma')-\delV(\gamma)
  &
  if \(s\neq\infty\)
  \end{cases*}
  \]
  This is well-defined by Lemma~\ref{deflem:delta:linear_curve}. 
  The two curves are said to have the same \(\delta\)-grading if \(\delta(\gamma,\gamma')=0\).
\end{definition}

\begin{deflemma}\label{deflem:delta:HF:linear_curves}
  Given two linear curves \(\gamma\) and \(\gamma'\) of different slopes \(s,s'\in\QPI\), the Lagrangian intersection theory \(\HF(\gamma,\gamma')\) is supported in a single \(\delta\)-grading, which is equal to 
  \[
  \delta(\gamma,\gamma')
  \coloneqq
  \begin{cases*}
  \delH(\gamma')-\delV(\gamma)+\tfrac{1}{2} & if \(s\in(\infty,s')\) for \(s'\in(0,\infty]\), or \(s\in(s',\infty)\) for \(s'\in[\infty,0)\)
  \\
  \delV(\gamma')-\delH(\gamma)+\tfrac{1}{2} &  if \(s\in(s',0)\) for \(s'\in[0,\infty)\), or \(s\in(0,s')\) for \(s'\in(\infty,0]\)
  \end{cases*}
  \]
  using the convention that \((\infty,\infty)\coloneqq\QPI\smallsetminus\{\infty\}\) and \((0,0)\coloneqq\QPI\smallsetminus\{0\}\).
\end{deflemma}

\begin{figure}[t]
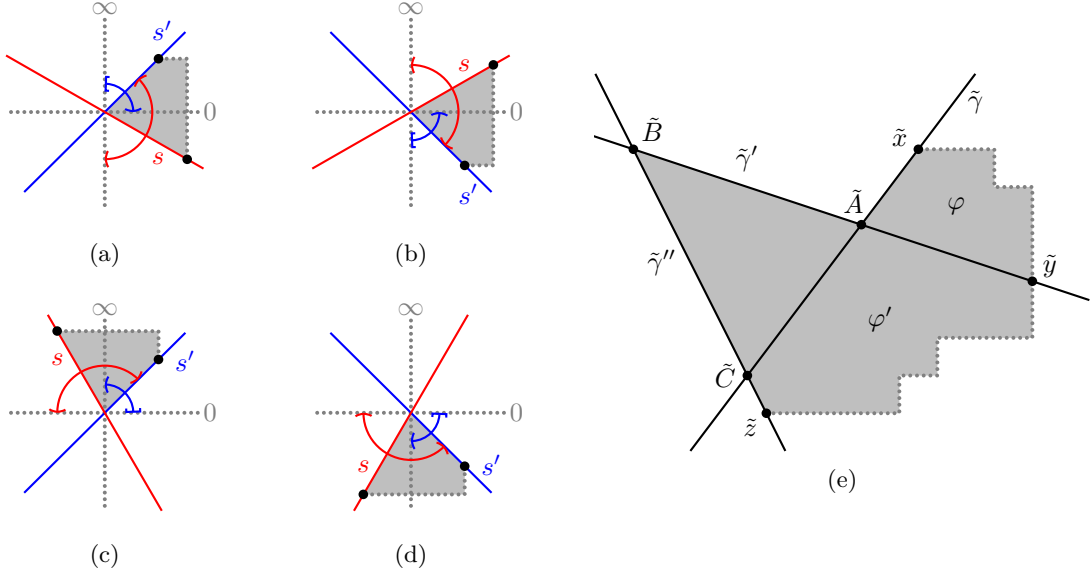

  \centering
  \begin{minipage}{0.5\textwidth}
  	\centering
	  \begin{subfigure}{0.49\textwidth}
	    \centering
	    \(\HFTintervalI\)
	    \caption{}\label{fig:LinearIsThin:1}
	  \end{subfigure}
	  \begin{subfigure}{0.49\textwidth}
	    \centering
	    \(\HFTintervalII\)
	    \caption{}\label{fig:LinearIsThin:2}
	  \end{subfigure}
  \bigskip
  \\
	  \begin{subfigure}{0.49\textwidth}
	    \centering
	    \(\HFTintervalIII\)
	    \caption{}\label{fig:LinearIsThin:3}
	  \end{subfigure}
	  \begin{subfigure}{0.49\textwidth}
	    \centering
	    \(\HFTintervalIV\)
	    \caption{}\label{fig:LinearIsThin:4}
	  \end{subfigure}
  \end{minipage}
	\begin{subfigure}{0.45\textwidth}
		\centering
		\(\HFTtransitivity\)
		\caption{}\label{fig:deltaDifference:HFT:transitivity}
	\end{subfigure}

  \caption{Some illustrations for the proofs of Lemma~\ref{deflem:delta:HF:linear_curves} (a--d) and  Theorem~\ref{thm:deltaDifference:HFT:transitivity} (e)}\label{fig:LinearIsThin}
\end{figure}

\begin{proof}
  Fix an intersection point of \(\Lgamma\) and \(\Lgamma'\). 
  In the first case, we can find a rectangular connecting domain for this intersection point between \(\Lgamma\) and \(\Lgamma'\) of multiplicity \(\pm1\) which starts at a vertical intersection point of \(\Lgamma\) and ends at a horizontal intersection point of \(\Lgamma'\); see Figures~\ref{fig:LinearIsThin:1} and~\ref{fig:LinearIsThin:2} for an illustration. We then use the formula from Lemma~\ref{lem:delta:HFT:Euler:two_curves}. We can argue similarly in the second case, which is illustrated in Figures~\ref{fig:LinearIsThin:3} and~\ref{fig:LinearIsThin:4}. One can easily check that any \((s,s')\in\QPI\times\QPI\) with \(s\neq s'\) belongs to at least one of these two cases. Finally, the formula for the \(\delta\)-grading is obviously independent of the particular intersection point that we considered. 
\end{proof}

\begin{corollary}[Symmetry of \(\delta\)]\label{cor:deltaDifference:HFT:commutativity}
  For any two linear curves \(\gamma\) and \(\gamma'\),
  \[
  \delta(\gamma,\gamma')+\delta(\gamma',\gamma)=
  \begin{cases*}
  0 & if \(s(\gamma)=s(\gamma')\)\\
  1 & if \(s(\gamma)\neq s(\gamma')\)
  \end{cases*}
  \]
\end{corollary} 

\begin{proof}
  This follows immediately from Observation~\ref{obs:delta:reverse_order} and Definition~\ref{def:delta:HF:linear_curves}.
\end{proof}

In the proof of Theorem~\ref{thm:reduction:HF} below, we will relate the grading function \(\gr\) from Section~\ref{sec:main} to the negative of \(\delta\). 
This explains the difference in sign between the symmetry of \(\delta\) compared to the \ref{eq:symmetry} of \(\gr\).  
Similarly, while the transitivity of \(\gr\) holds for \emph{increasing} triples of slopes, we establish transitivity of \(\delta\) in terms of \emph{decreasing} triples:

\begin{theorem}[Transitivity of \(\delta\)]\label{thm:deltaDifference:HFT:transitivity}
  For any triple \((\gamma,\gamma',\gamma'')\) of linear curves in \(\FourPuncturedSphere\) such that \((s(\gamma),s(\gamma'),s(\gamma''))\) is decreasing,
  \[
  \delta(\gamma,\gamma')+\delta(\gamma',\gamma'')=\delta(\gamma,\gamma'').
  \]
\end{theorem}

\begin{proof}
  If \(s(\gamma)=s(\gamma')\) or \(s(\gamma')=s(\gamma'')\), or both, this follows from Definition~\ref{def:delta:HF:linear_curves} and Lemma~\ref{deflem:delta:HF:linear_curves}. 
  So let us suppose the curves have pairwise different slopes. 
  Let us consider some infinite connected lifts \(\Lgamma\), \(\Lgamma'\), and \(\Lgamma''\) of \(\gamma\), \(\gamma'\), and \(\gamma''\), respectively. These lifts intersect in three points \(\LA\in\Lgamma\cap\Lgamma'\), \(\LB\in\Lgamma'\cap\Lgamma''\), and \(\LC\in\Lgamma''\cap\Lgamma\). 
  Consider a connecting domain \(\varphi\) for \(\LA\) starting at some point \(\Lx\in \Lgamma\cap \paraHF\) and ending at some point \(\Ly\in \Lgamma'\cap \paraHF\). Then choose a connecting domain \(\varphi'\) for \(\LB\) which starts at \(\Ly\) and ends at some point \(\Lz\in \Lgamma''\cap \paraHF\). 
  Since the slopes of these curves form a decreasing triple, the triangle \(\LDelta\) has multiplicity \(-1\) when the vertices are ordered counter-clockwise, as illustrated in Figure~\ref{fig:deltaDifference:HFT:transitivity}; therefore, \(e(\LDelta)=-\frac{1}{4}\). 
  So \(\varphi+\varphi'+\LDelta\) is a connecting domain for \(\LC\).
  Hence
  \begin{align*}
  \delta(\gamma,\gamma')+\delta(\gamma',\gamma'')
  &=
  \Big(\delta(y)-\delta(x)-2e(\varphi)+\tfrac{1}{2}\Big)+
  \Big(\delta(z)-\delta(y)-2e(\varphi')+\tfrac{1}{2}\Big)
  \\
  &=
  \delta(z)-\delta(x)-2e(\varphi+\varphi'+\LDelta)+\tfrac{1}{2}
  =
  \delta(\gamma,\gamma'').\qedhere
  \end{align*}
\end{proof}

\begin{lemma}\label{lem:HFT:epsilon_criterion}
	Given two curves with local systems \(\gamma\) and \(\gamma'\) in \(\FourPuncturedSphere\) that share the same slope \(s\in\QPI\), the vector space \(\HF(\gamma,\gamma')\) is either zero or it is supported in two consecutive \(\delta\)-gradings, namely \(\delta(\gamma,\gamma')\) and \(\delta(\gamma,\gamma')+1\). 
	Moreover, if \(\gamma\) is rational and \(\gamma'\) is special, or vice versa, then \(\HF(\gamma,\gamma')=0\). 
	Finally, for two rational curves with trivial local systems \(\HF(\Rational(s),\Rational(s))\neq 0\), as well as for special curves 
	\(
	\HF(\Special_n(s;\TEi,\TEj),\Special_m(s;\TEi,\TEj))\neq 0
	\) (given any \(n,m>0\) and pair \((\TEi,\TEj)\) of tangle ends compatible with the slope \(s\)). 
\end{lemma}

\begin{proof}
	Let us consider each combination of rational and special curves separately. Clearly, the Lagrangian Floer homology of a special and a rational curve vanishes. To compute the Lagrangian Floer homology of two rational curves with local systems of the same slope, we can apply~\cite[Theorem 4.45]{pqMod} to verify the first statement in this case. If the two local systems are trivial, then their Lagrangian Floer homology does not vanishes by \cite[Lemma~4.51]{pqMod}.
	Let us now turn to the final case that \(\gamma\) and \(\gamma'\) are special. Then, if they wrap around different tangle ends, their Lagrangian Floer homology vanishes. 
	So let us consider \(\gamma=\Special_n(s;\TEi,\TEj)\) and \(\gamma'=\Special_m(s;\TEi,\TEj)\). 
	To justify the support of \(\HF(\gamma,\gamma')\) for \(n\neq m\), we can argue as follows.
	Let us consider a straight line of slope \(s\) through two integer lattice points \(\tilde{\TEi}\) and  \(\tilde{\TEj}\) corresponding to the tangle ends \(\TEi\) and \(\TEj\), respectively. 
	After some homotopy, we may assume that \(\gamma\) and \(\gamma'\) are contained in a small neighbourhood of the embedded arc that is the image of this straight line in \(\FourPuncturedSphere\). 
	Moreover, we may assume that the slopes of any lifts of \(\gamma\) and \(\gamma'\) are contained in \((s-\varepsilon,s+\varepsilon)\) for some small \(\varepsilon>0\). 
	Now given some intersection point \(\bullet\in\HF(\gamma,\gamma')\), choose lifts \(\Lgamma\) and \(\Lgamma'\) that intersect transversely in some lift \(\Lbullet\) of \(\bullet\). Let \(t\) and \(t'\) be the slopes of \(\Lgamma\) and \(\Lgamma'\) at \(\Lbullet\), respectively. Then
	\[
	\delta(\bullet)=
	\begin{cases*}
	\delta(\gamma,\gamma')
	&
	if \(s-\varepsilon<t'<t<s+\varepsilon\)
	\\
	\delta(\gamma,\gamma')+1
	&
	if \(s-\varepsilon<t<t'<s+\varepsilon\)
	\end{cases*}
	\]
	which can be seen by applying Lemma~\ref{lem:delta:HFT:Euler:two_curves} to thin triangular domains enclosed  on two sides by \(\Lgamma\) and \(\Lgamma'\) and on the third side by either only the vertical or only the horizontal arcs in \(\paraHF\) (see also Figure~\ref{fig:domains:xy:1}).
\end{proof}

There exist local systems \(X\) and \(Y\) for which \(\HF(\Rational_X(s),\Rational_Y(s))=0\). For example, take \(X\) to be a permutation matrix of order \(n\) and let \(Y\) be the companion matrix of the polynomial \(f(x)=x^n+x+1\). Then \(\ker(f(X))=\ker(X)=0\), so by~\cite[Theorem 4.45 and Lemma~4.51]{pqMod}, \(\dim\HF(\Rational_X(s),\Rational_Y(s))=2\cdot\dim\ker(f(X))=0\). But we still have some control over local systems:

\begin{definition}\label{def:inhibited}
	We say two local systems \(X\) and \(Y\) are \textbf{complementary}, if \(\HF(\Rational_X(s),\Rational_Y(s))\) vanishes, where \(s\) is some slope. We call a local system \textbf{inhibited} if it is complementary to the (trivial) one-dimensional local system. 
	Similarly, we call a rational curve inhibited if its local system is inhibited. 
	We call a collection of curves \(\Gamma\) \(s\)-inhibited if any rational component of slope \(s\) is inhibited. 
\end{definition}

\begin{lemma}\label{lem:SomeControlOverLocalSystems}
	If \(X\) and \(Y\) are complementary local systems, then one of \(X\) and \(Y\) is inhibited.
\end{lemma}
\begin{proof}
	We show the contrapositive: If \(\HF(\Rational_X(s),\Rational(s))\) and \(\HF(\Rational_Y(s),\Rational(s))\) are non-zero, then so is \(\HF(\Rational_X(s),\Rational_Y(s))\).
	Let us first verify this in the case that \(X\) and \(Y\) are companion matrices \(X_f\) and \(X_g\) of two polynomials \(f,g\in\fieldTwoElements[x]\), respectively. By~\cite[Theorem~4.45 and Lemma~4.51]{pqMod}, \(\HF(\Rational_{X_f}(s),\Rational_{X_g}(s))\) is zero if and only if the matrix \(f(X_g)\) is invertible. 
	Similarly \(\HF(\Rational_{X_f}(s),\Rational(s))\) is zero if and only if \(f(1)\) is invertible, ie equal to \(1\); the same is true of course for \(g\). 
	So we need to show that \(f(1)=0=g(1)\) implies that \(f(X_g)\) is not invertible. 
	This follows from two observations: 
	First, note that \(\det(X_g+1)=g(1)=0\), because \(g\) is the characteristic polynomial of \(X_g\), 
	so \((X_g+1)\) is not invertible. 
	Second, \(f(1)=0\) implies that there exists some polynomial \(\tilde{f}\in\fieldTwoElements[x]\) such that \(f(x)=(x+1)\tilde{f}(x)\). So if \((X_g+1)\) is not invertible, then neither is \(f(X_g)\). 
	
	In the general case, \(\Rational_X\) and \(\Rational_Y\) are equivalent to disjoint unions of rational curves \(\Rational_{X_i}\) and \(\Rational_{Y_j}\), respectively, whose local systems are all companion matrices. Suppose there exist \(i,j\) such that \(X_i\) and \(Y_j\) are not inhibited. Then, by the above, \(X_i\) and \(Y_j\) are not complementary, so \(X\) and \(Y\) are not complementary.
\end{proof}

\begin{remark}\label{rem:inhibited}
	We would like to show that no rational component of \(\HFT(T)\) is inhibited for any Conway tangle \(T\). 
	However, the only known restriction is that after combining the local systems of all parallel rational components, the new local system should be conjugate to its inverse \cite[Theorem~0.10]{pqSym}. There are many such local systems: For example, for any invertible matrix \(X\), the diagonal block matrix \(Y\) with blocks \(X\) and \(X^{-1}\) is conjugate to its inverse. If we choose \(X\) to be inhibited, then so is \(Y\). 
\end{remark}

\begin{definition}\label{def:HF:exceptional}
	We say a multicurve \(\Gamma\) is \textbf{exceptional} if \(\Slopes_{\Gamma}=\{s,s'\}\) for two distinct slopes \(s,s'\in\QPI\) and there exists a constant \(c\neq0,1\) such that \(\delta(\gamma,\gamma')=c\) for all components \(\gamma,\gamma'\in\Gamma\) in slopes \(s,s'\), respectively. We say that a tangle \(T\) is \textbf{Heegaard Floer exceptional} if \(\HFT(T)\) is exceptional. 
\end{definition}

A tangle that is Heegaard Floer exceptional is described in Example~\ref{exp:whitehead}.

\subsection{Heegaard Floer thin fillings}\label{sec:HFT:Thin:results}
In this subsection, \(G\) is either \(\Z\) or \(\ZZ\). Define 
\[
\CurvesHF
\coloneqq
\{\HFT(T)\mid\text{Conway tangles }T\}
\]
In the following, we will make implicit use of the following properties that the curves in \(\CurvesHF\) are known to satisfy: 
Each multicurve \(\Gamma\in\CurvesHF\) consists of linear components only (Theorem~\ref{thm:geography_of_HFT}). Moreover, special components come in conjugate pairs of curves that are supported in identical \(\delta\)-gradings (Theorem~\ref{thm:Conjugation}). Finally, \(\HF(\Gamma_1,\Gamma_2)\neq 0\) for each \(\Gamma_1,\Gamma_2\in\CurvesHF\), because of Theorem~\ref{thm:GlueingTheorem:HFT} and the fact that knot Floer homology does not vanish. In particular, if \(\Slopes_\Gamma=\{s\}\) for some \(\Gamma\in\CurvesHF\), then \(\Gamma\) contains some rational component whose local system is not inhibited. 
Let \(\CurvesHFwb\subseteq\CurvesHF\) be the subset of \textbf{well-behaved} multicurves defined by
\[
\CurvesHFwb
\coloneqq
\{\Gamma\in\CurvesHF\mid \Gamma\text{ does not contain any inhibited rational component}\}
\]

Given two multicurves \(\Gamma\) and \(\Gamma'\) and a slope \(s\in\Slopes_\Gamma\cap\Slopes_{\Gamma'}\), the following two local conditions will be relevant:
\begin{enumerate}
	\myitem[\(\mathrm{R}\)] \label{local:HF:s-rational} At least one of \(\Gamma\) and \(\Gamma'\) is \(s\)-rational, ie it only contains rational components of slope \(s\).
	\myitem[\(\mathrm{R}_\star\)] \label{local:HF:complementary} The local systems of any two rational components of \(\Gamma\) and \(\Gamma'\) of slope \(s\) are complementary. 
\end{enumerate}
These are the local conditions for Heegaard Floer theory mentioned in Theorems~\ref{thm:gluing:ALink:intro} and~\ref{thm:gluing:Thin:intro}. 
Note that \ref{local:HF:complementary} is vacuously satisfied if any two rational components of \(\Gamma\) and \(\Gamma'\) have different slopes.
For instance, this is true if \(\Gamma=\mirror(\HFT(T_1))\), \(\Gamma'=\HFT(T_2)\), and \(T_1\cup T_2\) is a knot, see Lemma~\ref{lem:HFT_detects_connectivity}.

\begin{definition} 
	A \(\delta\)-graded vector space is \textbf{\(\bm{\Z}\)-thin} if it is thin and  \textbf{\(\bm{\ZZ}\)-thin} if it is supported in at most one \(\delta\)-grading modulo 2. In particular, the 0-dimensional vector space is \(G\)-thin for both \(G=\Z\) and \(G=\ZZ\).
	Given a relatively \(\delta\)-graded multicurve \(\Gamma\), define the space of \textbf{\(\bm{G}\)-thin rational fillings} of \(\Gamma\) by
	\[
	\ThinG(\Gamma)\coloneqq\Bigl\{s\in\QPI\Bigm| \HF(\Rational(s),\Gamma)\text{ is \(G\)-thin}\Bigr\}.
	\] 
	Suppose \(T\) is a Conway tangle in a three-ball. Then, by Theorem~\ref{thm:GlueingTheorem:HFT}, writing 
	\[
	\ThinHF(T)=\ThinZ(\HFT(T))
	\quad\text{ and }\quad
	\ALinkHF(T)=\ThinZZ(\HFT(T))
	\] 
	recovers the definitions of \(\ThinHF(T)\) and \(\ALinkHF(T)\) from the introduction.
\end{definition}

\begin{remark}\label{rem:mirroring:ThinG:HFT}
	Since by Lemma~\ref{lem:mirroring:HFT} the tangle invariant \(\HFT\) behaves in a natural way under mirroring,
	\(\ThinHF(T^*)=\mirrorThinHF(T)\)
	and
	\(\ALinkHF(T^*)=\mirrorALinkHF(T)\)
	 for any Conway tangle \(T\). 
\end{remark}

The following is the main technical result which links Section~\ref{sec:main} to the present discussion about \(\HFT\), by relating elements of \(\CurvesHF\) to line sets in Section~\ref{sec:main}. 
Recall that a line set is simply a finite collection of elements of \(\Curves=\QPI\times\,G\times\{0,1\}\). 
We will write \(\mathcal{P}_\mathrm{finite}(\Curves)\) for the set of all line sets.

\begin{theorem}\label{thm:reduction:HF}
	There exist a map \(\Phi\co\CurvesHF\rightarrow\mathcal{P}_\mathrm{finite}(\Curves)\) and a map \(\gr\co \Curves^2\rightarrow G\)	satisfying the \ref{eq:symmetry}, \ref{eq:transitivity}, and \ref{eq:linearity} properties as in Section~\ref{sec:main} such that for any \(\Gamma\in\CurvesHF\), the following holds:
	\begin{enumerate}[label=(\roman*)]
		\item \label{enu:reduction:HF:slopes} \(\Slopes_{\Gamma}=\Slopes_{\Phi(\Gamma)}\)
		\item \label{enu:reduction:HF:spaces}
		\(
		\ThinG(\Gamma)=\ThinG(\Phi(\Gamma)).
		\)
		\item \label{enu:reduction:HF:trivial} \(\Phi(\Gamma)\) is non-trivial. 
		\item \label{enu:reduction:HF:exceptional} \(\Phi(\Gamma)\) is exceptional if and only if \(\Gamma\) is exceptional. 
	\end{enumerate}
	Moreover, if \(\Gamma,\Gamma'\in\CurvesHFwb\),
	\begin{enumerate}[label=(\roman*)]\addtocounter{enumi}{4}
		\item \label{enu:reduction:HF:s-rational} For any slope \(s\in\QPI\), \(\Gamma\) is \(s\)-rational if and only if \(\Phi(\Gamma)\) is \(s\)-rational. 
		\item \label{enu:reduction:HF:pairs:Thin}  \(\HF(\Gamma,\Gamma')\) is \(G\)-thin if and only if the pair \((\Phi(\Gamma), \Phi(\Gamma'))\) is \(G\)-thin. 
	\end{enumerate}
\end{theorem}

\begin{proof}
Given \(c\in\Curves\), let \(\gamma(c)\) be an absolutely \(\delta\)-graded linear curve of slope \(s(c)\) such that \(\delH(\gamma(c))=\gr(c)\) if \(s(c)\neq0\) and \(\delV(\gamma(c))=\gr(c)-\tfrac{1}{2}\) if \(s(c)=0\). 
(Whether this curve is rational or special has no bearing on what follows.)
Now define \(\gr\co \Curves^2\rightarrow G\) by setting for each \(c,c' \in \Curves\)
\[
\gr(c,c')\coloneqq -\delta(\gamma(c'),\gamma(c))
\]
(Note that the order of \(c\) and \(c'\) is reversed.)
Then, by Corollary~\ref{cor:deltaDifference:HFT:commutativity}, \ref{eq:symmetry} of \(\gr\) holds, and by Theorem~\ref{thm:deltaDifference:HFT:transitivity}, so does \ref{eq:transitivity} of \(\gr\). 
Moreover, \ref{eq:linearity} of \(\gr\) follows from the definition. 

Before we define the map \(\Phi\), let us lift the \(\delta\)-grading of all curves in \(\CurvesHF\) to an absolute \(\delta\)-grading such that for each component \(\gamma\) of any element \(\Gamma=\HFT(T)\in\CurvesHF\), we have \(\delH(\gamma)\in \Z\) and \(\delV(\gamma)\in \Z+\tfrac 1 2\). 
For rational tangles \(T=Q_s\), this is clearly possible. To see that this is possible for arbitrary Conway tangles \(T\), we choose a slope \(s\not\in\Slopes_{\Gamma}\). Then \(\HF(\Rational(s),\Gamma)\) computes the (potentially once stabilized) knot Floer homology of \(Q_{-s}\cup T\), which, up to an overall grading shift, is supported in integer \(\delta\)-gradings. By our choice of slope~\(s\), \(\Rational(s)\) intersects each component of \(\Gamma\) non-trivially, so we conclude with Lemma~\ref{deflem:delta:HF:linear_curves}. 

Now, we are ready to define the map \(\Phi\). Given some absolutely \(\delta\)-graded rational or special curve \(\gamma\) of slope \(s\), let \(c=c(\gamma)\in\Curves\) be the line defined by \(s(c)=s\), \(\gr(c)=\delH(\gamma)\) if \(s\neq0\) and \(\delV(\gamma)+\tfrac{1}{2}\) if \(s(c)=0\), and \(\varepsilon(c)=1\) if \(\gamma\) is special or rational with inhibited local system and \(\varepsilon(c)=0\) otherwise.  
Then, given some \(\Gamma=\{\gamma_i\}_i\in\CurvesHF\), define \(\Phi(\Gamma)\) as the set corresponding to the multiset \(\{c(\gamma_i)\}_{i}\).

Clearly, properties \ref{enu:reduction:HF:slopes}--\ref{enu:reduction:HF:exceptional} hold by construction. 
Moreover, the only rational components that \(\Phi\) sends to special lines are those with inhibited local systems, so \ref{enu:reduction:HF:s-rational} follows. 
Suppose \(\Gamma,\Gamma'\in\CurvesHFwb\). 
Then by Lemma~\ref{lem:SomeControlOverLocalSystems}, these multicurves do not contain any rational components that are complementary to each other. 
Thus, if \(\gamma\in\Gamma\) and \(\gamma'\in\Gamma'\) are two components of the same slope, \(\HF(\gamma,\gamma')\) is \(G\)-thin if and only if \((c(\gamma),c(\gamma'))\) is \(G\)-thin by Lemma~\ref{lem:HFT:epsilon_criterion}. 
Together with Lemma~\ref{deflem:delta:HF:linear_curves}, this proves~\ref{enu:reduction:HF:pairs:Thin}.
\end{proof}

We now establish the results concerning Heegaard Floer theory from the introduction. We restate these here for clarity. If we ignore the technical issue of inhibited curves and restrict to well-behaved multicurves, Theorems~\ref{thm:gluing:ALink:intro} and~\ref{thm:gluing:Thin:intro}, specialized to the Heegaard Floer setting, follow immediately from the results of Section~\ref{sec:main} and Theorem~\ref{thm:reduction:HF}. The proof in the general case requires a more careful analysis of the arguments from Section~\ref{sec:main}. 

\begin{theorem}[Characterization of Heegaard Floer \(G\)-thin filling spaces; Theorems~\ref{thm:charactisation:ALink:intro} and~\ref{thm:charactisation:Thin:intro}]\label{thm:charactisation:Thin:HFT}
	For any Conway tangle \(T\), \(\ALinkHF(T)\) is either empty, a single point or an interval in \(\QPI\). 
	Furthermore, \(\ThinHF(T)\) is either empty, a single point, two distinct points or an interval in \(\QPI\).   
\end{theorem}

\begin{proof}
	This follows from Theorem~\ref{thm:charactisation:ThinG:main} and parts~\ref{enu:reduction:HF:spaces} and~\ref{enu:reduction:HF:trivial} of Theorem~\ref{thm:reduction:HF}.
\end{proof}

\begin{proposition}[Proposition~\ref{prop:spaces_coincide:intro}]\label{prop:spaces_coincide:HF}
	If \(\ThinHF(T)\) is an interval, \(\ThinHF(T)=\ALinkHF(T)\).
\end{proposition}

\begin{proof}
	This follows from Theorem~\ref{thm:reduction:HF}~\ref{enu:reduction:HF:spaces} in conjunction with Proposition~\ref{prop:spaces_coincide:main}.
\end{proof}

In the following, let \(T_1\) and \(T_2\) be two Conway tangles and write \(\Gamma_1\coloneqq\mirror(\HFT(T_1))\) as well as \(\Gamma_2\coloneqq\HFT(T_2)\). 

\begin{theorem}[A-link Gluing Theorem; Theorem~\ref{thm:gluing:ALink:intro}]\label{thm:glueing:ALink:HF}
	\(T_1\cup T_2\) is a Heegaard Floer A-link if and only if
	\begin{enumerate}
		\item 	\(
		\mirrorALinkHF(T_1)
		\cup 
		\ALinkHF(T_2)
		=
		\QPI
		\);
		and
		\item for every slope
		\( s\in 
		\mirrorBdryALinkHF(T_1)
		\cap
		\BdryALinkHF(T_2)
		\), \(\Gamma_1\) and \(\Gamma_2\) satisfy \ref{local:HF:s-rational}  and \ref{local:HF:complementary}.
	\end{enumerate}
\end{theorem}

\begin{theorem}[Thin Gluing Theorem; Theorem~\ref{thm:gluing:Thin:intro}]\label{thm:glueing:Thin:HF}
	Suppose at least one of  \(T_1\) and \(T_2\) is not Heegaard Floer exceptional. 
	Then \(T_1\cup T_2\) is Heegaard Floer thin if and only if
	\begin{enumerate}
		\item 	\(
		\mirrorThinHF(T_1)
		\cup 
		\ThinHF(T_2)
		=
		\QPI
		\);
		and
		\item for every slope 
		\(s\in
		\mirrorBdryThinHF(T_1)
		\cap
		\BdryThinHF(T_2)
		\), \(\Gamma_1\) and \(\Gamma_2\) satisfy \ref{local:HF:s-rational}  and \ref{local:HF:complementary}.
	\end{enumerate}
\end{theorem}

The assumption in Theorem~\ref{thm:glueing:Thin:HF} that \(T_1\) and \(T_2\) not be both exceptional is meaningful; see Example~\ref{exp:whitehead}.

\begin{remark}
  The property \ref{local:HF:complementary} can be dropped from Theorems~\ref{thm:glueing:ALink:HF} and~\ref{thm:glueing:Thin:HF} if we restrict to well-behaved curves. Indeed, if \(T_1\cup T_2\) is Heegaard Floer \(G\)-thin, then clearly, \(\Gamma_1\) and \(\Gamma_2\) satisfy~\ref{local:HF:complementary} for any slope 
	\(s\in\QPI\). 
	Conversely, suppose \(\Gamma_1\) and \(\Gamma_2\) are well-behaved. If 	
	\(
	\ThinG(\Gamma_1)
	\cup 
	\ThinG(\Gamma_2)
	=
	\QPI
	\) 
	then for any slope \(s\in\QPI\), at least one of \(\Gamma_1\) and \(\Gamma_2\) does not contain any rational component of slope \(s\), so \(\Gamma_1\) and \(\Gamma_2\) vacuously satisfy~\ref{local:HF:complementary} for all \(s\in\QPI\). 
\end{remark}

\begin{proof}[Proofs of Theorems~\ref{thm:glueing:ALink:HF} and~\ref{thm:glueing:Thin:HF} for well-behaved multicurves.]
	Suppose \(\Gamma_1,\Gamma_2\in\CurvesHFwb\). Let \(C_i=\Phi(\Gamma_i)\) for \(i=1,2\). By the previous remark, we may ignore property~\ref{local:HF:complementary} for this proof.
	By Theorem~\ref{thm:GlueingTheorem:HFT}, \(T_1\cup T_2\) is an A-link if and only if \(\HF(\Gamma_1,\Gamma_2)\) is \(\ZZ\)-thin. 
	By Theorem~\ref{thm:reduction:HF}~\ref{enu:reduction:HF:pairs:Thin}, 
	the latter is equivalent to \((C_1,C_2)\) being \(\ZZ\)-thin. 
	Note that since \(G=\ZZ\), neither \(C_1\) nor \(C_2\) are exceptional. Therefore, by Theorems~\ref{thm:glueing:ThinG:main} and~\ref{thm:reduction:HF}~\ref{enu:reduction:HF:spaces}, \ref{enu:reduction:HF:trivial}, and~\ref{enu:reduction:HF:s-rational}, this is equivalent to \(\ThinZZ(\Gamma_1)\cup\ThinZZ(\Gamma_2)=\QPI\) and for all \(s\in\BdryThinZZ(\Gamma_1)\cap\BdryThinZZ(\Gamma_2)\), at least one of \(\Gamma_1\) and \(\Gamma_2\) is \(s\)-rational. 
	This is equivalent to the right hand side of Theorem~\ref{thm:glueing:ALink:HF}. 
	
	The proof of Theorem~\ref{thm:glueing:Thin:HF} is analogous to the above, noting that at most one of \(C_1\) and \(C_2\) are exceptional by Theorem~\ref{thm:reduction:HF}~\ref{enu:reduction:HF:exceptional} and the additional hypothesis in Theorem~\ref{thm:glueing:Thin:HF}.
\end{proof}

\begin{corollary}[Corollaries~\ref{cor:one_direction:ALink:intro} and~\ref{cor:one_direction:Thin:intro}]\label{cor:one_direction:ThinG:HF}
	For any Conway tangles \(T_1\) and \(T_2\),
	\begin{align*}
		\mirrorIntALinkHF(T_1)
		\cup 
		\IntALinkHF(T_2)
		&=
		\QPI
		\Rightarrow
		\text{\(L\) is a Heegaard Floer A-link; and}
		\\
		\mirrorIntThinHF(T_1)
		\cup 
		\IntThinHF(T_2)
		&=
		\QPI
		\Rightarrow
		\text{\(L\) is Heegaard Floer thin.}
	\end{align*}
\end{corollary}

\begin{proof}[Proof of Corollary~\ref{cor:one_direction:ThinG:HF} for well-behaved multicurves.]
	Suppose that \(\Gamma_1,\Gamma_2\in\CurvesHFwb\). Let \(C_i=\Phi(\Gamma_i)\) for \(i=1,2\) as before. By Theorem~\ref{thm:reduction:HF}~\ref{enu:reduction:HF:trivial}, \(C_1\) and \(C_2\) are non-trivial. 
	By Theorem~\ref{thm:reduction:HF}~\ref{enu:reduction:HF:spaces}, the hypotheses imply that 
	\(
	\IntThinG(C_1)
	\cup 
	\IntThinG(C_2)
	=
	\QPI
	\), 
	and so by Corollary~\ref{cor:one_direction:main}, the pair \((C_1,C_2)\) is \(G\)-thin. 
	So by Theorem~\ref{thm:reduction:HF}~\ref{enu:reduction:HF:pairs:Thin}, \(\HF(\Gamma_1,\Gamma_2)\) is \(G\)-thin. Now conclude with Theorem~\ref{thm:GlueingTheorem:HFT}.
\end{proof}

\setcounter{casecount}{0}
\newcommand{\casei}[1]{%
	\pagebreak[2]\medskip\noindent%
	\textbf{%
		Case~\thecasecount: \(#1\).\stepcounter{casecount}%
	}%
	\nopagebreak\relax%
}

Before proving Theorems~\ref{thm:glueing:ALink:HF} and~\ref{thm:glueing:Thin:HF} and Corollary~\ref{cor:one_direction:ThinG:HF} in general, let us translate Lemma~\ref{lem:Thin:SlopeCaps_equal_ThinCaps} into the present setting.

\begin{lemma}\label{lem:HF:SlopeCaps_equal_ThinCaps}
	For any \(\Gamma,\Gamma'\in\CurvesHF\),
	\(
	\ThinG(\Gamma)
	\cup
	\ThinG(\Gamma')
	=\QPI
	\)
	implies \(
	\BdryThin(\Gamma)\cap\BdryThin(\Gamma')
	=
	\Slopes_\Gamma\cap\Slopes_{\Gamma'}.
	\)
\end{lemma}
\begin{proof}
	Let \(C=\Phi(\Gamma)\) and \(C'=\Phi(\Gamma')\).
	By Theorem~\ref{thm:reduction:HF}~\ref{enu:reduction:HF:trivial}, \(C\) and \(C'\) are non-trivial. 
	By Theorem~\ref{thm:reduction:HF}~\ref{enu:reduction:HF:spaces},
	\(\Thin(C)\cup\Thin(C')=\QPI\) and
	\(
	\BdryThinG(\Gamma)
	\cap
	\BdryThinG(\Gamma')
	=
	\BdryThin(C)
	\cap
	\BdryThin(C')
	\).
	By Lemma~\ref{lem:Thin:SlopeCaps_equal_ThinCaps}, the latter is equal to 
	\(
	\Slopes_{C}\cap\Slopes_{C'}
	\)
	which by Theorem~\ref{thm:reduction:HF}~\ref{enu:reduction:HF:slopes} equals
	\(
	\Slopes_{\Gamma}\cap\Slopes_{\Gamma'}
	\).
\end{proof}

\begin{proof}[Proof of Theorems~\ref{thm:glueing:ALink:HF} and~\ref{thm:glueing:Thin:HF} for general multicurves in \(\CurvesHF\).]
	By Lemma~\ref{lem:HF:SlopeCaps_equal_ThinCaps}, 
	\(
	\BdryThinG(\Gamma_1)
	\cap
	\BdryThinG(\Gamma_2)
	=
	\Slopes_{\Gamma_1}\cap\Slopes_{\Gamma_2}
	\)
	provided conditions (1) in Theorems~\ref{thm:glueing:ALink:HF} and~\ref{thm:glueing:Thin:HF} hold. 
	Moreover, \(\Gamma_1\) and \(\Gamma_2\) satisfy \ref{local:HF:s-rational}  and \ref{local:HF:complementary} for some slope~\(s\) if and only if \(\HF(\gamma_1,\gamma_2)=0\) for all \(\gamma_1\in\Gamma_1\) and \(\gamma_2\in\Gamma_2\) of slope~\(s\) by Lemma~\ref{lem:HFT:epsilon_criterion}. 
	So by Theorem~\ref{thm:GlueingTheorem:HFT}, it suffices to show that \(\HF(\Gamma_1,\Gamma_2)\) is \(G\)-thin if and only if 
	\begin{enumerate}
		\item 	\(
		\ThinG(\Gamma_1)
		\cup 
		\ThinG(\Gamma_2)
		=
		\QPI
		\);
		and
		\item[(2')] for every slope
		\( s\in 
		\Slopes_{\Gamma_1}
		\cap
		\Slopes_{\Gamma_2}
		\), \(\HF(\gamma_1,\gamma_2)=0\) for all \(\gamma_1\in\Gamma_1\) and \(\gamma_2\in\Gamma_2\) of slope \(s\).
	\end{enumerate}
	We now go through the same case-by-case analysis depending on \(\Slopes_{\Gamma_1}\) and \(\Slopes_{\Gamma_2}\) as in the proof of Theorem~\ref{thm:glueing:ThinG:main}. \(\Gamma_1\) will play the role of \(C\) and \(\Gamma_2\) the role of \(D\). 
	Clearly, if \(\HF(\Gamma_1,\Gamma_2)\) is \(G\)-thin, then condition~(2') holds. So let us assume this condition from now on. It plays the same role as condition~(2') in the proof of Theorem~\ref{thm:glueing:ThinG:main}. 
	
	\casei{\Slopes_{\Gamma_1}\cap\Slopes_{\Gamma_2}=\varnothing} The original proof goes through unchanged.
	
	\casei{\Slopes_{\Gamma_1}\cap\Slopes_{\Gamma_2}=\{s\}}
		\begin{enumerate}[label=\textnormal{(\alph*)}]
			\item The case \(\Slopes_{\Gamma_1}=\Slopes_{\Gamma_2}=\{s\}\) goes through unchanged, because both curves contain some non-inhibited rational component. 
			\item Suppose \(\Slopes_{\Gamma_1}=\{s\}\) and \(\Slopes_{\Gamma_2}\supsetneq\{s\}\). The argument from the original proof can be adapted as follows: 
			If \(\Gamma_1\) contains curves of different \(\delta\)-gradings, the statements on either side of the asserted equivalence are wrong. 
			If all components of \(\Gamma_1\) have the same \(\delta\)-grading, \(\ThinG(\Gamma_1)=\QPI\smallsetminus\{s\}\), so condition (1) is equivalent to \(s\in\ThinG(\Gamma_2)\). 
			Since \(\Gamma_1\) is non-trivial, it contains some non-inhibited rational component of slope \(s\). By Lemma~\ref{lem:SomeControlOverLocalSystems} and assumption~(2'), \(\Gamma_2\) is \(s\)-inhibited. So \(s\in\ThinG(\Gamma_2)\) is equivalent to \(\delta(\Rational(s),\gamma)\in G\) being constant for all components \(\gamma\in\Gamma_2\) of slope different from \(s\). Still assuming~(2'), this is equivalent to 
			\(\HF(\Gamma_1,\Gamma_2)\) being thin. 
			
			\item Suppose \(\Slopes_{\Gamma_1}\supsetneq\{s\}\) and  \(\Slopes_{\Gamma_2}=\{s\}\). Same as Case~1(b) with reversed roles of \(\Gamma_1\) and \(\Gamma_2\). 
			
			\item Suppose \(|\Slopes_{\Gamma_i}|>1\) for \(i=1,2\). The proof in this case goes through unchanged, noting that condition~(2') implies that if \((s,t_1)\subseteq\ThinG(\Gamma_2)\) and \((s_m,s)\subseteq\ThinG(\Gamma_1)\), then also \(s\in\ThinG(\Gamma_1)\cup \ThinG(\Gamma_2)\), since at least one of \(\Gamma_1\) and \(\Gamma_2\) is \(s\)-inhibited. 
		\end{enumerate}
	\casei{\Slopes_{\Gamma_1}\cap\Slopes_{\Gamma_2}=\{s,t\}}
	
		\begin{enumerate}[label=\textnormal{(\alph*)}]
		\item Suppose that \(|\Slopes_{\Gamma_i}|=2\) for \(i=1,2\). 
					Most of the proof in this case goes through unchanged. 
					It only remains to see that 
					\(
					\ThinG(\Gamma_1)
					\cup
					\ThinG(\Gamma_2)
					\supseteq\QPI\smallsetminus\{s,t\}
					\) 
					implies that also \(s,t\in\ThinG(\Gamma_1)\cup \ThinG(\Gamma_2)\), under the assumption that condition~(2') holds. 
					Indeed, the inclusion implies that \(\ThinG(\Gamma_1)\) and \(\ThinG(\Gamma_2)\) are intervals with endpoints equal to \(s\) and \(t\). 
					Now suppose for contradiction that
					\(s\not\in\ThinG(\Gamma_1)\cup \ThinG(\Gamma_2)\). 
					Then there exist non-inhibited rational components \(\gamma_1\in\Gamma_1\) and \(\gamma_2\in\Gamma_2\) of slope \(s\). But this violates condition~(2') according to Lemma~\ref{lem:SomeControlOverLocalSystems}. 
					The same argument works if we replace \(s\) by \(t\). 
		\item Suppose \(|\Slopes_{\Gamma_1}|>2\).	The original argument goes through unchanged, using the same observation as in Case~2(a).
		\item Suppose \(|\Slopes_{\Gamma_2}|>2\). Same as Case~2(b) with reversed roles of \(\Gamma_1\) and \(\Gamma_2\). 
	\end{enumerate}

	\casei{|\Slopes_{\Gamma_1}\cap\Slopes_{\Gamma_2}|>2} This case goes through unchanged.
\end{proof}

\begin{proof}[Proof of Corollary~\ref{cor:one_direction:ThinG:HF} for general multicurves in \(\CurvesHF\).]
	As in the proof of Corollary~\ref{cor:one_direction:main}, observe that the assumption in Theorem~\ref{thm:glueing:Thin:HF} about \(\Gamma_1\) and \(\Gamma_2\) not both being exceptional is only used in Case~2(a). So the same arguments as in the proof of Corollary~\ref{cor:one_direction:main} apply.
\end{proof}

\section{\texorpdfstring{The tangle invariant \(\Khr\)}{The tangle invariant Khr}}\label{sec:review:Kh}

In this section, we review some properties of the immersed curve invariant \(\Khr\) of Conway tangles from~\cite{KWZ}. 
We work exclusively over the field \(\fieldTwoElements\) of two elements, with remarks about other coefficient systems when appropriate.

\subsection{\texorpdfstring{The definition of \(\Khr\)}{The definition of Khr}}\label{sec:review:Kh:definition}
 
Let \(T\) be an oriented \textbf{pointed} Conway tangle, that is a Conway tangle \(T\) in the three-ball \(B^3\) with a choice of distinguished tangle end, which we usually mark by~\(\ast\). With such a tangle, we associate an invariant \(\Khr(T)\), which takes the form of a collection of immersed curves with local systems on the boundary of \(B^3\) minus the four tangle ends. Like \(\HFT(T)\), these immersed curves with local systems are defined in two steps, which we sketch below. 

First, one fixes a diagram \(\Diag_T\) of the pointed tangle $T$. Bar-Natan associates with such a diagram a bigraded chain complex \(\KhTl{\Diag_T}\) over a certain cobordism category \(\Cobl\), whose objects are crossingless tangle diagrams \cite{BarNatanKhT}. 
This complex is a tangle invariant up to bigraded chain homotopy, and thus is frequently denoted by \(\KhTl{T}\). 
Thanks to a process Bar-Natan calls \textit{delooping}~\cite[Observation~4.18]{KWZ}, any chain complex over \(\Cobl\) can be written as a chain complex over the full subcategory \(\End_{\Cobl}(\Lo\oplus\Li)\) of \(\Cobl\) generated by the crossingless tangles without closed components. 
This subcategory is isomorphic to the following quiver algebra \cite[Theorem~1.1]{KWZ}:
\begin{equation}\label{eq:B_quiver}
\BNAlgH =
\fieldTwoElements\Big[
\begin{tikzcd}[row sep=2cm, column sep=1.5cm]
\DotB
\arrow[leftarrow,in=145, out=-145,looseness=5]{rl}[description]{{}_{\bullet}D_{\bullet}}
\arrow[leftarrow,bend left]{r}[description]{{}_{\bullet}S_{\circ}}
&
\DotC
\arrow[leftarrow,bend left]{l}[description]{{}_{\circ}S_{\bullet}}
\arrow[leftarrow,in=35, out=-35,looseness=5]{rl}[description]{{}_{\circ}D_{\circ}}
\end{tikzcd}
\Big]\Big/\Big(
\parbox[c]{120pt}{\footnotesize\centering
${}_{\bullet}D_{\bullet} \cdot {}_{\bullet}S_{\circ}=0={}_{\bullet}S_{\circ}\cdot  {}_{\circ}D_{\circ}$\\
${}_{\circ}D_{\circ}\cdot  {}_{\circ}S_{\bullet}=0={}_{\circ}S_{\bullet}\cdot  {}_{\bullet}D_{\bullet}$
}\Big)
\end{equation}
The objects \(\Lo\) and \(\Li\) correspond to \(\DotB\) and \(\DotC\), respectively. We denote the idempotent constant paths on \(\DotB\) and \(\DotC\) by \(\iota_{\bullet}\) and \(\iota_{\circ}\), respectively. We will sometimes abuse notation by using $S$  for either ${}_{\circ}S_{\bullet}$  or ${}_{\bullet}S_{\circ}$ and using $D$ for either ${}_{\circ}D_{\circ}$  or ${}_{\bullet}D_{\bullet}$.
In addition, the subscript $\star\in\{\DotC,\DotB\}$ on the left or right of an algebra element $a$ will always indicate that multiplying by $\iota_\star$ from the left or right respectively preserves the element $a$. This allows shorthand notation like for instance $S_{\circ}={}_{\bullet}S_{\circ}$ and $S^3_{\bullet}=  {}_{\circ}S_{\bullet}   \cdot {}_{\bullet}S_{\circ} \cdot  {}_{\circ}S_{\bullet}$.
The algebra \(\BNAlgH\) carries a bigrading: The quantum grading \(q\) and the delta grading  \(\delta \) are determined by 
\[
\text{gr}(D_{\bullet}) = \text{gr}(D_{\circ}) = q^{-2}\delta^{-1}
\qquad 
\text{and}
\qquad
\text{gr}(S_{\bullet}) = \text{gr}(S_{\circ}) = q^{-1} \delta^{-\frac{1}{2}} 
\] 
Differentials of bigraded chain complexes over \(\BNAlgH\) are defined to preserve quantum grading and decrease \(\delta\)-grading by 1. 
The isomorphism \(\End_{\Cobl}(\Lo\oplus\Li) \cong \BNAlgH\) allows us to translate the delooped chain complex \(\KhTl{\Diag_T}\) into a bigraded chain complex \(\DD(\Diag_T)^\BNAlgH\)~\cite[Definition~1.2]{KWZ}. ({Chain complexes over ordinary algebras} have also appeared in the literature under the name of {type~D structures}~\cite[Definition~2.2.23]{LOTBimodules}, and we will therefore use the two terms interchangeably; see~\cite[Proposition 2.13]{KWZ} for more on the equivalence of these objects.)  
By construction, the bigraded chain homotopy type of \(\DD(\Diag_T)^\BNAlgH\) is an invariant of the tangle~\(T\), and thus we will sometimes write  \(\DD(T)^\BNAlgH\) for  \(\DD(\Diag_T)^\BNAlgH\). 
Moreover, using the central element 
\[
H\coloneqq D+S^2 = D_{\bullet} + D_{\circ} + S_{\circ}S_{\bullet} + S_{\bullet}S_{\circ} ~ \in ~ \BNAlgH 
\]   
we define a bigraded chain complex \(\DD_1(\Diag_T)\) as the mapping cone 
\[\DD_1(\Diag_T)\coloneqq [q^{-1}\delta^{\frac 1 2}\DD(\Diag_T)\xrightarrow{H\cdot \id} q^{1}\delta^{\frac 1 2}\DD(\Diag_T)]\]
where  $H\cdot \id (x) = x\otimes H$ for every generator $x$ in $\DD(\Diag_T)$.
The bigraded chain homotopy type of \(\DD_1(\Diag_T)\)  is also a tangle invariant, and we will write  \(\DD_1(T)\) for \(\DD_1(\Diag_T)\). 

\begin{figure}[t]
	\centering
	\begin{subfigure}[t]{0.48\textwidth}
		\centering
		$\PairingTrefoilArcINTRO$
		\caption{The curve associated with $\textcolor{blue}{
				[
				\protect\begin{tikzcd}[nodes={inner sep=2pt}, column sep=13pt,ampersand replacement = \&]
				\protect\DotCblue
				\protect\arrow{r}{S}
				\protect\&
				\protect\DotBblue
				\protect\arrow{r}{D}
				\protect\&
				\protect\DotBblue
				\protect\arrow{r}{S^2}
				\protect\&
				\protect\DotBblue
				\protect\end{tikzcd}
				]}$}\label{fig:exa:classification:curves:arc}
	\end{subfigure}
	\
	\begin{subfigure}[t]{0.48\textwidth}
		\centering
		$\PairingTrefoilLoopINTRO$
		\caption{The curve associated with $\textcolor{red}{
				[
				\protect\begin{tikzcd}[nodes={inner sep=2pt}, column sep=23pt,ampersand replacement = \&]
				\protect\DotCred
				\protect\arrow{r}{D+S^2}
				\protect\&
				\protect\DotCred
				\protect\end{tikzcd}
				]}$}\label{fig:exa:classification:curves:loop}
	\end{subfigure}
	\caption{The geometric interpretation of some chain complexes over the algebra $\BNAlgH$ illustrating the classification theorem in the second part of the construction of \(\BNr(T)\) and \(\Khr(T)\)
	}\label{fig:exa:classification:curves}
\end{figure}
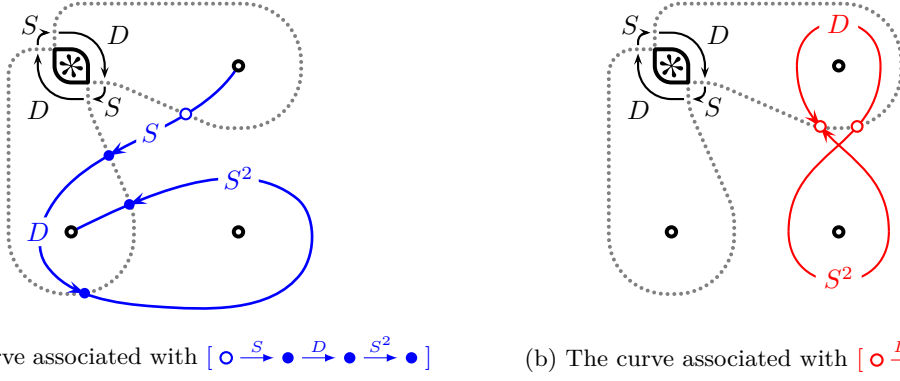

The second step in the definition of \(\Khr(T)\) relies on a classification result, similar to the one used in the definition of~\(\HFT(T)\). 
This classification result says that the chain homotopy classes of bigraded chain complexes over \(\BNAlgH\) are in one-to-one correspondence with free homotopy classes of bigraded immersed multicurves with local systems on the four-punctured sphere \(\FourPuncturedSphereKh\), where the latter has one \textbf{special} puncture distinguished by~\(\ast\) \cite[Theorem~1.5]{KWZ}.
In contrast to the \(\HFT(T)\) setting, here ``immersed curves'' also include non-compact curves, that is non-null-homotopic immersions of intervals into the four-punctured sphere, with ends on the three non-special punctures of \(\FourPuncturedSphereKh\); see \cite[Definition~1.4]{KWZ}. As in the \(\HFT(T)\) setting, the correspondence between chain complexes and immersed multicurves uses a parametrization of \(\FourPuncturedSphereKh\). 
This time, the parametrization is given by the two dotted arcs shown in Figure~\ref{fig:Kh:example:Curve:Downstairs}.
We will generally assume that the multicurves intersect these arcs minimally. Then, roughly speaking, the intersection points correspond to generators of the according chain complexes and paths between those intersection points correspond to the differentials. 
This is illustrated in Figure~\ref{fig:exa:classification:curves}, cf~\cite[Example~1.6]{KWZ}. 

Finally, the multicurve invariant \(\Khr(T)\) is defined as the collection of bigraded immersed curves on \(\FourPuncturedSphereKh\) that corresponds to \(\DD_1(T)\). 
Within the equivalence class of type~D structures that are chain homotopy equivalent to \(\DD_1(T)\), there exist certain distinguished representatives from which the multicurve \(\Khr(T)\) can be read off directly, as in Figure~\ref{fig:exa:classification:curves}. 
We denote such representatives by $\DD_1^c(T)$. 
Similarly, the type~D structure \(\DD(T)\) corresponds to a multicurve \(\BNr(T)\) and we write $\DD^c(T)$ for a type~D structure from which \(\BNr(T)\) can be read off directly.

While \(\Khr(T)\) only consists of compact curves, ie immersed circles, \(\BNr(T)\) also contains \(2^{|T|}\) non-compact components, where \(|T|\) is the number of closed components of \(T\). 
\begin{remark}[Coefficients]
In general, the construction of the tangle invariants \(\DD(T)^\BNAlgH\) and \(\DD_1(T)^\BNAlgH\) can be done over $\Z$. 
The classification result works over arbitrary fields, and hence so does the construction of the immersed curve invariants. However, in this paper, \(\BNr(T)\) and \(\Khr(T)\) will always denote the curves over $\fieldTwoElements$, unless stated otherwise. 
\end{remark}
 
One can identify \(\FourPuncturedSphereKh\) with \(\partial B^3\smallsetminus \partial T\) using the parametrization of the latter shown in Figure~\ref{fig:Kh:example:tangle}. 
This identification is natural in the same sense as for \(\HFT(T)\), see Theorem~\ref{thm:HFT:Twisting}, provided we work over \(\fieldTwoElements\) \cite[Theorem~1.13]{KWZ}. We expect the same to hold over arbitrary fields.	

\begin{figure}[t]
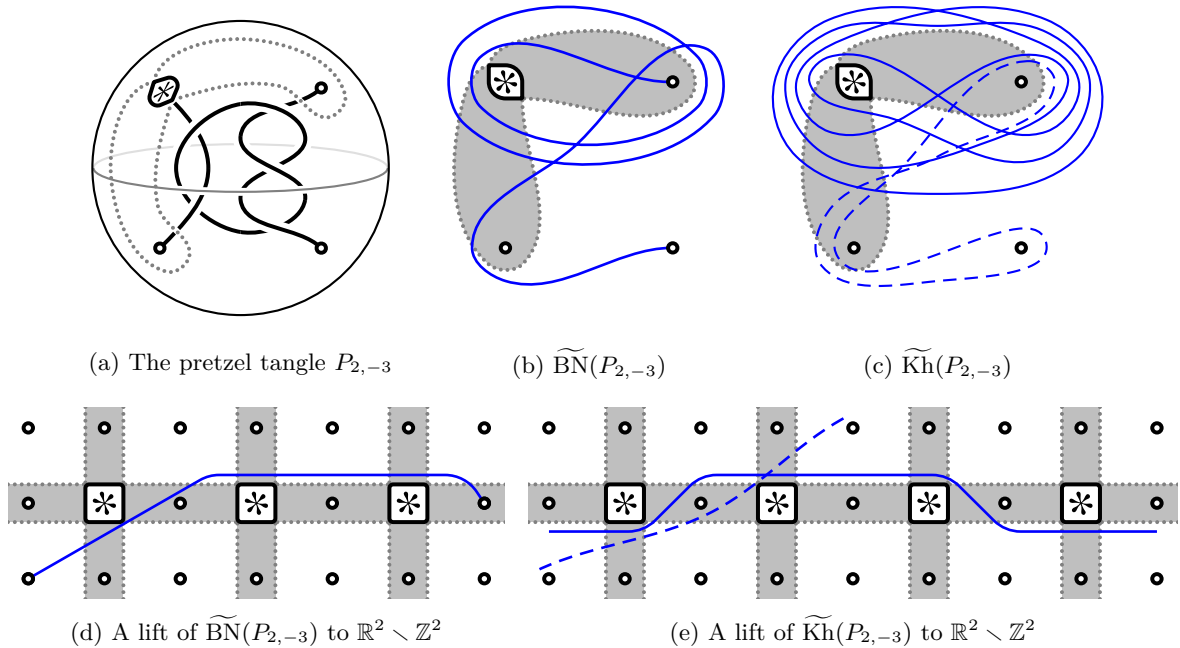

	\centering
	\begin{subfigure}{0.28\textwidth}
		\centering
		\(\pretzeltangleKh\)
		\caption{The pretzel tangle \(P_{2,-3}\)}\label{fig:Kh:example:tangle}
	\end{subfigure}
	\begin{subfigure}{0.28\textwidth}
		\centering
		\(\pretzeltangleDownstairsBNr\)
		\caption{\(\BNr(P_{2,-3})\)}\label{fig:BN:example:Curve:Downstairs}
	\end{subfigure}
	\begin{subfigure}{0.28\textwidth}
	\centering
	\(\pretzeltangleDownstairsKhr\)
	\caption{\(\Khr(P_{2,-3})\)}\label{fig:Kh:example:Curve:Downstairs}
	\end{subfigure}
	\bigskip
	\\
	\begin{subfigure}{0.42\textwidth}
		\centering
		\(\pretzeltangleUpstairsBNr\)
		\caption{A lift of \(\BNr(P_{2,-3})\) to \(\PuncturedPlane\) }\label{fig:BN:example:Curve:Upstairs}
	\end{subfigure}
	\begin{subfigure}{0.55\textwidth}
		\centering
		\(\pretzeltangleUpstairsKhr\)
		\caption{A lift of \(\Khr(P_{2,-3})\) to \(\PuncturedPlane\) }\label{fig:Kh:example:Curve:Upstairs}
	\end{subfigure}
	\caption{The Khovanov and Bar-Natan invariant of the pretzel tangle from Figure~\ref{fig:HFT:example} }\label{fig:Kh:example}
\end{figure}

\begin{theorem}\label{thm:Kh:Twisting}
	For all 
	\( 
	\tau\in \Mod(\FourPuncturedSphere)
	\), 
	\(
	\Khr(\tau(T))
	=
	\tau(\Khr(T))
	\)
	and 
	\(
	\BNr(\tau(T))
	=
	\tau(\BNr(T))
	\).
\end{theorem}

\begin{remark}
	When working over $\fieldTwoElements$, the distinguished tangle end \(\ast\) only plays a role in the second step of the construction of \(\BNr(T)\) and \(\Khr(T)\). 
	If one works away from characteristic 2, it also plays a subtle role in the first step: In this case, there are four different isomorphisms between \(\End_{\Cobl}(\Lo\oplus\Li)\) and \(\BNAlgH\), which only differ by the signs on the basic morphisms \(D_{\bullet}\) and \( D_{\circ}\). Each of these isomorphisms corresponds to a choice of distinguished tangle end; see~\cite[Theorem~4.21, Observation~4.24]{KWZ} for details. 
\end{remark}

\begin{example}\label{exa:Khr:2m3pt}
	We usually draw the four-punctured sphere \(\FourPuncturedSphereKh\) as the plane plus a point at infinity minus the four punctures and indicate its standard parametrization that identifies \(\FourPuncturedSphereKh\) with \(\partial B^3\smallsetminus \partial T\) by two dotted arcs as in  
	Figures~\ref{fig:BN:example:Curve:Downstairs} and~\ref{fig:Kh:example:Curve:Downstairs}. 
	The blue curves in these figures show \(\BNr(P_{2,-3})\) and \(\Khr(P_{2,-3})\), respectively, where \(P_{2,-3}\) is the \((2,-3)\)-pretzel tangle from Figure~\ref{fig:Kh:example:tangle}, cf~\cite[Example~6.7]{KWZ}.
	All components of these curves carry the (unique) one-dimensional local system over $\fieldTwoElements$.
\end{example}

\begin{example}
	For any slope \(s\in\QPI\), \(\BNr(Q_s)\) consists of a single arc which is obtained by pushing the tangle strand that does not end on the distinguished tangle end \(\ast\) onto \(\FourPuncturedSphereKh\). \(\Khr(Q_s)\) is equal to a figure-eight curve that lies in a small neighbourhood of \(\BNr(Q_s)\) and encloses the two tangle ends on either side, see \cite[Example~6.6]{KWZ}. 
	The local system on this curve is one-dimensional.
	We expect that over arbitrary fields, the underlying curve for \(\Khr(Q_s)\) is the same as over \(\fieldTwoElements\), and that the local systems on these curves are equal to \((-1)\).
\end{example}

In particular, \(\Khr(Q_s)\) is not embedded, unlike \(\HFT(Q_s)\). In fact, we have the following \cite[Proposition~6.18]{KWZ}, over any field: 

\begin{proposition}\label{prop:Khr_not_embedded}
	For any pointed Conway tangle \(T\), no component of \(\Khr(T)\) is embedded.
\end{proposition}

Like \(\HFT\), the tangle invariants in Khovanov theory detect rational tangles. 

\begin{theorem}
	\label{thm:Kh:rational_tangle_detection}
	A tangle \(T\) is rational if and only if \(\Khr(T)\) consists of a single figure-eight curve carrying the unique one-dimensional local system. 
\end{theorem}

\begin{proof}
	This follows from essentially the same arguments as \cite[Theorem~6.2]{pqMod}. More generally, any tangle invariant detects rational tangles, as long as the tangle invariant satisfies a gluing theorem to a link invariant that detects the two-component unlink. 
	The gluing theorem for \(\Khr\) is Theorem~\ref{thm:GlueingTheorem:Kh} below.  The requisite detection result was proven by Hedden and Ni~\cite{KhDetectsTwoComponentUnlink}, based on Kronheimer and Mrowka's unknot detection of Khovanov homology~\cite{KhDetectsUnknot}.
\end{proof}

\subsection{\texorpdfstring{A gluing theorem for \(\Khr\)}{A gluing theorem for Khr}}\label{sec:review:Kh:gluing}

Denote the two-dimensional vector space supported in \(\delta\)-grading \(+\tfrac{1}{2}\) and quantum gradings \(\pm1\) by
\[V\coloneqq \delta^{\frac 1 2} q^1\F \oplus \delta^{\frac 1 2} q^{-1} \F\]

\begin{theorem}[\text{\cite[Theorem~1.9]{KWZ}}]\label{thm:GlueingTheorem:Kh}
	Let \(L=T_1\cup T_2\) be the result of gluing two oriented pointed Conway tangles as in Figure~\ref{fig:tanglepairing} such that the orientations match. Let \(T^*_1\) be the mirror image of \(T_1\) with the orientation of all components reversed. Then  
	\[
	\begin{aligned}
	\Khr(L)\otimes V
	&\cong
	\HF\left(\Khr(T^*_1),\Khr(T_2)\right) \\
	\Khr(L)&\cong \HF\left(\Khr(T_1^*),\BNr(T_2)\right)
	\end{aligned}
	\]
	as relatively bigraded $\F$-vector spaces.
\end{theorem}

\(\Khr(T^*_1)\) can be easily computed from  \(\Khr(T_1)\) \cite[Proposition~7.1]{KWZ}. For this, let \(\mirror\) denote the mirror operation, ie the involution of the four-punctured sphere that fixes the four punctures pointwise, fixes the parametrizing arcs setwise, and interchanges the front and back, as in Section~\ref{sec:review:HFT}. 

\begin{lemma}\label{lem:mirroring:Khr}
	For any pointed Conway tangle \(T\), \(\Khr(T^*)=\mirror(\Khr(T))\) up to an appropriate bigrading shift.
\end{lemma}

\begin{remark}
	Shumakovitch showed that over \(\fieldTwoElements\) the unreduced Khovanov homology of a link splits into two copies of reduced Khovanov homology, whose \(\delta\)-gradings differ by one \cite{Shum_torsion}. 
	Theorem~\ref{thm:GlueingTheorem:Kh} does \emph{not} compute the unreduced Khovanov homology; instead we have that the \(\delta\)-gradings of the two copies of reduced Khovanov homology are identical. 
	To obtain unreduced Khovanov homology, one can use yet another immersed curve invariant, namely \(\Kh(T)\), which we introduced in~\cite{KWZ} and which satisfies analogous gluing theorems. 
\end{remark}

\subsection{\texorpdfstring{The geography problem for \(\Khr\)}{The geography problem for Khr}}
\label{sec:review:Khr:geography}

Just as for \(\HFT\), it is often useful to consider the multicurves \(\Khr\) and \(\BNr\) in the covering space \(\PuncturedPlane\) of the (now pointed) four-punctured sphere \(\FourPuncturedSphereKh\).
This covering space is illustrated in Figures~\ref{fig:BN:example:Curve:Upstairs} and~\ref{fig:Kh:example:Curve:Upstairs}, where the parametrization of \(\FourPuncturedSphereKh\) has been lifted to \(\PuncturedPlane\) and the two non-adjacent faces and their preimages under the covering map are shaded grey. 
The two figures also include the lifts of the components of \(\BNr(P_{2,-3})\) and \(\Khr(P_{2,-3})\), respectively. 
Note that for \(\Khr(P_{2,-3})\), the lift of each component can be isotoped into an arbitrarily small neighbourhood of a straight line of some rational slope \(\nicefrac{p}{q}\in\QPI\) going through some punctures. In fact, this is true in general for the invariant \(\Khr(T)\).

To make this statement precise, let us consider the following two classes of curves. 
For \(n\in\mathbb N\), let \(\rKh_{n}(0)\) and \(\sKh_{2n}(0)\) be the immersed curves in \(\FourPuncturedSphereKh\) that respectively admit lifts to the curves $\tr_{n}(0)$ and $\ts_{2n}(0)$ in Figure~\ref{fig:geography:Upstairs}; curves for $n=1,2,3$ are illustrated in Figures~\ref{fig:geography:Downstairs1}--\ref{fig:geography:Downstairs3}.
We will refer to the subscripts \(n\), respectively \(2n\), as the \textit{lengths} of those curves. 
For every slope \(\nicefrac{p}{q}\in\QPI\), we respectively define the curves $\rKh_n(\nicefrac{p}{q})$ and $\sKh_{2n}(\nicefrac{p}{q})$ as the images of \(\rKh_{n}(0)\) and \(\sKh_{2n}(0)\) under the action of the matrix
\[
\begin{bmatrix*}[c]
q & r \\
p & s
\end{bmatrix*}, \qquad qs-pr=1,
\]
considered as an element the mapping class group $\operatorname{Mod}(\FourPuncturedSphereKh) \cong \mathit{PSL}(2,\Z)$ consisting of mapping classes fixing the special puncture $\ast$. (This transformation maps straight lines of slope 0 to straight lines of slope \(\nicefrac{p}{q}\). The isomorphism $\operatorname{Mod}(\FourPuncturedSphereKh) \cong \mathit{PSL}(2,\Z)$ is induced by the two-fold cover $\mathbb T^2 \to \FourPuncturedSphereKh$ and the isomorphism $\operatorname{Mod}(\mathbb T^2) \cong \mathit{SL}(2,\Z)$.)
The local systems on all these curves are defined to be trivial.

\begin{definition}\label{def:RationalVsSpecialKht}
	We call the curves \(\rKh_{n}(\nicefrac{p}{q})\) \textbf{rational} and the curves \(\sKh_{2n}(\nicefrac{p}{q})\) \textbf{special}. 
	As in the Heegaard Floer setting, given some slope \(s\in\QPI\), we will call a multicurve \textbf{\(s\)-rational} if it does not contain any special component of slope \(s\), and \textbf{\(s\)-special} if it does not contain any rational component of slope \(s\). 
\end{definition}

\begin{example}
	For all slopes \(\nicefrac{p}{q}\in\QPI\), 
	\(\Khr(Q_{p/q})=\rKh_1(\nicefrac{p}{q})\). 
	Moreover, \(\Khr(P_{2,-3})\) from Figure~\ref{fig:Kh:example:Curve:Upstairs} consists of a rational curve of slope \(\nicefrac{1}{2}\) and a special curve of slope \(0\). 
\end{example}

The following classification result, which is \cite[Theorem~1.2]{KWZ-KhMutation}, establishes a similar dichotomy between rational and special components of \(\Khr(T)\) as Theorem~\ref{thm:geography_of_HFT} does for \(\HFT(T)\).

\begin{theorem}\label{thm:geography_of_Khr}
	Each component of the invariant \(\Khr(T)\) is either rational or special.
\end{theorem}

\begin{figure}[t]
	\centering
	\begin{subfigure}{0.3\textwidth}
		\centering
		\labellist 
		\footnotesize \color{blue}
		\pinlabel $\sKh_{2}(0)$ at 65 115
		\pinlabel $\rKh_{1}(0)$ at 65 28
		\endlabellist
		\includegraphics[scale=1]{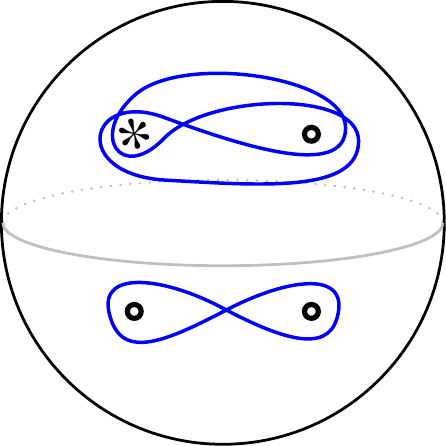}
		\caption{\(n=1\)}\label{fig:geography:Downstairs1}
	\end{subfigure}
	\begin{subfigure}{0.3\textwidth}
		\centering
		\labellist 
		\footnotesize \color{blue}
		\pinlabel $\sKh_{4}(0)$ at 65 118
		\pinlabel $\rKh_{2}(0)$ at 65 13
		\endlabellist
		\includegraphics[scale=1]{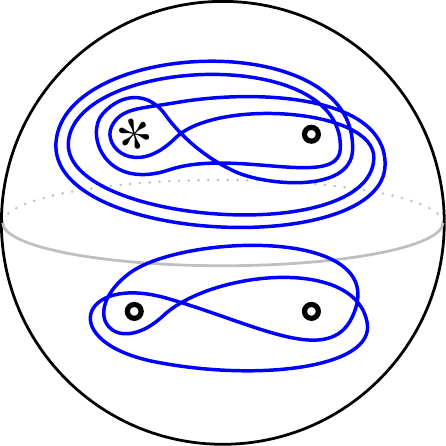}
		\caption{\(n=2\)}\label{fig:geography:Downstairs2}
	\end{subfigure}
	\begin{subfigure}{0.3\textwidth}
		\centering
		\labellist 
		\footnotesize \color{blue}
		\pinlabel $\sKh_{6}(0)$ at 65 121
		\pinlabel $\rKh_{3}(0)$ at 65 10
		\endlabellist
		\includegraphics[scale=1]{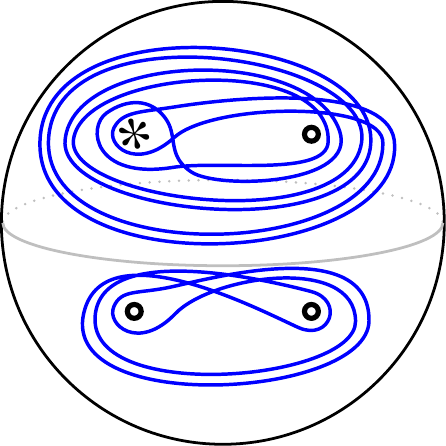}
		\caption{\(n=3\)}\label{fig:geography:Downstairs3}
	\end{subfigure}
	\\
	\begin{subfigure}{0.9\textwidth}
		\centering
		\(\GeographyCovering\)
		\caption{}\label{fig:geography:Upstairs}
	\end{subfigure}
	\caption{The curves \(\rKh_n(0)\) and \(\sKh_{2n}(0)\) (a--c) and their lifts to \(\PuncturedPlane\) (d). }\label{fig:geography}
\end{figure}

%

The components of the invariant $\BNr(T)$ can be much more intricate than components of \(\Khr(T)\), 
as the curve \(\BNr(P_{2,-3})\) from Figure~\ref{fig:BN:example:Curve:Upstairs} illustrates. 
Even compact components can be more complicated, see for example \cite[Figure~15]{KWZ-KhMutation}. 
However, the proof of Theorem~\ref{thm:geography_of_Khr} still gives some restrictions for $\BNr(T)$, too: 
For instance, the curve \(\BNr(P_{2,-3})\) ``wraps'' around the upper right tangle end, which is a non-special tangle end; 
in general, this kind of wrapping cannot occur around the special tangle end \(\ast\) \cite[Theorem~3.9]{KWZ-KhMutation}. 


\section{Khovanov thin fillings}\label{sec:Kh:ThinFillings}

\subsection{\texorpdfstring{The \(\delta\)-grading of curves in the covering space for \(\Khr\) and \(\BNr\)}{The δ-grading of curves in the covering space for Khr and BNr}}\label{sec:Kh:Thin:Covering}

The parametrization of the four-punctured sphere that we use for the definition of the invariants \(\BNr\) and \(\Khr\) is different from the one used for \(\HFT\): There are only two parametrizing arcs instead of four and one puncture is treated different from the other three. This also has consequences for how we think of the covering space \(\mathbb{R}^2\smallsetminus\mathbb{Z}^2\) of the four-punctured sphere. In this covering space, we will think of the special puncture in terms of marked squares and of the non-special punctures as marked points. The arcs end at the vertices of the squares, as is drawn in Figure~\ref{fig:BN:example:Curve:Upstairs}.

To stay consistent with the conventions in~\cite{KWZ}, the normal vector field of all surfaces determined by the right-hand rule points out of the page. Thus, the boundary of a region of multiplicity 1 in the plane is oriented counter-clockwise. This is opposite to the conventions used in Section~\ref{sec:HFT:ThinFillings}; in particular, this results in the appearance of minus signs.

\begin{lemma}\label{lem:delta:Kh:Euler:one_curve}
  For any connecting domain \(\varphi\) from \(x\) to \(x'\) (see Definition~\ref{def:connecting_domain:same_curve}), 
  \[
  \delta(x')-\delta(x)=-2e(\varphi). 
  \]
\end{lemma}

\begin{lemma}\label{lem:delta:Kh:Euler:two_curves}
	With notation as in Definition~\ref{def:connecting_domain:two_curves}, the \(\delta\)-grading of \(\bullet\) is equal to 
	\[
	\delta(y)-\delta(x)-\tfrac{1}{2}+2e(\varphi).
	\]
\end{lemma}

\begin{proposition}\label{prop:Kh:Euler:multiple-curves}
	For any tuple \((x_i)_{i=1,\dots,n}\) of intersection points, a domain \(\varphi\) as in Definition~\ref{def:symmetric_domain} satisfies
	\[
	\sum_{i=1}^n \delta(x_i)=2 e(\varphi)-\tfrac{n}{2}.
	\]
\end{proposition}

\begin{corollary}\label{cor:Kh:true_domain}
	For any domain \(\varphi\) from \(x\) to \(y\) (see Definition~\ref{def:asymmetric_domain}), 
	\[\delta(y)-\delta(x)=-2e(\varphi).\]
\end{corollary}

\begin{example}\label{exa:Kh:bigon}
	If \(\varphi\) is a bigon of multiplicity 1 as in Figure~\ref{fig:Kh:bigon}, Corollary~\ref{cor:Kh:true_domain} implies that \(\delta(y)-\delta(x)=-2e(\varphi)=-1\); see also~\cite[Lemma~5.21]{KWZ}.
\end{example}

\begin{figure}[b]
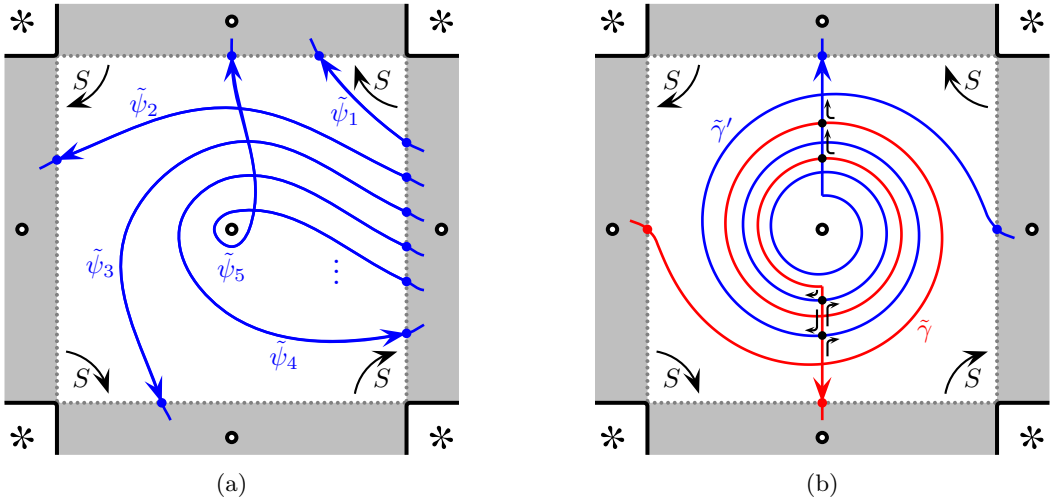

	\begin{subfigure}{0.48\textwidth}
		\centering
		\(\KhOneCurve\)
		\caption{}\label{fig:Kh:one_curve}
	\end{subfigure}
	\begin{subfigure}{0.48\textwidth}
		\centering
		\(\KhTwoCurves\)
		\caption{}\label{fig:Kh:two_curves}
	\end{subfigure}
	\caption{Some basic curve segments (a) and their pairings (b) that illustrate the proofs of Lemmas~\ref{lem:delta:Kh:Euler:one_curve} and~\ref{lem:delta:Kh:Euler:two_curves}}\label{fig:delta:Kh}
\end{figure}

\begin{figure}[t]
	\centering
	\(\Khbigon\)
	\caption{A bigon illustrating Example~\ref{exa:Kh:bigon}; compare with Figure~\ref{fig:HFT:bigon} and~\cite[Figure~17]{KWZ}}\label{fig:Kh:bigon}
\end{figure}

\begin{proof}[Proofs of results~\ref{lem:delta:Kh:Euler:one_curve}--\ref{cor:Kh:true_domain}]
  The proofs of the two lemmas are very similar to those of Lemma~\ref{lem:delta:HFT:Euler:one_curve} and~\ref{lem:delta:HFT:Euler:two_curves}, respectively, since we were careful not to assume linearity of the curves in those proofs.
  Note that the square punctures can always be filled in if necessary, since they contribute 0 to the Euler measure. 
  
  For the first lemma, first consider the basic curve segments \(\psi_i\), \(i=1,2,3,\dots\), that are confined to a single face. 
  In the case of an \(S\)-face, this is illustrated in Figure~\ref{fig:Kh:one_curve}; for the \(D\)-faces, only the curve segments \(\psi_i\) for even \(i\) are relevant. If \(\partial\tilde{\psi}_i=\Lx'_i-\Lx_i\), then for each integer \(i\), there is a unique connecting domain \(\varphi_i\) from \(\Lx_i\) to \(\Lx'_i\) avoiding the puncture of the face. The number of convex corners is equal to \(i+2\), so \(e(\varphi_i)=\frac{1}{2}-\frac{i}{4}\). Moreover, \(\psi_i\) corresponds to a component of the differential labelled by some algebra element \(a_i\), which is equal to \(S^i\) for \(S\)-faces and equal to \(D^{i/2}\) for \(D\)-faces. So in both cases \(\delta(a_i)=-\frac{i}{2}\). Following the conventions from~\cite{KWZ}, the \(\delta\)-grading decreases along the differential by 1, so 
  \[
  \delta(x'_i)-\delta(x_i)=-1-\delta(a_i)=-1+\tfrac{i}{2}=-2 e(\varphi_i)
  \]
  The argument for general domains is identical to the one for \(\HFT\).
  
  Similarly, for Lemma~\ref{lem:delta:Kh:Euler:two_curves}, once we know the identity for regions confined to a single rectangle, the general case follows as in the proof of Lemma~\ref{lem:delta:HFT:Euler:two_curves}, using Lemma~\ref{lem:delta:Kh:Euler:one_curve} in place of Lemma~\ref{lem:delta:HFT:Euler:one_curve}. 
  So let us consider those basic domains, which in the case of an \(S\)-face are illustrated in Figure~\ref{fig:Kh:two_curves}. 
  The differential corresponding to an intersection point can be easily read off by considering the two generators \(\Lx\) and \(\Ly\) on the first and second curve, respectively, that are connected by a path that turns right at the intersection point. The retraction of this path to the boundary of the face determines the label of the corresponding differential, namely \(S^i\) or \(D^{i/2}\), where \(i\) is the number of corners. So the grading of the intersection point is
  \[
  \delta(y)-\delta(x)-\tfrac{i}{2}
  \]
  The connecting domain is a disc with \(i+3\) convex corners and multiplicity \(+1\), so its Euler measure is equal to \(\tfrac{1-i}{4}\).
  
  Finally, the proofs of Proposition~\ref{prop:Kh:Euler:multiple-curves} and Corollary~\ref{cor:Kh:true_domain} are identical to the proofs of Proposition~\ref{prop:HFT:Euler:multiple-curves} and Corollary~\ref{cor:HFT:true_domain}, respectively, except that we use
  Lemma~\ref{lem:delta:Kh:Euler:two_curves} in place of Lemma~\ref{lem:delta:HFT:Euler:two_curves}.
\end{proof}

\subsection{Linear curves}\label{sec:Kh:Thin:Linear}
Restricting to linear curves (Definition~\ref{def:linearity_via_derivative}), we obtain results very similar to those in Section~\ref{sec:HFT:Thin:Linear}.

\begin{deflemma}\label{deflem:delta:Kh:linear_curve}
  Let \(\gamma\) be a linear curve of slope \(s\in\QPI\). Then unless \(s=0\), all intersection points with the horizontal lines of \(\paraKh\) have the same \(\delta\)-grading \(\delH\coloneqq\delH(\gamma)\), and unless \(s=\infty\), all intersection points with the vertical lines of \(\paraKh\) have the same \(\delta\)-grading \(\delV\coloneqq\delV(\gamma)\). Moreover, 
  \[
  \delV=
  \begin{cases}
  \delH-\frac{1}{2} &\text{ if \(0<s<\infty\)} \\
  \delH+\frac{1}{2} &\text{ if \(\infty<s<0\)}
  \end{cases}
  \]
\end{deflemma}

\begin{deflemma}\label{deflem:delta:Kh:linear_curves}
  Given two linear curves \(\gamma\) and \(\gamma'\) of different slopes \(s,s'\in\QPI\), the Lagrangian intersection theory \(\HF(\gamma,\gamma')\) is supported in a single \(\delta\)-grading, which is equal to 
  \[
  \delta(\gamma,\gamma')
  \coloneqq
  \begin{cases*}
  \delH(\gamma')-\delV(\gamma)-\tfrac{1}{2} & if \(s\in(\infty,s')\) for \(s'\in(0,\infty]\), or \(s\in(s',\infty)\) for \(s'\in[\infty,0)\)
  \\
  \delV(\gamma')-\delH(\gamma)-\tfrac{1}{2} &  if \(s\in(s',0)\) for \(s'\in[0,\infty)\), or \(s\in(0,s')\) for \(s'\in(\infty,0]\)
  \end{cases*}
  \]
\end{deflemma}

\begin{corollary}\label{cor:deltaDifference:Kh:commutativity}
  For any two linear curves \(\gamma\) and \(\gamma'\),
  \[
  \delta(\gamma,\gamma')+\delta(\gamma',\gamma)=
  \begin{cases*}
  0 & if \(s(\gamma)=s(\gamma')\)\\
  -1 & if \(s(\gamma)\neq s(\gamma')\)
  \end{cases*}
  \]
\end{corollary}

\begin{theorem}\label{thm:deltaDifference:Kh:transitivity}
  For any increasing triple \((\gamma,\gamma',\gamma'')\) of linear curves, 
  \[
  \delta(\gamma,\gamma')+\delta(\gamma',\gamma'')=\delta(\gamma,\gamma'').
  \]
\end{theorem}

\begin{proof}[Proof of results~\ref{deflem:delta:Kh:linear_curve}--\ref{thm:deltaDifference:Kh:transitivity}]
	The proof of Lemma~\ref{deflem:delta:Kh:linear_curve} is identical to the proof of Lemma~\ref{deflem:delta:linear_curve}, except that the multiplicities of the triangles change sign and we use 
	Lemma~\ref{lem:delta:Kh:Euler:one_curve} in place of Lemma~\ref{lem:delta:HFT:Euler:one_curve}.
	Similarly, Lemma~\ref{deflem:delta:Kh:linear_curves} follows from the same arguments as Lemma~\ref{deflem:delta:HF:linear_curves} using Lemma~\ref{lem:delta:Kh:Euler:two_curves} in place of Lemma~\ref{lem:delta:HFT:Euler:two_curves}; note the opposite signs of the summands \(\tfrac{1}{2}\).
	Corollary~\ref{cor:deltaDifference:Kh:commutativity} follows immediately from Lemma~\ref{lem:delta:Kh:Euler:two_curves}, similar to Corollary~\ref{cor:deltaDifference:HFT:commutativity}.
	For Theorem~\ref{thm:deltaDifference:Kh:transitivity}, we can adapt the proof of Theorem~\ref{thm:deltaDifference:HFT:transitivity} as follows: First, note that the triple \((\gamma'',\gamma',\gamma)\) is decreasing, so let us swap the roles of \(\gamma''\) and \(\gamma\) in the proof of Theorem~\ref{thm:deltaDifference:HFT:transitivity}. Then the domains \(\varphi\) and \(\varphi'\) only contain regions of multiplicity \(-1\) and \(e(\LDelta)=+\tfrac{1}{4}\). Using Lemma~\ref{lem:delta:Kh:Euler:two_curves} in place of Lemma~\ref{lem:delta:HFT:Euler:two_curves}, this implies that
	\begin{align*}
	\delta(\gamma'',\gamma')+\delta(\gamma',\gamma)
	&=
	\Big(\delta(y)-\delta(x)-\tfrac{1}{2}+2e(\varphi)\Big)+
	\Big(\delta(z)-\delta(y)-\tfrac{1}{2}+2e(\varphi')\Big)
	\\
	&=
	\delta(z)-\delta(x)-\tfrac{3}{2}+2e(\varphi+\varphi'+\LDelta)
	=
	\delta(\gamma'',\gamma)-1.
	\end{align*}
  We add 2 on both sides, apply Corollary~\ref{cor:deltaDifference:Kh:commutativity}, multiply both sides by \(-1\), and obtain 
  \begin{equation*}
  \delta(\gamma',\gamma'')+\delta(\gamma,\gamma')
  =
  \delta(\gamma,\gamma'').\qedhere
  \end{equation*}
\end{proof}

\begin{lemma}\label{lem:Kh:epsilon_criterion}
	Suppose \(\gamma\) and \(\gamma'\) are two linear curves with local systems in \(\FourPuncturedSphereKh\) that share the same slope \(s\in\QPI\). 
	Then, if \(\gamma\) is rational and \(\gamma'\) is special (or vice versa), the vector space \(\HF(\gamma,\gamma')\) is zero. 
	Otherwise, \(\HF(\gamma,\gamma')\) is non-zero and supported in two consecutive \(\delta\)-gradings, namely \(\delta(\gamma,\gamma')\) and \(\delta(\gamma,\gamma')-1\). 
\end{lemma}

\begin{proof}
	Clearly, the Lagrangian Floer homology of a special and a rational curve vanishes. 
	Moreover, any two curves of the same slope that are either both rational or both special intersect non-trivially, so \(\HF(\gamma,\gamma')\) does not vanish. The support of this vector space can be computed in the same way as in the proof of Lemma~\ref{lem:HFT:epsilon_criterion} for the case of two special curves.
\end{proof}

\begin{definition}\label{def:Kh:exceptional}
	In analogy to Definition~\ref{def:HF:exceptional}, we say that a tangle \(T\) is \textbf{Khovanov exceptional} if the multicurve \(\Kh(T)\) is exceptional.
\end{definition}

\begin{conjecture}\label{conj:non-adjacent:Kh}
	There exists no link \(L\) for which \(\Khr(L)\) is supported in precisely two non-adjacent \(\delta\)-gradings.  
\end{conjecture}

\begin{proposition}[Proposition~\ref{prop:exceptional:intro}]\label{prop:exceptional:Kh}
	If there exists a Khovanov exceptional tangle, then Conjecture~\ref{conj:non-adjacent:Kh} is false. 
\end{proposition}

\begin{proof}
	Suppose there exists a Khovanov exceptional tangle \(T\). Let us write \(\Gamma=\Khr(T)\) and \(\Slopes_{\Gamma}=\{s,s'\}\). 
	By assumption, \(T\) is exceptional, so either \(|\delta(\gamma,\gamma')|>1\) or \(|\delta(\gamma',\gamma)|>1\) for all components \(\gamma,\gamma'\in\Gamma\) with \(s(\gamma)=s\) and \(s(\gamma')=s'\). Without loss of generality, let us assume the former. Then, if we pick a slope \(t\neq s\) such that \((t,s,s')\) is increasing,  \(\delta(\Rational(t),\gamma')=\delta(\Rational(t),\gamma)+\delta(\gamma,\gamma')\) for all pairs \((\gamma,\gamma')\) as above. Since \(\Gamma\) is assumed to be \(s\)- and \(s'\)-consistent, \(\HF(\Rational(t),\Gamma)\) is supported in precisely two non-adjacent \(\delta\)-gradings. So by Theorem~\ref{thm:GlueingTheorem:Kh}, the link \(Q_{-t}\cup T\) is a counterexample to Conjecture~\ref{conj:non-adjacent:Kh}. 
\end{proof}

\subsection{Khovanov thin fillings}
In this subsection, \(G\) is either \(\Z\) or \(\ZZ\). Define 
\[
\CurvesKh
\coloneqq
\Bigl\{\Khr(T;\fieldTwoElements)\Bigm|\text{Conway tangles }T\Bigr\}
\]
In the following, we will make implicit use of the following properties that \(\CurvesKh\) is known to satisfy: 
Each multicurve \(\Gamma\in\CurvesKh\) consists of linear components only (Theorem~\ref{thm:geography_of_Khr}); and \(\HF(\Gamma_1,\Gamma_2)\neq 0\) for each \(\Gamma_1,\Gamma_2\in\CurvesKh\), because of Theorem~\ref{thm:GlueingTheorem:Kh} and the fact that reduced Khovanov homology does not vanish. 

Given two multicurves \(\Gamma\) and \(\Gamma'\) and a slope \(s\in\Slopes_\Gamma\cap\Slopes_{\Gamma'}\), the following condition will be relevant:
\begin{enumerate}
	\myitem[\(\mathrm{R}\)] \label{local:Kh} At least one of \(\Gamma\) and \(\Gamma'\) is \(s\)-rational, ie it only contains rational components of slope \(s\).
\end{enumerate}
This is the condition for reduced Khovanov theory mentioned in Theorems~\ref{thm:gluing:ALink:intro} and~\ref{thm:gluing:Thin:intro}. 

\begin{definition}
	Given a relatively \(\delta\)-graded multicurve \(\Gamma\), let
	\[
	\ThinG(\Gamma)\coloneqq\Bigl\{s\in\QPI\Bigm| \HF(\Rational(s),\Gamma)\text{ is \(G\)-thin}\Bigr\}
	\]
	be the spaces of \textbf{\(\bm{G}\)-thin rational fillings} of \(\Gamma\).  
	If \(T\) is a Conway tangle in a three-ball, writing 
	\[
	\ThinKh(T)=\ThinZ(\Khr(T))
	\quad\text{ and }\quad
	\ALinkKh(T)=\ThinZZ(\Khr(T))
	\] 
	recovers the definition from the introduction, which follows from Theorem~\ref{thm:GlueingTheorem:Kh}. 
\end{definition}
 
\begin{remark}\label{rem:mirroring:ThinG:Kh}
	Since by Lemma~\ref{lem:mirroring:Khr} the tangle invariant \(\Khr\) behaves in a natural way under mirroring,
	\(\ThinKh(T^*)=\mirrorThinKh(T)\)
	and
	\(\ALinkKh(T^*)=\mirrorALinkKh(T)\)
	for any Conway tangle \(T\). 
\end{remark}

The following result plays the same role for \(\Khr(T)\) as Theorem~\ref{thm:reduction:HF} does for \(\HFT(T)\).
Again, we denote the set of all line sets in the sense of Section~\ref{sec:main} by \(\mathcal{P}_\mathrm{finite}(\Curves)\).

\begin{theorem}\label{thm:reduction:Kh}
	There exist a map \(\Phi\co\CurvesKh\rightarrow\mathcal{P}_\mathrm{finite}(\Curves)\) and a map \(\gr\co \Curves^2\rightarrow G\)	satisfying the \ref{eq:symmetry}, \ref{eq:transitivity}, and \ref{eq:linearity} properties as in Section~\ref{sec:main} such that for any \(\Gamma,\Gamma'\in\CurvesKh\), the following holds:
	\begin{enumerate}[label=(\roman*)] 
		\item \label{enu:reduction:Kh:slopes} \(\Slopes_{\Gamma}=\Slopes_{\Phi(\Gamma)}\).
		\item \label{enu:reduction:Kh:spaces}
		\(
		\ThinG(\Gamma)=\ThinG(\Phi(\Gamma)).
		\)
		\item \label{enu:reduction:Kh:trivial} \(\Phi(\Gamma)\) is non-trivial. 
		\item \label{enu:reduction:Kh:exceptional} \(\Phi(\Gamma)\) is exceptional if \(\Gamma\) is exceptional. 
		\item \label{enu:reduction:Kh:s-rational} for any slope \(s\in\QPI\), \(\Gamma\) is \(s\)-rational if and only if \(\Phi(\Gamma)\) is \(s\)-rational. 
		\item \label{enu:reduction:Kh:pairs:Thin}  \(\HF(\Gamma,\Gamma')\) is \(G\)-thin if and only if the pair \((\Phi(\Gamma), \Phi(\Gamma'))\) is \(G\)-thin.
	\end{enumerate}
\end{theorem}

\begin{proof}
	Analogously to the proof of Theorem~\ref{thm:reduction:HF}, given \(c\in\Curves\), let \(\gamma(c)\) be an absolutely \(\delta\)-graded linear curve of slope \(s(c)\) such that \(\delH(\gamma(c))=\gr(c)\) if \(s(c)\neq0\) and \(\delV(\gamma(c))=\gr(c)+\tfrac{1}{2}\) if \(s(c)=0\). 
	Then define \(\gr\co \Curves^2\rightarrow G\) by setting for each \(c,c' \in \Curves\)
	\[
	\gr(c,c')\coloneqq \delta(\gamma(c),\gamma(c'))
	\]
	By Corollary~\ref{cor:deltaDifference:Kh:commutativity}, \ref{eq:symmetry} of \(\gr\) holds, and by Theorem~\ref{thm:deltaDifference:Kh:transitivity}, so does \ref{eq:transitivity} of \(\gr\). 
	Moreover, \ref{eq:linearity} of \(\gr\) follows from the definition. 
	
	We can lift the \(\delta\)-grading of all curves in \(\CurvesKh\) to an absolute \(\delta\)-grading such that for each component \(\gamma\) of any element in \(\CurvesKh\) we have \(\delH(\gamma)\in \Z, ~ \delV(\gamma)\in \Z+\tfrac 1 2\). 
	To see that this is possible, we can apply the same arguments as in the proof of Theorem~\ref{thm:reduction:HF}, now using Lemma~\ref{deflem:delta:Kh:linear_curves} in place of Lemma~\ref{deflem:delta:HF:linear_curves} and the fact that reduced Khovanov homology is supported in integer \(\delta\)-gradings up to an overall shift, like knot Floer homology. 
	
	Now, given some absolutely \(\delta\)-graded rational or special curve \(\gamma\) of slope \(s\), let \(c=c(\gamma)\in\Curves\) be the line defined by \(s(c)=s\), \(\varepsilon(c)=0\) if \(\gamma\) is rational and 1 if \(\gamma\) is special, and \(\gr(c)=\delH(\gamma)\) if \(s\neq0\) and \(\delV(\gamma)-\tfrac{1}{2}\) if \(s(c)=0\).
	Then, given some \(\Gamma=\{\gamma_i\}_i\in\CurvesKh\), define \(C(\Gamma)\) as the set corresponding to the multiset \(\{c(\gamma_i)\}_{i}\).
	
	Properties~\ref{enu:reduction:Kh:slopes}--\ref{enu:reduction:Kh:s-rational} follow immediately from the definition of the map~\(\Phi\).
	The property~\ref{enu:reduction:Kh:pairs:Thin} follows from Lemma~\ref{deflem:delta:Kh:linear_curves} and Lemma~\ref{lem:Kh:epsilon_criterion}.
\end{proof}

We can now show those parts of all theorems from the introduction concerning reduced Khovanov homology. Again, we are restating them here for clarity. 

\begin{theorem}[Characterization of Khovanov \(G\)-thin filling spaces; Theorems~\ref{thm:charactisation:ALink:intro} and~\ref{thm:charactisation:Thin:intro}]\label{thm:charactisation:Thin:Kh}
	For any Conway tangle \(T\), \(\ALinkKh(T)\) is either empty, a single point or an interval in \(\QPI\). 
	Furthermore, \(\ThinKh(T)\) is either empty, a single point, two distinct points or an interval in \(\QPI\). 
\end{theorem}

\begin{proof}
	This follows from Theorems~\ref{thm:charactisation:ThinG:main} and \ref{thm:reduction:Kh}~\ref{enu:reduction:Kh:spaces} and~\ref{enu:reduction:Kh:trivial}.
\end{proof}

\begin{proposition}[Proposition~\ref{prop:spaces_coincide:intro}]\label{prop:spaces_coincide:Kh}
	If \(\ThinKh(T)\) is an interval, \(\ThinKh(T)=\ALinkKh(T)\).
\end{proposition}

\begin{proof}
	This follows from Theorem~\ref{thm:reduction:Kh}~\ref{enu:reduction:Kh:spaces} in conjunction with Proposition~\ref{prop:spaces_coincide:main}.
\end{proof}

In the following, let \(T_1\) and \(T_2\) be two Conway tangles and write \(\Gamma_1\coloneqq\mirror(\Khr(T_1))\) and \(\Gamma_2\coloneqq\Khr(T_2)\). 

\begin{theorem}[A-link Gluing Theorem; Theorem~\ref{thm:gluing:ALink:intro}]\label{thm:glueing:ALink:Kh}
	Let \(T_1\) and \(T_2\) be two Conway tangles. 
	Then \(T_1\cup T_2\) is a Khovanov A-link if and only if
	\begin{enumerate}
		\item 	\(
		\mirrorALinkKh(T_1)
		\cup 
		\ALinkKh(T_2)
		=
		\QPI
		\);
		and
		\item for every slope
		\( s\in 
		\mirrorBdryALinkKh(T_1)
		\cap
		\BdryALinkKh(T_2)
		\), the multicurves \(\Gamma_1\) and \(\Gamma_2\) satisfy~\ref{local:Kh}.
	\end{enumerate}
\end{theorem}

\begin{proof}[Proof of Theorem~\ref{thm:glueing:ALink:Kh}]
	By Theorem~\ref{thm:GlueingTheorem:Kh}, \(T_1\cup T_2\) is a A-link if and only if \(\HF(\Gamma_1,\Gamma_2)\) is \(\ZZ\)-thin. 
	By Theorem~\ref{thm:reduction:Kh}~\ref{enu:reduction:Kh:pairs:Thin}, 
	the latter is equivalent to \((C_1,C_2)=\Phi(\Gamma_1,\Gamma_2)\) being \(\ZZ\)-thin. 
	By Theorem~\ref{thm:glueing:ThinG:main}, this is equivalent to \(\ThinZZ(C_1)\cup\ThinZZ(C_2)=\QPI\) and for all \( s\in 
	\mirrorBdryALinkKh(T_1)
	\cap
	\BdryALinkKh(T_2)
	\), at least one of \(C_1\) and \(C_2\) is \(s\)-rational. 
	By Theorem~\ref{thm:reduction:Kh}~\ref{enu:reduction:Kh:s-rational}, the latter is equivalent to \(\Gamma_1\) and \(\Gamma_2\) satisfying \ref{local:Kh} for all \( s\in 
	\mirrorBdryALinkKh(T_1)
	\cap
	\BdryALinkKh(T_2)
	\). 
	Now use Theorem~\ref{thm:reduction:Kh}~\ref{enu:reduction:Kh:spaces} to conclude.  
\end{proof}

\begin{theorem}[Thin Gluing Theorem; Theorem~\ref{thm:gluing:Thin:intro}]\label{thm:glueing:Thin:Kh}
	Suppose \(T_1\) and \(T_2\) are two Conway tangle and at least one of them is not Khovanov exceptional. 
	Then \(T_1\cup T_2\) is Khovanov thin if and only if
	\begin{enumerate}
		\item 	\(
		\mirrorThinKh(T_1)
		\cup 
		\ThinKh(T_2)
		=
		\QPI
		\);
		and
		\item for every slope 
		\(s\in
		\mirrorBdryThinKh(T_1)
		\cap
		\BdryThinKh(T_2)
		\), the multicurves \(\Gamma_1\) and \(\Gamma_2\) satisfy~\ref{local:Kh}.
	\end{enumerate}
\end{theorem}

\begin{proof}
	The proof is identical to that of Theorem~\ref{thm:glueing:ALink:Kh}, noting that the additional condition about exceptional multicurves is now needed when applying Theorem~\ref{thm:glueing:ThinG:main}.
\end{proof}

Unlike the Heegaard Floer setting, we do not know whether the exceptionality assumption in Theorem~\ref{thm:glueing:Thin:HF} is required, since Khovanov exceptional tangles have not been observed.

\begin{corollary}[Corollaries~\ref{cor:one_direction:ALink:intro} and~\ref{cor:one_direction:Thin:intro}]\label{cor:one_direction:ThinG:Kh}
	For any Conway tangles \(T_1\) and \(T_2\),
	\begin{align*}
	\mirrorIntALinkKh(T_1)
	\cup 
	\IntALinkKh(T_2)
	=
	\QPI
	&\Rightarrow
	\text{\(L\) is a Khovanov A-link; and}
	\\
	\mirrorIntThinKh(T_1)
	\cup 
	\IntThinKh(T_2)
	=
	\QPI
	&\Rightarrow
	\text{\(L\) is Khovanov thin.}
	\end{align*}
\end{corollary}

\begin{proof}
	Let \(C_i=\Phi(\Gamma_i)\) for \(i=1,2\). By Theorem~\ref{thm:reduction:Kh}~\ref{enu:reduction:Kh:spaces}, 
	\(
	\mirrorIntALinkKh(T_1)
	\cup 
	\IntALinkKh(T_2)
	=
	\QPI
	\) implies that \(\IntThinZZ(C_1)\cup\IntThinZZ(C_2)=\QPI\). 
	By Corollary~\ref{cor:one_direction:main}, \((C_1,C_2)\) is \(\ZZ\)-thin, so by Theorem~\ref{thm:reduction:Kh}~\ref{enu:reduction:Kh:pairs:Thin}, this implies that \(\HF(\Gamma_1,\Gamma_2)\) is \(\ZZ\)-thin. Now conclude with Theorem~\ref{thm:GlueingTheorem:Kh}. 
	
	The second part follows from the same line of reasoning.
\end{proof}


\section{Examples}\label{sec:examples}
In order to place the theorems in this paper in context, we conclude with a collection of examples. Of note is the fact that the behaviour one encounters in practice is relatively tame, by comparison with the delicate casework seen in the proofs. In particular, if one chooses to focus on the invariants that are encountered in nature---for instance in the examples computed in \cite{KWZ-strong}---most of the forgoing material simplifies considerably. We will attempt to highlight this here, and with the reader who has skipped directly to this section from the introduction in mind, our aim is to present this material in a vaguely self-contained way. 

For simplicity, we will focus on Khovanov homology throughout this section; the analogous statements hold for link Floer homology as well. 
In fact, for many examples, the notions of thinness are independent of the homology theory and the field of coefficients. 
Using the programs \cite{JavaKh} and \cite{HFKcalc}, we checked that through 14-crossing knots the invariants $\Khr(K;\F_2)$, $\Khr(K;\Q)$, $\HFK(K;\F_2)$, and $\HFK(K;\F_3)$ are either simultaneously thin or simultaneously not thin.
However, Shumakovitch gave the following cautionary example \cite{shumakovitch2018torsion}; 
we thank Lukas Lewark for pointing it out to us. 

\begin{example}\label{exa:thinness_depends_on_coefficients:knots}
	The Khovanov homology of the knot \(16^n_{197566}\) in the knotscape knot table \cite{knotscape} is thin when computed with rational coefficients, but not over \(\F_2\). 
	Shumakovitch used unreduced Khovanov homology, but this statement is also true for reduced Khovanov homology; see the example \texttt{K\_16n197566} in \cite{tangle-atlas}. 
	Interestingly, knot Floer homology is thin over \(\F\); we checked this using Szabó's program \cite{HFKcalc}.
	We expand on these calculations in Example~\ref{exa:thinness_depends_on_coefficients:tangles} below.
\end{example}

Despite this example, the following question remains open:

\begin{question}\label{qst:main} Does the notion of thinness agree for Khovanov and Heegaard Floer theories when working with coefficients in $\Q$?\end{question}

\labellist \tiny
\pinlabel $\sKh_4(\infty)$ at 98 213  \pinlabel $\rKh_1(-2)$ at 122 73 
\pinlabel $\sKh_4(1)$ at 422 213  \pinlabel $\rKh_1(2)$ at 298 73 
	  		\endlabellist
\begin{figure}[ht]
\includegraphics[scale=0.75]{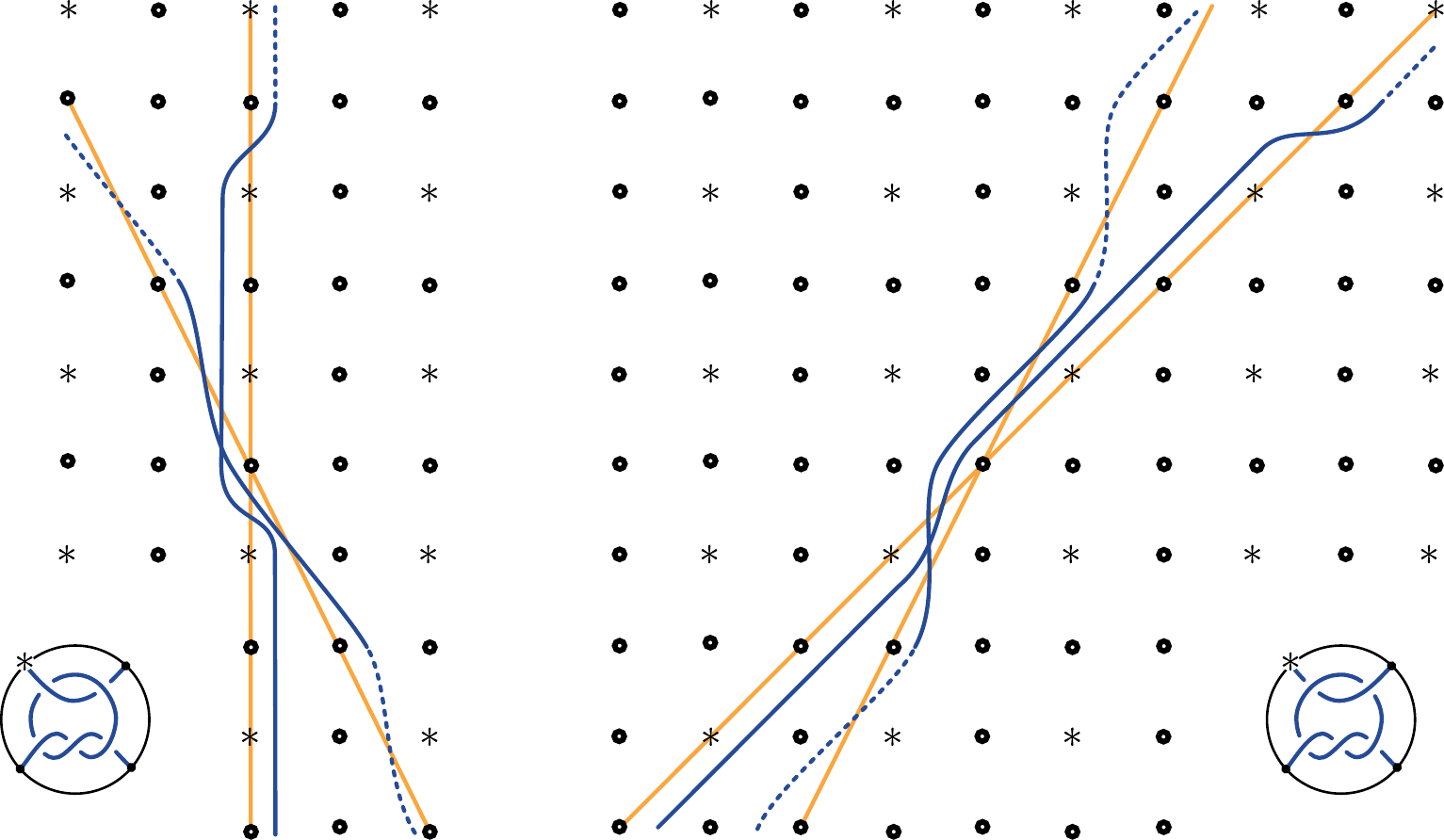}
\caption{The Khovanov invariant of a tangle as curves lifted to the cover $\R^2\smallsetminus\Z^2$. Notice that the tangles $P_{2,-3}$ and $P_{-2,-3}$ are related by twisting the lower endpoints; this is reflected in the plane shear taking one invariant to the other. As expected, both special and rational components (in the sense of Definition \ref{def:RationalVsSpecialKht}) appear.}
\label{fig:prelim}
\end{figure}

It is convenient to describe the Khovanov invariants of tangles in the planar cover $\R^2\smallsetminus\Z^2$ of the tangle boundary minus the tangle endpoints. One reason for this is the somewhat surprising fact, stated in Theorem \ref{thm:geography_of_Khr}, that for any Conway tangle $T$, all of the components are linear. This is illustrated in Figure \ref{fig:prelim}, which revisits Example \ref{exa:Khr:2m3pt}. This particular $(2,-3)$-pretzel tangle serves as a running example through this section.
Note that Theorem \ref{thm:Kh:Twisting} says that $\Khr$ commutes with the action of the mapping class group; this group is generated by a pair of plane shears on $\R^2\smallsetminus\Z^2$. 
As a result, the bottom braid move relating the pretzel tangles $P_{2,-3}$ and $P_{-2,-3}$ lifts to a linear transformation of the planar cover. For the class of tangles admitting an unknot closure, there is a sense in which the behaviour one sees is not more complicated than that observed in this single example; see \cite{KWZ-strong} for more. This is an ungraded statement, however---the grading information is subtle and important. 
 
\subsection{Rational tangles and two-bridge knots: conventions}\label{sub:rat} We begin by providing a cheat sheet of sorts in order to fix our conventions. The left-hand trefoil, expressed as the closure of the three-crossing rational tangle $Q_3$ by the trivial tangle $Q_0$, is shown in Figure~\ref{fig:prelude}. With this example we mean to highlight that there is a strong interplay between the Khovanov and Bar-Natan invariants of a given tangle. 
Indeed, while we have been working almost exclusively with $\Khr(T_1\cup T_2)\otimes V$ to this point (see Section~\ref{sec:review:Kh:gluing}), recall that $\Khr(T_1\cup T_2)$ can also be recovered by considering $\HF(\Khr(T_1^*),\BNr(T_2))$. To compute the Floer homology in the planar cover, it is sufficient to consider the \emph{preimage} of one multicurve, the \emph{lift} of the other multicurve, and then count intersections after pulling tight.  This strategy is used on the left of Figure~\ref{fig:prelude}: the preimage of $\Khr(Q_0)$ are the lines of slope $0$, and the lift of the invariant $\Khr(Q_3)$ is the line of slope $3$ missing the punctures, twice longer than the other line of slope 3 shown. The latter line, which passes through punctures, coincides with the lift of the Bar-Natan invariant $\BNr(Q_3)$. 
The fact that the trefoil is thin is well known; through the lens of our results, this is the fact that a line of slope 3 in the plane intersects a line of slope 0 once.  

\begin{figure}[t]
\includegraphics[scale=0.75]{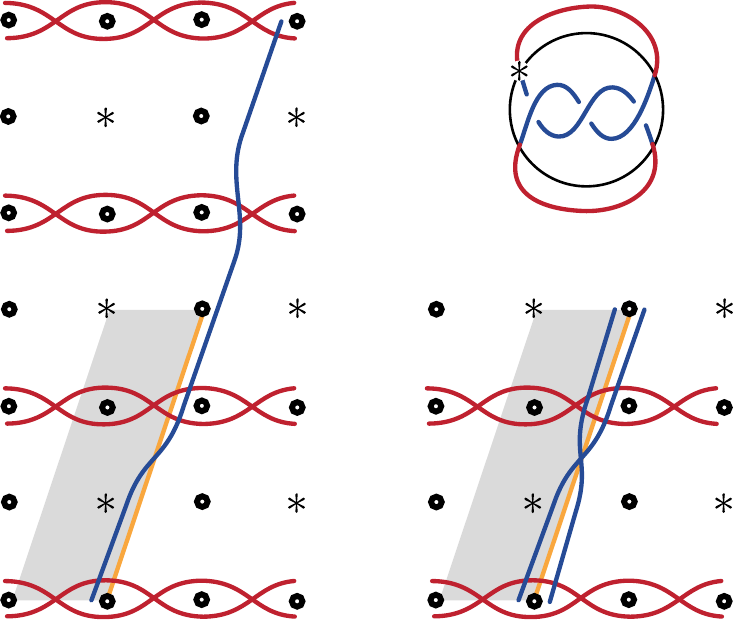}
\caption{A decomposition of the trefoil knot into the three-crossing rational tangle $Q_3$ and the trivial tangle \(Q_0\) (top right) and the corresponding computation of the reduced Khovanov homology of the trefoil knot in terms of Lagrangian Floer homology in the covering space (left):
$\Khr(Q_0\cup Q_3)\otimes V 
 \cong \HF(\Khr(Q_0^*),\Khr(Q_3)) = \F^6$ and $\Khr(Q_0^* \cup Q_3) \cong  \HF(\Khr(Q_0^*),\BNr(Q_3)) =\F^3$. 
A shorthand for this calculation is depicted on the bottom-right.
}
\label{fig:prelude}
\end{figure}

More generally, a central observation in this work is that the invariant of a rational tangle corresponds to/is controlled by a line of the appropriate rational slope. While this has come up repeatedly already, we review this basic fact here in order to make some conventions concrete and transparent; see Figure \ref{fig:rational-closures}. Rational fillings of the trivial tangle $Q_0$ are non-split two-bridge links, with the exception of the slope 0 rational filling, which is the two-component unlink. This unlink is not an A-link. However, non-split two-bridge links are alternating and hence thin by~\cite[Theorem~3.12]{Lee2005}, see also \cite[Theorem~1]{ManolescuOzsvath2008}. So we know that \(\Thin(Q_0)=\ALink(Q_0)=\QPI\smallsetminus\{0\}\).

\labellist 
\small
\pinlabel $\beta$ at 414 85
\tiny
\pinlabel $[2]$ at 22 2  \pinlabel $[3]$ at 22 82
	\pinlabel $[1,1,1]$ at 269 2  \pinlabel $[1,1,2]$ at 269 82
	  		\endlabellist
\begin{figure}[t]
\includegraphics[scale=0.75]{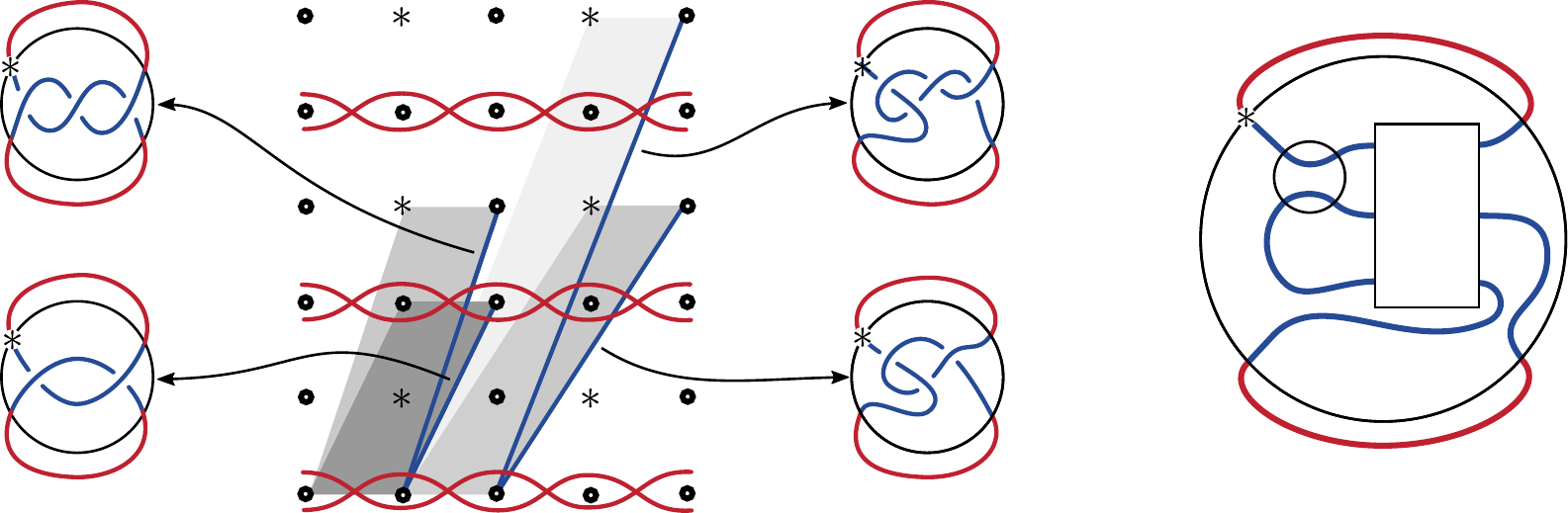}
\caption{Some closures of rational tangles associated with odd-length positive continued fractions, together with their Khovanov homology groups calculated via intersection in the plane according to the shorthand explained in Figure \ref{fig:prelude}. Note that $[3]=\nicefrac{3}{1}$ is the left-hand trefoil while $[1,1,1]=\nicefrac{3}{2}$ is the right-hand trefoil.  We recover the Khovanov homology for the Hopf link (dimension 2), both trefoils (dimension 3), and the figure-eight knot (dimension 5). In general, the alternating three-braid $\beta=\sigma_2^{-a_0}\sigma_1^{a_1}\sigma_2^{-a_2}\sigma_1^{a_3}\cdots\sigma_2^{-a_n}$, inserted into the tangle as indicated on the right, gives rise to the two-bridge link associated with the odd-length continued fraction $[a_0,a_1,\ldots,a_{n}]$.}
\label{fig:rational-closures}
\end{figure}

We now consider this in more detail, making the choice to vary the rational tangle and to fix the particular closure---the numerator closure $Q_0$ as in Figure \ref{fig:rational-closures}. 
Given a positive, reduced rational number $\nicefrac{p}{q}\ge1$, there is a unique non-split two-bridge link associated with it. To construct it, one chooses an odd-length positive continued fraction expansion
$$
\nicefrac{p}{q}
=
[a_0,a_1,\ldots,a_n]
=
a_n+\frac{1}{a_{n-1}+\frac{1}{\dots+\frac{1}{a_0}}}
$$
where $a_i>0$ and $n>0$ is even. Since $[a_0,\ldots,a_{n}] = [1,a_0-1,\ldots,a_{n}]$, such a continued fraction expansion always exists.  With this choice in hand, Figure~\ref{fig:rational-closures} illustrates some examples of two-bridge knots obtained as the numerator closures of rational tangles. 
Each rational number is associated with a slope in the plane, and the intersection of the corresponding line with the preimage of $\Khr(Q_0)$ in the plane calculates the Khovanov homology of the associated two-bridge knot. We have shown the slopes $\nicefrac{3}{2}<2<\nicefrac{5}{2}<3$ in the plane to illustrate these thin fillings.  The fact that the numerator $p$ calculates the determinant {\it and} the dimension of the reduced Khovanov homology is a helpful check for these examples. It can be instructive to consider the base-length 1 parallelograms determined by the Khovanov invariants in each case; the added twists dictated by the continued fraction correspond in a natural way to the plane shears moving between any two parallelograms. Moreover, with the above conventions in place, the area of the parallelogram agrees with the determinant of the link.

\subsection{An aside on alternating fillings} In these first examples, thinness was deduced from the stronger statement that all tangle fillings in question were alternating. In general, we can say a little more. 
For terminology, we say a tangle diagram is alternating if the crossings alternate between under and over crossings as one travels along the tangle, regardless of where one starts. We call a tangle diagram connected if the underlying planar graph is connected. 

\begin{proposition}\label{prp:alternating}
	For any tangle \(T\) admitting a connected alternating diagram, the space of thin fillings (relative to a choice of alternating tangle diagram) contains either \([\infty,0]\) or \([0,\infty]\). Moreover, these thin fillings are in fact alternating fillings. 
\end{proposition}

It is interesting to compare Proposition~\ref{prp:alternating} to a result of Bar-Natan and Burgos-Soto \cite[Theorem~1]{B-N-B-S2014}. 
When restricted to Conway tangles, their result says that the vertical (horizontal) intersection points of \(\Khr(T)\) have the same \(\delta\)-grading \(\delV\) (\(\delH\)), and that \(\delV\) and \(\delH\) differ by \(\pm\tfrac{1}{2}\). 
The fact that both \(\delV\) and \(\delH\) are constant implies that \(\Khr(T)\) neither contains any special component of slope \(0\) nor any special component of slope  \(\infty\).
Indeed, observe that special components of slope \(0\) contain two pairs of generators whose \(\delta\)-gradings are equal to \(\delV-\tfrac{1}{2}\) and  \(\delV+\tfrac{1}{2}\), respectively. 
Similarly, any special component of slope \(\infty\) contains four generators whose \(\delta\)-gradings are equal to \(\delH-\tfrac{1}{2}\) and  \(\delH+\tfrac{1}{2}\).
In fact, \(\ThinKh(T)\) does not contain any rational component of slope \(0\) or \(\infty\) either, since \(0,\infty\in\ThinKh(T)\). 
This implies the following strengthening of Proposition~\ref{prp:alternating}:

\begin{corollary}
	For any tangle \(T\) admitting a connected alternating diagram, \(\ThinKh(T)\) contains an open interval containing both \(\infty\) and \(0\).
	\qed
\end{corollary}

For \(\HFT\), a similar result seems plausible. 
In fact, the corresponding statement about the horizontal and vertical \(\delta\)-grading also holds for \(\HFT(T)\), which follows from the Generalised Clock Theorem \cite{PrizeEssay}. 
However, \(\HFT\) may contain rational components of slope \(0\) or \(\infty\) that carry inhibited local systems, see Definition~\ref{def:inhibited} and Remark~\ref{rem:inhibited}. 

\begin{proof}[Proof of Proposition~\ref{prp:alternating}]
	Consider a connected alternating diagram \(D\) of the tangle \(T\). 
	The two closures of \(D\) representing the links \(T(0)\) and \(T(\infty)\) are alternating diagrams and, since they are non-split, the links \(T(0)\) and \(T(\infty)\) are non-split \cite[Theorem~4.2]{Lickorish-IntroToKnotTheory}. 
	Similarly, either the \(+1\)- or \(-1\)-closure of \(D\) is an alternating diagram, so at least one of \(T(+1)\) and \(T(-1)\) is an alternating non-split link. 
	Any alternating non-split link has thin Khovanov homology. 
	So \(\Thin(T)\) contains \(0\), \(\infty\) and either \(+1\) or \(-1\). By Theorem~\ref{thm:charactisation:Thin:intro}, it is therefore an interval containing either \([0,\infty]\) or \([\infty,0]\). 

One can now check directly that \(T(s)\) is alternating either for all positive or for all negative \(s\in\QPI\). Indeed, without loss of generality, suppose that $T(+1)$ is alternating. Then according to our conventions $T(n)$ is an alternating diagram for all $n\ge 0$. More generally, we simply observe that choosing an odd-length continued fraction representing a positive rational number $s$ (compare Figure \ref{fig:rational-closures}), the closure $T(s)$ is an alternating diagram. 
\end{proof}

One can easily check the proposition on the class of two-bridge links, for example, by starting from a rational tangle diagram with one crossing.

\subsection{A more instructive example.} 
Perhaps the simplest non-rational tangle without closed components is the $(2,-3)$-pretzel tangle \(P_{2,-3}\). The Khovanov invariant associated with this tangle is given in Figure \ref{fig:prelim} and revisited in Figure \ref{fig:first-pretzel}. It consists of a special component (the curve of slope \(\infty\)) and a rational component (the curve of slope \(-2\)). 

\labellist \tiny
	\pinlabel $\simeq$ at 190 27  
	\pinlabel $\simeq$ at 190 112
	\pinlabel $\simeq$ at 190 125
	\pinlabel $\simeq$ at 190 212
	 		\endlabellist
\begin{figure}[t]
\includegraphics[scale=0.75]{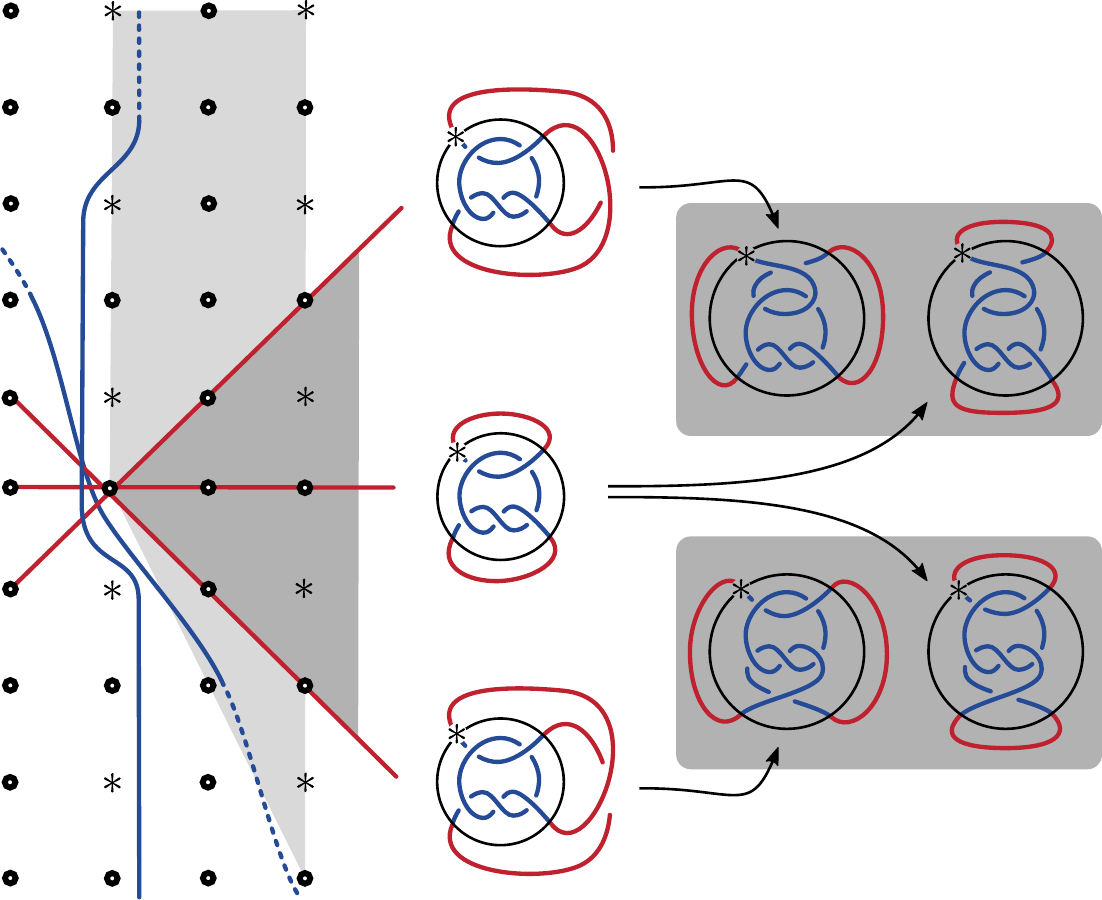}
\caption{The invariant for $P_{2,-3}$ illustrating that $\ThinKh(P_{2,-3})=\protect\ALinkKh(P_{2,-3})=(-2,\infty]$. The fillings $-1$, $0$, and $+1$ have been indicated, each of which is an alternating link. Notice that, after an appropriate isotopy fixing the tangle boundary on each of the links in the shaded boxes, the closures we have identified are realized as closures of alternating tangles. As a result, $[-1,1]\subset (-2,\infty]$ gives a subset of alternating fillings according to Proposition \ref{prp:alternating}. 
}
\label{fig:first-pretzel}
\end{figure}

To compute the spaces of thin and A-link fillings of this tangle, observe that the 0-rational filling \(P_{2,-3}(0)\) is a connected sum of the trefoil knot and the Hopf link. 
So this filling is thin and an A-link.
(Alternatively, this follows from the fact that the horizontal \(\delta\)-gradings \(\delH\) of the two components of \(\Khr(P_{2,-3})\) agree; see Table~\ref{tab:prime_tangles} on page~\pageref{tab:prime_tangles}.) 
Having found one thin filling whose slope does not agree with one of the supporting slopes of \(\Khr(P_{2,-3})\), we now know that \(\ThinKh(P_{2,-3})\) and \(\ALinkKh(P_{2,-3})\) are intervals containing \(0\) with endpoints \(-2\) and \(\infty\) and we know that those intervals agree. 
Since there is a rational component of \(\Khr(P_{2,-3})\) of slope \(-2\), the endpoint \(-2\) is not contained in this interval; for the opposite reason, \(\infty\) \emph{is} contained in the interval. 
In summary, \(\ThinKh(P_{2,-3})=\ALinkKh(P_{2,-3})=(-2,\infty]\); see Figure \ref{fig:first-pretzel}. 
As a check, one might consider the knot $P_{2,-3}(-3)$: This pretzel knot is the knot $8_{19}$ in the Rolfsen knot table, which is the first non-thin knot encountered in enumerated examples. 

The Heegaard Floer invariant \(\HFT(P_{2,-3})\) consists of a single rational component (with trivial local system) of slope \(-2\) and a conjugate pair of special components of slope \(\infty\); see Table~\ref{tab:prime_tangles}. 
Repeating the same arguments as above, we see that \(\ThinHF(P_{2,-3})=\ALinkHF(P_{2,-3})=(-2,\infty]\).

We can now revisit the observations made about alternating fillings in this setting: As indicated in Figure \ref{fig:first-pretzel}, there is a sequence of three alternating tangle fillings given by $-1$, $0$, and $+1$. So, a transformation of the plane taking either of $\{1,0\}$ or $\{-1,0\}$ to $\{\infty, 0\}$ (compare Figure \ref{fig:prelim}) together with an application of Proposition \ref{prp:alternating} gives two infinite collections of alternating fillings. Expressed in the framing shown, there is a subset of alternating fillings $[-1,1]\subset (-2,\infty]=\Thin(P_{2,-3})$. More generally, we remark that  the  subset $[-1,\infty]\subset (-2,\infty]$  gives rise to an infinite family of quasi-alternating fillings (this is established in \cite{Watson2011}). Of course, adding a single positive twist to the top of the \((2,-3)\)-pretzel tangle yields the \((-2,-3)\)-pretzel tangle \(P_{-2,-3}\). The invariant \(\Khr(P_{-2,-3})\) is obtained from a plane shear as shown in Figure \ref{fig:prelim}, so that \(\Thin(P_{-2,-3})=(2,1]\subset\QPI\). 


\subsection{Bar-Natan curves} In the context of Khovanov invariants, thinness can also be defined in terms of Bar-Natan homology, a generalization of Khovanov homology taking the form of a bigraded \(\field[H]\)-module. Recall that for a (pointed) link $L$ with \(|L|\) components, we have that $\BNr(L)\cong \field[H]^{2^{|L|-1}}\oplus H$-torsion. If \(L\) is a knot,  the quantum grading of the term $\field[H]\subset \BNr(K)$ agrees with Rasmussen's $s$-invariant over \(\field\). In this subsection we make some general observations that hold over any field $\field$. 

\begin{definition}
	We call a \(\delta\)-graded \(\field[H]\)-module \(M\) thin 
	if the \(H\)-torsion part of \(M\) and a homogeneous generating set of the free part of \(M\) are supported in a single \(\delta\)-grading.
\end{definition}
It suffices to focus on reduced Khovanov homology, according to the following observation. 
\begin{proposition}\label{prop:BNr_thin_equals_Khr_thin}
	For any pointed link \(L\) and field \(\field\), \(\BNr(L;\field)\) is thin iff \(\Khr(L;\field)\) is thin. 
\end{proposition}

\begin{proof}
	On the level of chain complexes \(\CBNr(L;\field)\) determines \(\CKhr(L;\field)\) via a mapping cone formula:
	\[
	\CKhr(L;\field)\simeq
	\Big[\begin{tikzcd}
	q^{-2}h^{-1}\delta^0\CBNr(L;\field) 
	\arrow{r}{H}
	&
	q^{0}h^{0}\delta^0\CBNr(L;\field)
	\end{tikzcd}\Big]
	\]
	Consequently, \(\Khr(L;\field)\) is the homology of some map
	\[
	\begin{tikzcd}
	q^{-2}h^{-1}\delta^0\BNr(L;\field) 
	\arrow{r}{}
	&
	q^{0}h^{0}\delta^0\BNr(L;\field)
	\end{tikzcd}
	\]
	sending each generator of a free summand to itself times \(H\). Therefore, if \(\BNr(L;\field)\) is thin, so is \(\Khr(L;\field)\). 
	
	Conversely, suppose \(\Khr(L;\field)\) is thin. Recall that \(\Khr(L;\field)\) can be promoted to a type~D structure \(\Khr(L;\field)^{\field[H]}\) by connecting pairs of generators by differentials labelled by some powers of \(H\) \cite[Sections~3.2 and~3.3, in particular Proposition~3.6]{KWZ}. Since \(\Khr(L;\field)\) is assumed to be thin, the only possible labels are \(H\). Together with 
	\[
	\CBNr(L;\field)_{\field[H]}
	\simeq
	\Khr(L;\field)^{\field[H]}\boxtimes \prescript{}{\field[H]}{\field[H]}_{\field[H]}
	\] 
	establishes the result.
\end{proof}

We can extract the following from the final steps of the proof:  

\begin{corollary} If $L$ is a thin link then the torsion part of the $\field[H]$-module \(\BNr(L;\field)\) agrees with $\ker(H)$. In particular, $2\operatorname{rk}(\ker H)+ 2^{|L|-1} = \det(L)$. 
	\qed
\end{corollary}

As a result, it is possible to define A-links in terms of Bar-Natan homology. 

\begin{definition} Let $N$ be the dimension of the torsion part of the $\field[H]$-module \(\BNr(L;\field)\) as a $\field$-vector space. Then $L$ is an A-link whenever $2N+2^{|L|-1}=\det(L)$.
\end{definition}

Of course, assuming full support in the sense of Definition \ref{def:full}, we have that (Bar-Natan) A-links are (Bar-Natan) thin links. One can also check that this definition of A-link agrees with the definition given in the introduction asking that the total dimension of the reduced Khovanov homology agree with the determinant of the link.

The reduced Bar-Natan homology \(\BNr(T)\) of a Conway tangle \(T\) satisfies a gluing theorem similar to the one for \(\Khr\) \cite[Theorem~7.2]{KWZ}:
\[
\BNr(T_1\cup T_2)
\cong
\HF\left(\BNr(T_1^*),\BNr(T_2)\right)
\]
Here, the right-hand side denotes the \emph{wrapped} Lagrangian Floer homology of the two tangle invariants. 
As Example~\ref{exa:Khr:2m3pt} illustrates, and as we reiterate here, the components of the multicurve \(\BNr(T)\) need not be linear. 
If the multicurve consists of just a single component, this allows us to compute the space of thin fillings very easily from the space of tangent slopes. 

We illustrate how this is done in the example of the curve \(\BNr(P_{2,-3})\) for the \((2,-3)\)-pretzel tangle \(P_{2,-3}\) from Figure~\ref{fig:Kh:example}. 
A lift \(\tilde{\gamma}\) of this curve to \(\PuncturedPlane\) is redrawn in Figure~\ref{fig:BN-Kh}. 
Consider the family of ``\(\varepsilon\)-peg-board representatives'' \(\tilde{\gamma}_\varepsilon\) of \(\tilde{\gamma}\), ie representatives of the homotopy class of \(\tilde{\gamma}\) which have minimal length among all representatives of distance \(\varepsilon\in(0,\nicefrac{1}{2})\) to all punctures in \(\PuncturedPlane\) except the two punctures at the ends of \(\tilde{\gamma}\). 
Following \cite{HRW}, the intuition behind this definition is to think of the punctures of \(\PuncturedPlane\) as pegs of radii \(\varepsilon\) and to imagine pulling the curve \(\tilde{\gamma}\) ``tight'', like a rubber band. 
If \(\tau_\varepsilon\) denotes the set of rational tangent slopes of the curve \(\tilde{\gamma}_\varepsilon\) then the interior of \(\ALinkKh(P_{2,-3})\) is equal to the complement of~\,\(\bigcap\tau_{\varepsilon}\).
Here, the obstruction to being an A-link is the existence of bigons near the points where the limit curve of \(\tilde{\gamma}_\varepsilon\) as \(\varepsilon\rightarrow0\) (the ``singular peg-board representative'') changes its slope; this is illustrated in Figure~\ref{fig:changing_direction}.

\labellist \small
	\pinlabel $\F$ at 252 27 \pinlabel $\F$ at 252 40 \pinlabel $\F$ at 252 53 \pinlabel $\F$ at 252 66 \pinlabel $\F$ at 252 92
	 \pinlabel $\F$ at 457 40 \pinlabel $\F$ at 457 66 \pinlabel $\F[H]$ at 460 92
	 \pinlabel $\delta$ at 274 19  \pinlabel $q$ at 242 113
	 \pinlabel $\delta$ at 476 19 \pinlabel $q$ at 444 113.5
		\endlabellist
\begin{figure}[t]
\includegraphics[scale=0.75]{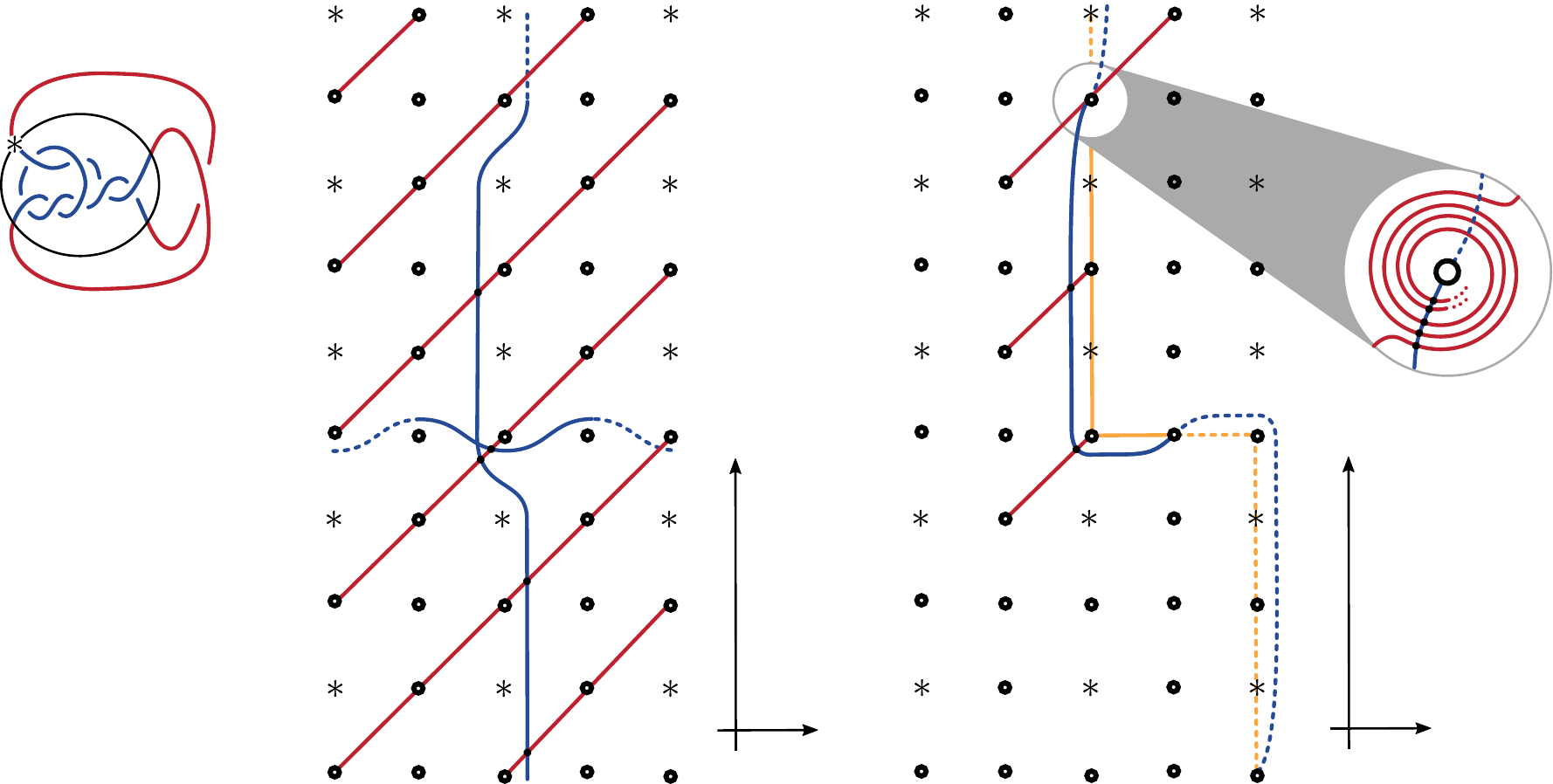}
\caption{The reduced Khovanov (left) and Bar-Natan (right) invariants associated with the cinqfoil, a thin knot, obtained as the closure of the (reframed) $(2,-3)$-pretzel tangle. Note that the framing given here is such that the thin filling interval is $(0,\infty]$, as determined by the pulled-tight curve (in yellow) shown for the Bar-Natan invariant on the right.}
\label{fig:BN-Kh}
\end{figure}

This bears a strong resemblance to how the space of L-space fillings \(\Lspace(M)\) of a three-manifold \(M\) with torus boundary is characterized via the immersed curve invariant \(\HFhat(M)\) due to Hanselman, Rasmussen, and the second author. There, it is shown that the interior of \(\Lspace(M)\) is equal to the complement of the space of rational tangent slopes of the singular peg-board representative of \(\HFhat(M)\) \cite[Theorem~54]{HRW}. 

\begin{figure}
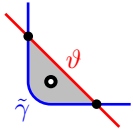

	\centering
	\(\ThinnessObstructionBNr\)
	\caption{The rational filling of \(\textcolor{blue}{\tilde{\gamma}}\) along the slope of the straight line \(\textcolor{red}{\vartheta}\) is not thin, since the \(\delta\)-gradings of the two intersection points that are connected by the shaded bigon differ by 1.}\label{fig:changing_direction}
\end{figure}

\subsection{A-links and L-spaces} \label{sec:examples:def_Lspaces}
Given a link $L$, let $\boldsymbol{\Sigma}_L$ denote the two-fold branched cover of $S^3$ with branch set $L$. (Similarly, we will use $\boldsymbol{\Sigma}_T$ to denote the two-fold branched cover of $B^3$ with branch set the tangle arcs of $T$.)
Owing to the fact that $\Khr(L^*)$ arises as the $E_2$-page of a spectral sequence computing $\HFhat(\boldsymbol{\Sigma}_L)$ \cite{OzsvathSzabo2005}, one might naturally wonder about the relationship between L-spaces and A-links. In particular, one expects an interplay between $\ALinkKh(T)$ and $\Lspace(\boldsymbol{\Sigma}_T)$. Before exploring this relationship further, we make some general comments about the definition of L-spaces.  In Section~\ref{sec:intro}, L-spaces were  introduced as solutions of the identity
\[
\dim\HFhat(Y) = \chi\HFhat(Y).
\]
Usually, L-spaces are defined in terms of the following two conditions:
\begin{enumerate}[label=(\alph*)]
	\item being a rational homology sphere, that is $b_1(Y)=0$; and
	\item satisfying \(\dim\HFhat(Y) = |H_1(Y;\Z)|\).
\end{enumerate}
Coefficients are often chosen to be in \(\fieldTwoElements\); we do the same and suppress \(\fieldTwoElements\) in our notation.
We observe that these two definitions are equivalent. For rational homology spheres  we have the equality \(\chi\HFhat(Y)=|H_1(Y;\Z)|\) and if \(b_1(Y)>0\) then \(\chi\HFhat(Y)=0\) \cite{OzsvathSzabo2004}. So it suffices to show:  

\begin{proposition}\label{prop:HF_does_not_vanish}
	\(\HFhat(Y)\) does not vanish for any three-manifold \(Y\).
\end{proposition}

We are not aware of a reference for this fact in the literature; Jake Rasmussen suggested the following argument. 

\begin{proof}[Proof of Proposition~\ref{prop:HF_does_not_vanish}]
	By the definition of \(\HF^\infty(Y)\) as the homology of \(\CFhat(Y)\otimes\F[U,U^{-1}]\) with higher
	differentials, there exists a spectral sequence from \(\HFhat(Y)\otimes\F[U,U^{-1}]\) to \(\HF^\infty(Y)\). 
	Therefore, it suffices to show that \(\HF^\infty(Y)\) does not vanish. 
	Lidman computed these groups for all closed orientable three-manifolds \cite[Theorem~1.1]{LidmanHFinfty}; compare \cite[Conjecture~4.10]{OSHFinfty}. He showed that for any torsion \(\Spinc\)-structure \(\mathfrak{s}\), 
	one can write \(\HF^\infty(Y,\mathfrak{s})\) as the homology of a chain
	complex whose underlying chain module is equal to 
	\[
	\Lambda^*(H^1(Y;\Z))\otimes\fieldTwoElements[U,U^{-1}]
	\]	
	and with differential of the form
	\[
	\Lambda^i(H^1(Y;\Z))\otimes U^j 
	\rightarrow 
	\Lambda^{i-3}(H^1(Y;\Z))\otimes U^{j-1}.
	\]
	(Torsion \(\Spinc\)-structures always exist: it suffices to recall that isomorphism classes of oriented plane fields on a closed and oriented three-manifold are determined by elements of $H^2(Y;\Z)$, and choose a plane-field on $Y$ with vanishing Euler class.) 
	In particular, the quotient \(Q\) obtained from \(\HF^\infty(Y,\mathfrak{s})\) by
	setting \(U=1\) is the homology of a chain complex whose underlying chain
	module is
	\[
	\Lambda^*(H^1(Y;\Z))\otimes\fieldTwoElements
	\]	
	and whose differential lowers the grading of the exterior product by 3. The
	Euler characteristic of this complex is
	\[
	\sum_i x^i \dim
	\left(
		\Lambda^i(H^1(Y;\Z))\otimes\fieldTwoElements
	\right)\in R\coloneqq\Z[x]/(x^3=-1).
	\]
	Note that this value remains invariant under taking
	homology. So the Euler characteristic of the quotient \(Q\) is equal to the
	Euler characteristic of 
	\(\Lambda^*(H^1(Y;\Z))\otimes\fieldTwoElements\), which is equal to
	\[
	(1+x)^a	\in R \]
	where $a = \dim(H^1(Y;\Z)\otimes\fieldTwoElements)$. This element is non-zero, which can be seen by embedding \(R\) into the complex plane. So \(Q\) is non-zero, and so is \(\HF^\infty(Y,\mathfrak{s})\).
\end{proof}

In the introduction, we pointed out a close relationship between Khovanov A-links and L-spaces:

\begin{theorem}\label{thm:ALink_implies_Lspace}
	If \(L\) is a Khovanov A-link then the two-fold branched cover \(\boldsymbol{\Sigma}_L\) is an L-space.
\end{theorem}
\begin{proof}
	The Ozsváth--Szabó spectral sequence \cite{OzsvathSzabo2005} from \(\Khr(L^*)\) to \(\HFhat(\boldsymbol{\Sigma}_L)\) implies that 
	\[
	\dim\Khr(L)=\dim\Khr(L^*) \geq \dim\HFhat(\boldsymbol{\Sigma}_L)\geq |H_1(\boldsymbol{\Sigma}_L;\Z)|=\det(L)
	\] 
	so the claim follows from the fact that A-links satisfy \(\dim\Khr(L)=\det(L)\).  
\end{proof}

\begin{corollary}\label{cor:ALink_implies_Lspace}
	For any Conway tangle \(T\), \(\ALinkKh(T)\subseteq\Lspace(\boldsymbol{\Sigma}_T)\). \qed
\end{corollary}

\labellist \tiny
	\pinlabel $\lambda$ at 55 15 \pinlabel $\mu$ at 102 53
	 \pinlabel $(4,\infty]$ at 267 250 \pinlabel $[1,\infty]$ at 320 215
	 \pinlabel $\lambda$ at 347 116 \pinlabel $\mu$ at 318 150
		\endlabellist
\begin{figure}[b]
\includegraphics[scale=0.75]{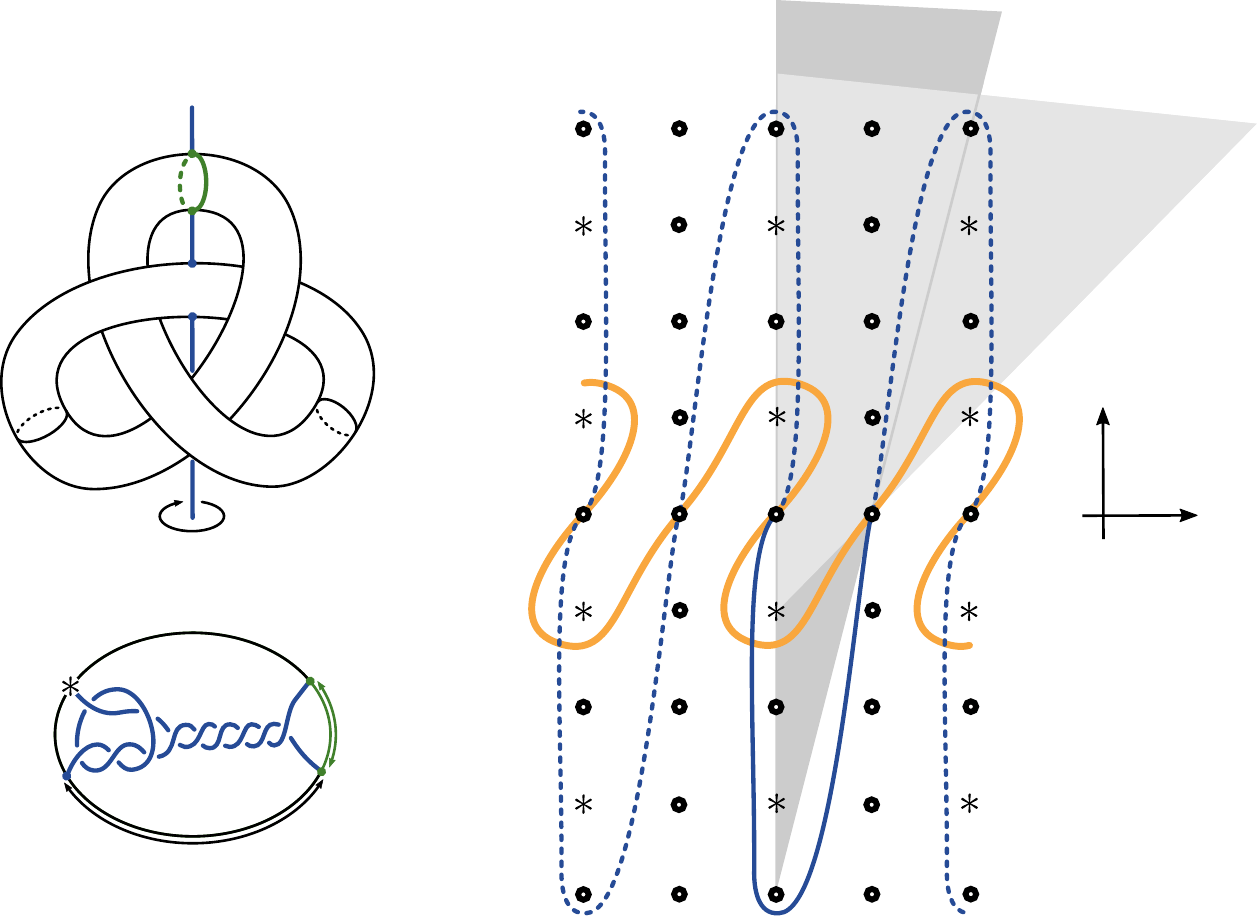}
\caption{Comparing the invariant $\BNr(P^\lambda_{2,-3})$ with the invariant $\HFhat(M)$, where $M$ is the complement of the right-hand trefoil. Note that $M$ is homeomorphic to the two-fold branched cover of $P^\lambda_{2,-3}$; the framing is such that the Seifert longitude descends to the arc labeled $\lambda$ and the meridian descends to the arc labeled $\mu$. }
\label{fig:L-space-int}
\end{figure}

When $\Lspace(\boldsymbol{\Sigma}_L)$ is a closed interval, this inclusion appears to be strict, in general. For instance, consider the now-familiar example of the pretzel tangle $P_{2,-3}$. The two-fold branched cover of this tangle is homeomorphic to the complement of the right-hand trefoil; the Seifert structure on this knot complement (two Seifert fibred solid tori glued along an essential annulus) is encoded by the sum of rational tangles generating this pretzel. (This is described in more detail and exploited in \cite{Watson2012}, for example.) These observations are collected in Figure~\ref{fig:L-space-int}, together with the Bar-Natan curve invariant and the curve corresponding to $\HFhat(M)$, where $M=\boldsymbol{\Sigma}_{P_{2,-3}}$ is the complement of the right-hand trefoil. The important thing to check, which accounts for our change of framing on the tangle, is that the $0$-filling of $P_{2,-3}$ coincides with the $+6$-surgery on the right-hand trefoil. 
Let $P^{\lambda}_{2,-3}$ denote the reframed tangle, that is, the tangle $P_{2,-3}$ with the six additional half-twists, so that the $0$-closure of the tangle is the branch set for $0$-surgery on the trefoil. 
We have shown: \[\ALinkKh(P^{\lambda}_{2,-3})=(4,\infty]\subset [1,\infty] = \Lspace(\boldsymbol{\Sigma}_{P^{\lambda}_{2,-3}})\]

This example fits into a simple infinite family, observing that the $(2,2n+1)$ torus knots (for integers $n>0$) have complements that branch double cover an infinite family of tangles. Denote the former by $T_{(2,2n+1)}$ and the latter by $T_n$, so that $T_1$ agrees with $P_{2,-3}$ (appropriately reframed); see Figure \ref{fig:Infinite-family}. Since the Seifert  genus of $T_{(2,2n+1)}$ is $n$, we compute: 
 \[\ALinkKh(T_n)=(4n,\infty]\subset [2n-1,\infty] = \Lspace(\boldsymbol{\Sigma}_{T_n})= \Lspace(S^3\smallsetminus\nu(T_{(2,2n+1)}))\]

\labellist 
	\pinlabel \rotatebox{90}{$\underbrace{\phantom{AAAAAA}}$} at 75 102
\pinlabel \rotatebox{270}{$\underbrace{\phantom{AAAAAAAAAAAAAAAAAAAA}}$} at 247 110
\small 
\pinlabel $\vdots$ at 72 89
\pinlabel $\BNr(T_n)\!=\!$ at 205 112  
\pinlabel {$\oplus\, \frac{n-1}{2}$ copies of $\sKh_{4}(\infty)$ for $n$ odd} at 390 163
\pinlabel {$\oplus\, \frac{n}{2}$ copies of $\sKh_{4}(\infty)$ for $n$ even} at 385 30
\tiny
\pinlabel $n$ at 89 102
	\pinlabel $\lambda+4n\mu$ at 82 35 
		\endlabellist
\begin{figure}[t]
\hspace*{-5cm}
\includegraphics[scale=0.75]{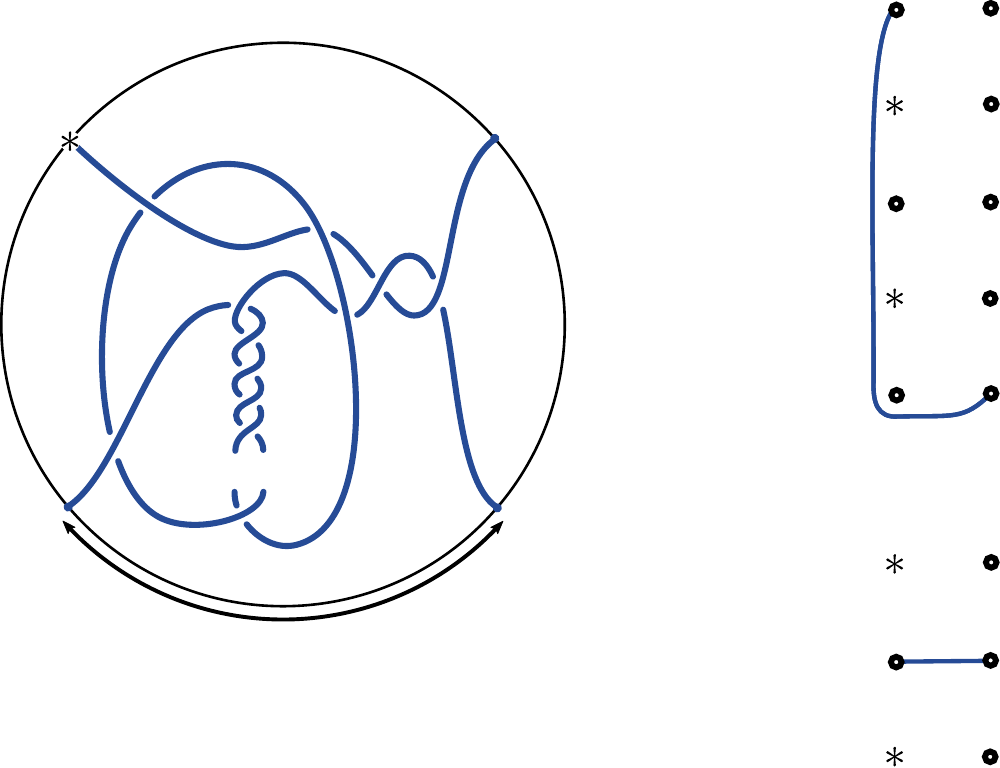}
\caption{A tangle $T_n$ whose two-fold branched cover is the exterior of the torus knot $T_{(2,2n+1)}$. Note that in the case $n=1$, we recover the right-hand trefoil. The image of the slope $\lambda+4n\mu$ descends to the arc indicated on the tangle boundary. 
}
\label{fig:Infinite-family}
\end{figure}

In particular, the interval of L-space fillings on a given knot (with fixed strong inversion) that do not branch over a thin link can be made arbitrarily large---it is $[2n-1,4n]$, for these examples. 
A compelling pattern emerges, and one might reasonably ask about the relationship between the curves $\BNr(T)$ and $\HFhat(\boldsymbol{\Sigma}_T)$ in general; see also \cite[Section 7]{HRW2} for another point of view. 

\subsection{Other manifolds admitting a strong inversion.} The principle exploited above can be thought of as the Montesinos trick: Given a strongly invertible knot $K$, there is an involution on the complement $M=S^3\setminus K$ with quotient a tangle $T$. (This is the idea behind the enumeration of tangles given in \cite{KWZ-strong}.) This tangle will always have the property of being {\it cap trivial}, that is, the $\infty$-filling of $T$ is unknotted.  For example, we saw in Subsection \ref{sub:rat} that $\boldsymbol{\Sigma}_{Q_0}$ is the complement of the trivial knot, and each non-zero filling gives a lens space, which branch double-covers the given two-bridge knot. So as a result, \(\Lspace(\boldsymbol{\Sigma}_{Q_0})=\QPI\smallsetminus\{0\}\).

We have also seen that the exterior of the right-hand trefoil is the two-fold branched cover of $P_{2,-3}$. This same trick applies to any knot admitting a strong inversion. Here is another example: The exterior of the figure-eight knot is the two-fold branched cover of a tangle which we denote by $T_{4_1}$. (In fact, there are two strong inversions on this knot, but in this case changing the choice of one for the other results in the mirror image of the tangle.) The rational filling along slope \(\infty\) results in an unknot, by construction, so this is a thin filling. However, in both the Heegaard Floer and the Khovanov setting, this is the only A-link filling. This is because both \(\Khr(T_{4_1})\) and \(\HFT(T_{4_1})\) contain special components of slope \(\infty\) in adjacent \(\delta\)-gradings. It is also remarkable that the two special components of \(\Khr(T_{4_1})\) correspond to the two conjugate pairs of special curves in \(\HFT(T_{4_1})\) and the rational components of the invariants have the same slope. On the Khovanov side, the lack of thin-fillings is consistent:  The cover $\boldsymbol{\Sigma}_{T_{4_1}}$ is homeomorphic to the exterior of the figure-eight knot, which has no L-space fillings other than the trivial filling. 
	
Now consider the pretzel tangle \(P_{2,-2}\); it is an instructive exercise to check that $\boldsymbol{\Sigma}_{P_{2,-2}}$ is not the exterior of a knot in $S^3$. Indeed, this example is not cap trivial; the cover $\boldsymbol{\Sigma}_{P_{2,-2}}$ is a Seifert fibred space known as the twisted $I$-bundle over the Klein bottle. It can be realized as the complement of a knot in $S^2\times S^1$. This manifold belongs to a class of manifolds known as Heegaard Floer homology solid tori, which enjoy the property that all fillings, other than the rational longitude filling, are L-spaces. That is, $\Lspace(\boldsymbol{\Sigma}_{P_{2,-2}}) =\QPI\setminus\{\infty\}$. From the perspective of tangle invariants this example is quite interesting, because the Heegaard Floer invariant contains a pair of special components that do not correspond to a special component in \(\Khr(P_{2,-2})\). For the space of A-link and thin rational fillings, this additional pair of special curves has no consequence; all spaces are equal to \((-2,2)\) and can be computed from the tangle invariants \(\Khr(P_{2,-2})\) and \(\HFT(P_{2,-2})\) (following the same strategy as for $P_{2,-3}$). In particular:  
	\[\ALinkKh(P_{2,-2})=(-2,2)\subset\QPI\setminus\{\infty\} = \Lspace(\boldsymbol{\Sigma}_{P_{2,-2}})\]

\begin{table}[p]
	\centering
	\begin{tabular}{rcccc}
		\toprule
		& 
		\includegraphics[height=3.3cm]{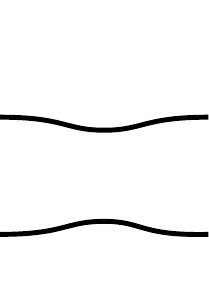}
		&
		\includegraphics[height=3.3cm]{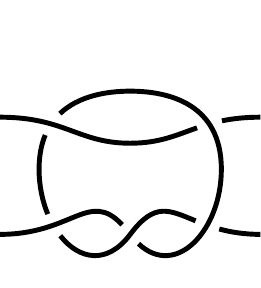}
		&
		\includegraphics[height=3.3cm]{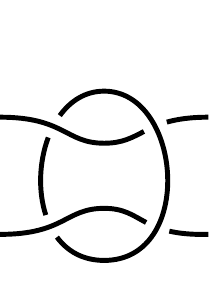}
		&
		\includegraphics[height=3.3cm]{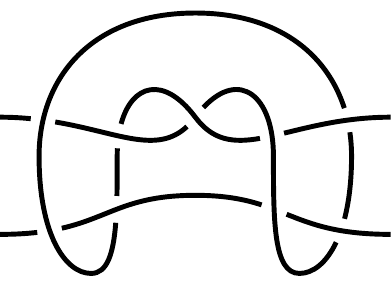}
		\\
		\(T\)
		&
		\(Q_0\)
		&
		\(P_{2,-3}\)
		&
		\(P_{2,-2}\)
		&
		\(T_{4_1}\)
		\\
		\midrule
		\(\Khr(T)\)
		&
		$
		\begin{aligned}
		\Rational_1(0)&:\delH^{0}
		\end{aligned}
		$
		&
		$
		\begin{aligned}
		\Special_4(\infty)&:\delH^{-1}\\
		\Rational_1(-2)&:\delH^{-1}
		\end{aligned}
		$
		&
		$
		\begin{aligned}
		\Rational_1(2)&:\delH^{-1/2}\\
		\Rational_1(-2)&:\delH^{-1/2}
		\end{aligned}
		$
		&
		$
		\begin{aligned}
		\Special_4(\infty)&:\delH^{3}+\delH^{4}\\
		\Rational_1(4)&:\delH^{3}
		\end{aligned}
		$
		\\
		\midrule
		\(\HFT(T)\)
		&
		$
		\begin{aligned}
		\Rational(0)&:\delH^{0}
		\end{aligned}
		$
		&
		$
		\begin{aligned}
		\Special_4(\infty)&:\delH^{-1}\\
		\Rational(-2)&:\delH^{-1}
		\end{aligned}
		$
		&
		$
		\begin{aligned}
		\Special_4(\infty)&:\delH^{-1/2}\\
		\Rational(2)&:\delH^{-1/2}\\
		\Rational(-2)&:\delH^{-1/2}
		\end{aligned}
		$
		&
		$
		\begin{aligned}
		\Special_4(\infty)&:\delH^{3}+\delH^{4}\\
		\Rational(4)&:\delH^{3}
		\end{aligned}
		$
		\\
		\midrule
		\(\Theta(T)\)
		&
		\(\QPI\smallsetminus\{0\}\)
		&
		\((-2,\infty]\)
		&
		\((-2,2)\)
		&
		\(\{\infty\}\)
		\\
		\bottomrule
	\end{tabular}
	\medskip
	\caption{Some prime Conway tangles \(T\), their invariants \(\Khr(T)\) and \(\HFT(T)\), and their spaces of thin rational fillings. 
	The polynomial expressions in \(\delH\) are the Poincaré polynomials that indicate how often the respective curves appear in which gradings in the invariants. 
	In all examples \(\ThinHF(T)=\ThinKh(T)\).
	For \(\HFT\), an entry \(\Special_4(\infty)\) represents a conjugate pair of special curves \(\Special_1(\infty;\TEI,\TEII)\) and \(\Special_1(\infty;\TEIII,\TEIV)\) in identical \(\delta\)-gradings. 
	The computations for \(\Khr\) were made using the program \cite{khtpp}; for the raw data and the tangle orientations used to fix the absolute \(\delta\)-grading, see \cite{tangle-atlas}. 
	The computations of \(\HFT(Q_0)\) and \(\HFT(P_{2,-3})\) can be found in \cite{pqMod}. \(\HFT(P_{2,-2})\) and \(\HFT(T_{4_1})\) were computed using the Mathematica packages \cite{PQM.m} and \cite{APT.m}, respectively. 
	In all cases, the absolute \(\delta\)-grading on \(\HFT\) was chosen such that it matches the one on \(\Khr\). %
}\label{tab:prime_tangles}
\end{table}

\begin{table}[p]
	\centering
	\begin{tabular}{rcc}
		\toprule
		&
		\includegraphics[height=4.25cm]{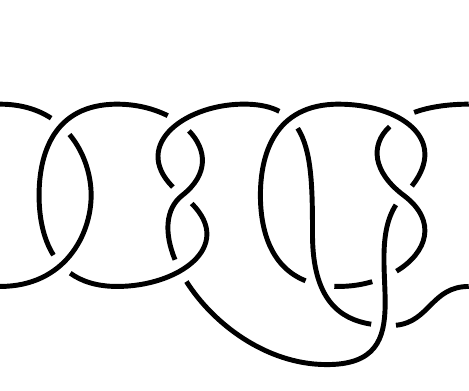}
		&
		\includegraphics[height=4.25cm]{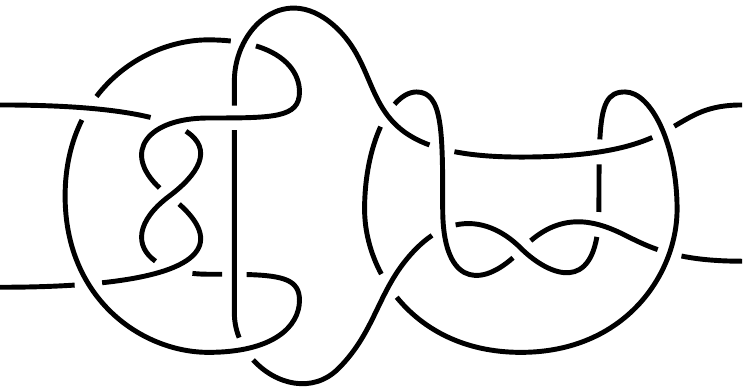}
		\\
		\(T\)
		&
		\(T_a\)
		&
		\(T_b\)
		\\
		\midrule
		\(\Khr(T)\)
		&
		$
		\begin{aligned}
		\Special_4(2)&:2\cdot\delV^{1}\\
		\Special_4(1)&:6\cdot\delV^{1}\\
		\Rational_1(\tfrac{1}{2})&:\delV^{1}\\
		\Special_4(0)&:2\cdot\delV^{1}
		\end{aligned}
		$
		&
		$
		\begin{aligned}
		\Special_4(\infty)&:4\cdot\delH^{11/2}+12\cdot\delH^{9/2}+8\cdot\delH^{7/2}\\
		\Special_4(4)&:\delV^{3}+\delV^{4}\\
		\Rational_1(\tfrac{15}{4})&:\delV^{3}
		\end{aligned}
		$
		\\
		\midrule
		\(\ThinKh(T)\)
		&
		\([2,0]\)
		&
		\(\varnothing\)
		\\
		\bottomrule
	\end{tabular}
	\medskip
	\caption{Two more prime Conway tangles \(\protect T\), their invariant \(\protect \Khr(T)\), and the corresponding space of thin rational fillings.The polynomial expressions in \(\protect\delH\) and \(\protect\delV\) are the Poincar\'e polynomials that indicate how often the respective curves appear in which gradings in \(\protect\Khr(T)\). The computations were made using the program \protect\cite{khtpp}; for the raw data and the tangle orientations used to fix the absolute \(\delta\)-grading, see \protect\cite{tangle-atlas}.}\label{tab:prime_tangles:closed+empty}
\end{table}

This discussion is summarized in Table~\ref{tab:prime_tangles}. In each case, the space of A-link fillings agrees with the space of thin fillings and, perhaps more surprisingly, the spaces agree in both the Heegaard Floer and the Khovanov setting. 

\subsection{Amalgamation: thin knots containing essential Conway spheres.} We can now illustrate what is perhaps the main observation of this paper, that is, the fact that understanding the thin filling slopes for tangles $T_1$ and $T_2$ allows us to determine when the link $T_1\cup T_2$ will be thin. This is shown for our main running example in Figure~\ref{fig:amalg}.

\labellist 
\small 
\pinlabel $\BNr(K)\cong \F[H]\oplus \F^8$ at 275 120
  \pinlabel $\Khr(K)\cong\F^{17}$ at 255 100
		\endlabellist
\begin{figure}[t]
\includegraphics[scale=0.75]{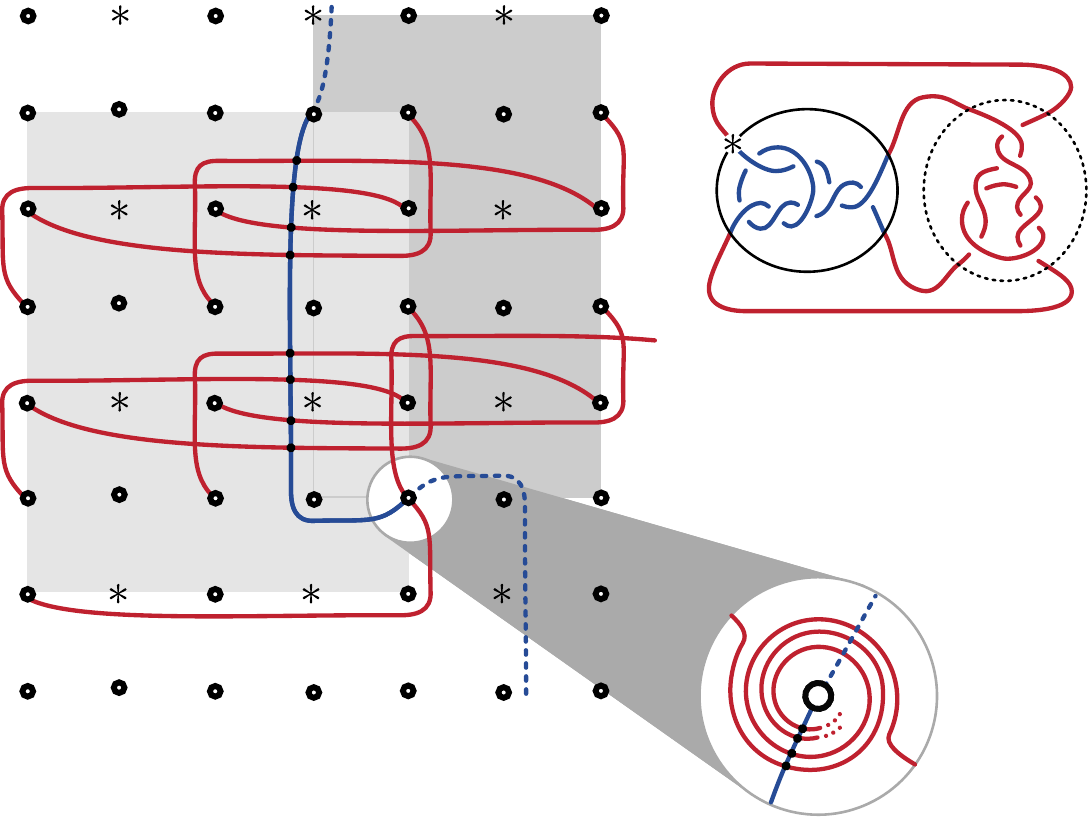}
\caption{A tangle decomposition of a thin knot $K$ along an essential Conway sphere and its reduced Bar-Natan homology computed from the two tangle invariants. Observe that the spaces of thin fillings of the two tangles are $(-\infty,0]$ and $(0,\infty]$, so their union is indeed $\QPI$, in accordance with Theorem~\ref{thm:gluing:Thin:intro}.
We expect that the dimensions of knot Floer homology and reduced Khovanov homology in this example are minimal among all knots containing essential Conway spheres. This will be explored in~\cite{KLMWZ}. 
}
\label{fig:amalg}
\end{figure}

\begin{proposition}
Unions of $(2,-3)$-pretzel tangles give links that are thin in Khovanov homology if and only if they are thin in Heegaard Floer homology. 
\end{proposition}

\begin{proof}Immediate: the set of thin filling slopes agrees in both cases.\end{proof}

While this amounts, essentially, to a single example, we remark that Example~\ref{exa:thinness_depends_on_coefficients:tangles} is the only example we have seen so far in which the spaces of thin fillings do not agree for the two theories, and even in this example, they may actually agree over \(\Q\). 
Note that one can also check that the thin links obtained in this way are a strict subset of the L-spaces one obtains by gluing a pair of trefoil exteriors together.

\subsection{Exotica} The examples collected in Table \ref{tab:prime_tangles} show that the spaces of A-link and thin rational fillings can be open and half-open intervals. But they may also be closed intervals, as the tangle \(T_a\) from Table~\ref{tab:prime_tangles:closed+empty} illustrates.
The tangle \(T_b\) from the same table is obtained by taking a tangle sum of two copies of the tangle \(T_{4_1}\) from Table~\ref{tab:prime_tangles} after rotating one of them by \(\tfrac{\pi}{2}\). This tangle does not admit any A-link filling, since its invariants contain special components in adjacent \(\delta\)-gradings in two distinct slopes. So, the space of A-link fillings of a tangle (as well as, consequently, the space of thing fillings) can be either empty, a singleton, an open interval, a half-open interval, or a close interval. This is in contrast with the space of L-space fillings of a three-manifold with torus boundary, which can only be  empty, a singleton, a closed interval, or $\QPI$ minus a point \cite{Rasmussenx2}. In summary, all types of A-link rational filling spaces from Theorem~\ref{thm:charactisation:ALink:intro} arise in actual examples. We do not know if the same is true for the additional case of precisely two distinct thin rational fillings in Theorem~\ref{thm:charactisation:Thin:intro}. 

\begin{conjecture}
	There is no tangle \(T\) such that \(\ThinHF(T)\) or \(\ThinKh(T)\) consists of two points.
\end{conjecture}

\begin{figure}[b]
	\includegraphics[scale=0.75]{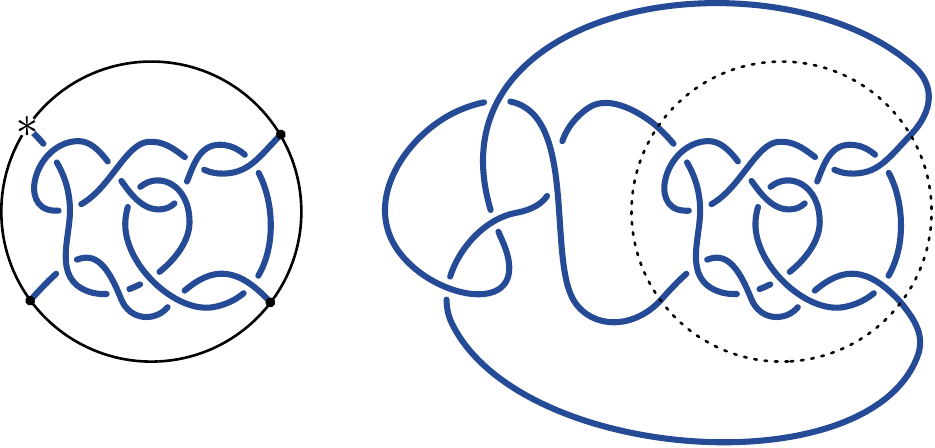}
	\caption{The tangle from Example~\ref{exa:thinness_depends_on_coefficients:tangles} and its \(\nicefrac{5}{3}\)-rational filling, which is the knot from Example~\ref{exa:thinness_depends_on_coefficients:knots}.}
	\label{fig:16n197566}
\end{figure}

\begin{example}\label{exa:thinness_depends_on_coefficients:tangles} 
	From the knot \(16^n_{197566}\) in Example~\ref{exa:thinness_depends_on_coefficients:knots}, one can easily obtain tangles for which the spaces of thin fillings depend on the field of coefficients. 
	For instance, Figure~\ref{fig:16n197566} shows a tangle whose rational filling along the slope \(\nicefrac{5}{3}\) is equal to \(16^n_{197566}\). 
	The Khovanov invariant \(\Khr\) of this tangle can be summarized as follows, using the same notation as in Tables~\ref{tab:prime_tangles} and~\ref{tab:prime_tangles:closed+empty}; see example \texttt{T\_16n197566} in \cite{tangle-atlas}:
	\begin{align*}
	\text{over }\F_2:
	&&
	\Special_4(\infty)
	&:16\cdot\delH^{2}
	&
	\Special_8(\infty)
	&:\delH^{2}
	&
	\Rational_1(\nicefrac{4}{3})
	&:\delV^{3/2}
	\\
	\text{over }\F_3:
	&&
	\Special_4(\infty)
	&:16\cdot\delH^{2}
	&
	\Special_6(\infty)
	&:\delH^{2}
	&
	\Rational_1(2)
	&:\delV^{3/2}
	&
	\Rational_2(2)
	&:\delV^{3/2}
	&
	\end{align*} 
	Thus, the space of thin fillings of this tangle is equal to \([\infty,\nicefrac{4}{3})\) for Khovanov homology over \(\F_2\), but \([\infty,2)\) for Khovanov homology over \(\F_3\).
	As a result, Shumakovitch's example is part of an infinite family of links: Pick any closure of this tangle along a slope \(s\in[\nicefrac{4}{3},2)\). 
	
	The Heegaard Floer invariant of this tangle is equal to 
	\(
	\{2\cdot\Rational(2),\Rational(4),17\cdot\Special_4(\infty)\}
	\), where \(2\cdot \Special_4(\infty)\) represents two conjugate pairs of special curves \(\Special_1(\infty;\TEI,\TEII)\) and \(\Special_1(\infty;\TEIII,\TEIV)\). All special curves and the two rational curves all live in the same \(\delta\)-grading. Moreover, \(\delta(17\cdot\Special_4(\infty),\Rational(4))=\delta(\Rational(4),2\cdot\Rational(2))=0\). 
	(These computations were done indirectly using \cite{HFKcalc}.)
	Thus the space of thin fillings in Heegaard Floer theory is equal to \([\infty,2)\).
\end{example}

\begin{figure}[t]
	\includegraphics[scale=0.75]{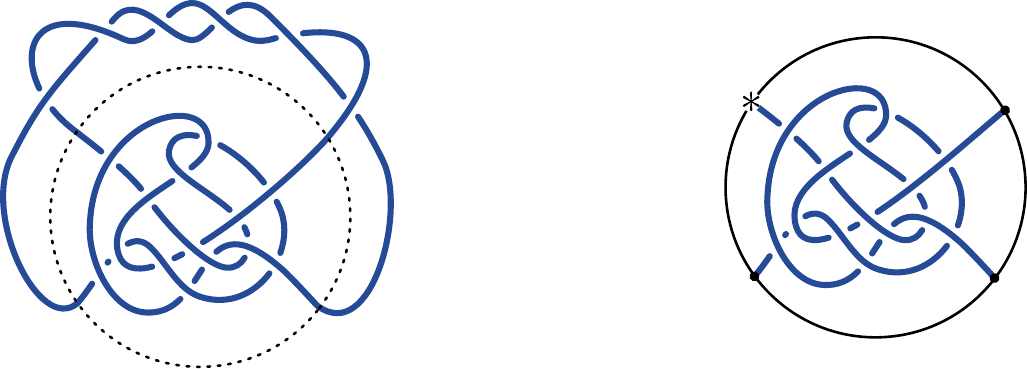}
	\caption{The knot Floer homology of the 6-twisted Whitehead double of the right-handed trefoil knot (shown on the left) does not have full support, and the subtangle (shown on the right) is Heegaard Floer exceptional.}
	\label{fig:whitehead}
\end{figure}

We are grateful to Matt Hedden for pointing out that knot Floer homology does not have full support; see Figure~\ref{fig:whitehead}. 
This example was used by Hedden and Ording to establish that the Ozsváth-Szabó concordance invariant $\tau$ and the Rasmussen invariant $s$ do not agree \cite{HO2008}.
It is interesting to study this example in the context of tangles:

\begin{example}\label{exp:whitehead}
	Let \(T\) be the subtangle of the thick A-knot illustrated in Figure~\ref{fig:whitehead}.
	We computed the Heegaard Floer invariants of \(T\) as
	\[
	\HFT(T)=\{\Rational(-\nicefrac{5}{2}),2\cdot\Special_4(\infty)\}
	\qquad
	\ALinkHF(T)=(-\nicefrac{5}{2},\infty]
	\qquad
	\ThinHF(T)=\{\infty\}
	\]
	Here, \(2\cdot \Special_4(\infty)\) represents two conjugate pairs of special curves \(\Special_1(\infty;\TEI,\TEII)\) and \(\Special_1(\infty;\TEIII,\TEIV)\), which all live in the same \(\delta\)-grading. Moreover, \(\delta(\Rational(-\nicefrac{5}{2}),2\cdot\Special_4(\infty))=2\). 
	These computations were done indirectly using \cite{HFKcalc}. 
	Note in particular that \(T\) is an example of a Heegaard Floer exceptional tangle (Definition~\ref{def:HF:exceptional}). Together with Example~\ref{exa:ExceptionalExample}, it illustrates that Theorem~\ref{thm:glueing:Thin:HF} is indeed wrong if we drop the assumption that  both tangles not be Heegaard Floer exceptional. Thus, it validates the shift in perspective from thin to A-links. 	
	
	In contrast, the Khovanov homology of the knot in Figure~\ref{fig:whitehead} is not thin and it has full support. Moreover, we compute 
	\begin{align*}
	\Khr(T;\F_2)
	&=\{\Rational_1(\nicefrac{3}{2}),8\cdot\Special_4(\infty)\}
	&
	\ALinkHF(T;\F_2)=\ThinHF(T;\F_2)&=\{\infty\}
	\\
	\Khr(T;\field)
	&=\{\Rational_1(-\nicefrac{1}{2}),2\cdot\Special_4(\infty),6\cdot\Special_2(\infty)\}
	&
	\ALinkHF(T;\field)=\ThinHF(T;\field)&=\{\infty\}
	\end{align*}
	where \(\field=\Q, \F_3,\F_5, \F_7, \F_{11}\); see example \texttt{3\_1} in \cite{tangle-atlas}.
	Here, the special components sit in three consecutive \(\delta\)-gradings. 
	
	It is interesting that we see three distinct slopes of rational components in this example. This is related to the fact that rational closures of \(T\) not only provide examples for which \(\tau\) and Rasmussen's original invariant \(s=s^\Q\) are distinct, but they also include examples for which \(s^\Q\neq s^{\F_2}\). This phenomenon will be explored in upcoming work of Lukas Lewark and the third author; see also~\cite{MRP2}. 
\end{example}

\begin{remark}\label{rem:boundary_compressible}
	As an observation to summarize: In Theorem~\ref{thm:L_space_gluing}, one might ask why the assumption that the two manifolds are boundary incompressible is needed. The answer is that when one of the $M_i$ is boundary compressible then the condition for L-spaces is $\mathcal{L}(M_0) \cup h(\mathcal{L}(M_1)) = \QPI$; consider for instance Dehn surgery along the figure-eight knot; for more discussion, see \cite{HRW} and compare \cite{L-space_graph_mnflds}.  In contrast, because the intervals have a wider range of endpoint behaviour in the tangle case, our A-link and Thin Gluing Theorems do not admit cleaner statements if we assume that the tangles are boundary incompressible (ie non-split). 
\end{remark}


\begin{small}
	\pdfbookmark[section]{Acknowledgements}{Acknowledgements}
	\noindent\textbf{Acknowledgements.}
	The authors thank Antonio Alfieri, Jonathan Hanselman, Matt Hedden, Lukas Lewark, Tye Lidman, Allison Moore, and Jake Rasmussen for helpful conversations. 
\end{small}

\nocite{khtpp}
\newcommand*{\arxiv}[1]{\href{http://arxiv.org/abs/#1}{ArXiv:\ #1}}
\newcommand*{\arxivPreprint}[1]{\href{http://arxiv.org/abs/#1}{ArXiv preprint #1}}
\bibliographystyle{alpha}
\bibliography{ThinFillings}
\end{document}